\documentclass{siamltex}

\usepackage{amsfonts}
\usepackage{float}
\usepackage{graphicx}
\usepackage{epstopdf}
\usepackage{amstext}
\usepackage{amsmath}
\usepackage{amssymb}

\newtheorem{assumption}[theorem]{Assumption}

\newcommand{\eps}{\varepsilon}

\newcommand{\PP}{\mathbb{P}}
\newcommand{\QQ}{\mathbb{Q}}
\newcommand{\EE}{\mathbb{E}}
\newcommand{\RR}{\mathbb{R}}

\title{Anomalous Shock Displacement Probabilities for a Perturbed Scalar Conservation Law}

\author{
Josselin Garnier\thanks{Laboratoire de Probabilit\'{e}s et Mod\`{e}les Al\'{e}atoires 
\& Laboratoire Jacques-Louis Lions, Universit\'{e} Paris VII 
(\tt{garnier@math.univ-paris-diderot.fr})} 
\and 
George Papanicolaou\thanks{Mathematics Department, Stanford University 
(\tt{papanicolaou@stanford.edu})}
\and 
Tzu-Wei Yang\thanks{Institute for Computational and Mathematical Engineering (ICME), Stanford University 
(\tt{twyang@stanford.edu})}}

\begin{document}

\maketitle

\begin{abstract}
We consider an one-dimensional conservation law with random space-time forcing and calculate using large 
deviations the exponentially small probabilities of anomalous shock profile displacements. Under suitable 
hypotheses on the spatial support and structure of random forces, we analyze the scaling behavior of the 
rate function, which is the exponential decay rate of the displacement probabilities. For small 
displacements we show that the rate function is bounded above and below by the square of the displacement 
divided by time. For large displacements the corresponding bounds for the rate function are proportional to 
the displacement. We calculate numerically the rate function under different conditions and show that the 
theoretical analysis of scaling behavior is confirmed. We also apply a large-deviation-based importance 
sampling Monte Carlo strategy to estimate the displacement probabilities. We use a biased distribution 
centered on the forcing that gives the most probable transition path for the anomalous shock profile, which 
is the minimizer of the rate function. The numerical simulations indicate that this strategy is much more 
effective and robust than basic Monte Carlo.
\end{abstract}



\pagestyle{myheadings}

\thispagestyle{plain}

\section{Introduction}

It is well known that nonlinear waves are not very sensitive to perturbations in initial conditions or 
ambient medium inhomogeneities. This is in contrast to linear waves in random media where even weak
inhomogeneities can affect significantly wave propagation over long times and distances. It is natural, 
therefore to consider perturbations of shock profiles of randomly forced conservation laws as rare events
and use large deviations theory. The purpose of this paper is to calculate probabilities of anomalous shock 
profile displacements for randomly perturbed one-dimensional conservation laws. We analyze the rate function 
that characterizes the exponential decay of displacement probabilities and show that under suitable 
hypotheses on the random forcing they have scaling behavior relative to the size of the displacement and the 
time interval on which it occurs.

The theory of large deviations for conservation laws with random forcing is an extension of the 
Freidlin-Wentzell theory of large deviations \cite{Freidlin1998,Dembo2010} to partial differential 
equations. This has been carried out extensively \cite{Cardon-Weber1999,Cardon-Weber2001,Mariani2010} and we 
use this theory here. We are interested in a more detailed analysis of the exponential probabilities of 
anomalous shock profile displacements, which leads to an analysis of the rate function associated with large 
deviations for this particular class of rare events. We derive upper and lower bounds for the exponential 
decay rate of the small probabilities using suitable test functions for the variational problem involving 
the rate function. This is the main result of this paper. The applications we have in mind come from 
uncertainty quantification \cite{Iaccarino2011} in connection with simplified models of flow and combustion 
in a scramjet. Numerical calculations of the exponential decay rates from variational principles associated 
with the rate function have been carried out before \cite{E2004,Zhou2008}. We carry out such numerical 
calculations here and confirm the scaling behavior of bounds obtained theoretically. We use a gradient 
descent method to do the optimization numerically and we note that its convergence is quite robust even 
though the functional under consideration is not known to be convex. This robustness suggests the 
Monte Carlo simulations with importance sampling using a change of measure based on the minimizer of the 
discrete rate function is likely to effective. Our simulations show that indeed such importance sampling 
Monte Carlo performs much better than the basic Monte Carlo method.

The paper is organized as follows. In Section \ref{sec:formulation} we formulate the one-dimensional 
conservation law problem. In Section \ref{sec:LDP} we state the large deviation principle and identify the 
rate function which we will use. In Section \ref{sec:rare event1} we give a simple, explicitly computable 
case of shock profile displacement probabilities that can be used to compare with the results of
the large deviation theory. Section \ref{sec:displacement_general} contains the main results of the paper,
which as the upper and lower bounds for the exponential decay rate of the displacement probabilities, under 
different conditions on the random forcing. We identify scaling behavior of these probabilities, relative to 
size of the displacement and the time interval of interest. In Section \ref{sec:discrete LDP} we introduce a 
discrete form for the conservation law and the associated large deviations, and calculate numerically the 
displacement probabilities from the discrete variational principle. In Section \ref{sec:importance sampling} 
we implement importance sampling Monte Carlo based on the minimizer of the discrete rate function and we 
compare it with the basic Monte Carlo method. We end with a brief section summarizing the paper and our 
conclusions.
\section{The perturbed conservation law}
\label{sec:formulation}

We consider the scalar viscous conservation law:
\begin{eqnarray}
	&&u_t + ( F(u) )_x = (Du_x)_x, \quad \quad x\in \RR, \quad  t \in [0,\infty) , \\
	&&u(t=0,x)=u_0(x), \quad \quad x\in \RR .
\end{eqnarray}
Here $F \in {\cal C}^2(\RR)$ and the initial condition satisfies $u_0(x) \to {u}_\pm$ as $x\to \pm \infty$
where ${u}_- < {u}_+$. We are interested in traveling wave solutions of the form ${U}(x-\gamma t)$, where 
the ${\cal C}^2$ profile ${U}$ satisfies ${U}(x) \to {u}_\pm$ as $x\to \pm \infty$ and the wave speed 
$\gamma$ is given by the Rankine-Hugoniot condition
\begin{equation}
	\label{eq:Rankine-Hugoniot}
	\gamma = \frac{F({u}_+)-F({u}_-)}{{u}_+-{u}_-} .
\end{equation}
A traveling wave ${U}$ with wave speed $\gamma$ exists provided the following conditions are fulfilled:
\begin{eqnarray}
	\label{hyp1a}
	&&F( u) - F({u}_-) > \gamma (u-{u}_-) , \quad \forall u \in (u_+,u_-),\\
	\label{hyp1b}
	&&F'({u}_+) < \gamma < F'({u}_-) .
\end{eqnarray}
The first condition is the Oleinik entropy condition \cite{Ilin1960} and the second one is the Lax entropy 
condition \cite{Lax1957}. Under these conditions the traveling wave profile exists, it is the solution of 
the ordinary differential equation
$$
	U_x =\frac{1}{D} \big( F(U)- F(u_-) - \gamma(U-u_-) \big), \quad \quad 
	U(x) \stackrel{x \to -\infty}{\longrightarrow} u_- ,
$$
and it is orbitally stable, which means that perturbations of the profile decay in time, and thus initial 
conditions near the traveling wave profile converge to it. Note that a physically admissible viscous profile 
must have the stability property; otherwise, it would not be observable. As noted in \cite{Jones1993}, the 
motivating idea behind the orbital stability result is that in the stabilizing process, information is 
transferred from spatial decay of the profile ${U}$ at infinity to temporal decay of the perturbation.

The purpose of this paper is to address another type of stability, that is, the stability with 
respect to external noise. We consider the perturbed scalar viscous conservation law with additive noise:
\begin{eqnarray}
	\label{eq:pert0}
	&&u^\eps_t + ( F(u^\eps) )_x 
	= (Du^\eps_x)_x + \eps \dot{W}(t,x), \quad \quad x\in \RR, \quad  t \in [0,\infty) , \\
	&&u^\eps(t=0,x)={U}(x), \quad \quad x\in \RR .
\end{eqnarray}
Here $\eps$ is a small parameter, $W(t,x)$ is a zero-mean random process (described below), and the dot 
stands for the time derivative. We would like to address the stability of the traveling wave 
${U}(x-\gamma t)$ driven by the noise $\eps \dot{W}$. Motivated by an application modeling combustion in a 
scramjet \cite{Iaccarino2011,West2011}, we have in mind a specific rare event, which is an exceptional or 
anomalous shift of the position of the traveling wave compared to the unperturbed motion with the constant 
velocity $\gamma$.

We consider mild solutions which satisfy (denoting $u^\eps(t)=(u^\eps(t,x))_{x\in \RR}$)
\begin{equation}
	\label{mild:ueps}
	u^\eps(t) = S(t) U - \int_0^t S(t-s) N(u^\eps(s)) ds +  \eps \int_0^t S(t-s) {dW}(s)   ,
\end{equation}
where $S(t)$ is the heat semi-group with kernel
$$
	S(t,x,y) = \frac{1}{\sqrt{2 \pi D t}} \exp \Big( - \frac{(x-y)^2}{2Dt} \Big) ,
$$
and $N(u)(x)= ( F( u(x) ) )_x$. 
The main result about the heat kernel 
\cite[Chap. XVI, Sec. 3]{Dautray1992} is as follows.

\begin{lemma}
	\label{lem:dautray}%
	If $f = (f(t))_{t \in [0,T]} \in L^2([0,T],H^{-1}(\RR))$, then the function $ \int_0^t S(t-s) f(s) ds$ 
	is in $L^2([0,T],H^1(\RR)) \cap {\cal C}( [0,T], L^2(\RR))$.
\end{lemma}

A white noise or cylindrical Wiener process $B(t,x)$ in the Hilbert space $L^2(\RR)$ is such that for any 
complete orthonormal system $(f_n(x))_{n \geq 1}$ of $L^2(\RR)$, there exists a sequence of independent 
Brownian motions $(\beta_n(t))_{n \geq 1}$ such that 
\begin{equation}
	\label{eq:whitenoise}
	B(t,x)=\sum_{n=1}^\infty \beta_n(t) f_n(x).
\end{equation}
$B(t,x)$ can be seen as the (formal) spatial derivative of the Brownian sheet on $[0,\infty)\times \RR$,
which means  that it is the Gaussian process with mean zero and covariance 
$\EE [ B(t,x) B(s,y)]= s\wedge t\, \delta(x-y)$. Note that the sum (\ref{eq:whitenoise}) does not converge 
in $L^2$ but in any Hilbert space $H$ such that the embedding from $L^2(\RR)$ to $H$ is Hilbert-Schmidt. 
Therefore, the image of the process $B$ by a linear mapping on $L^2(\RR)$ is a well-defined process in the 
Sobolev space $H^k(\RR)$ (for $k=0$, we take the convention $H^0=L^2$) when the mapping is Hilbert-Schmidt 
from $L^2(\RR)$ into $H^k(\RR)$. Then, if the kernel $\Phi$ is Hilbert-Schmidt in the sense that
$$
	\sum_{n=1}^\infty
	\| \Phi f_n \|_{H^k(\RR)}^2 
	= \sum_{j=0}^k \| \partial_{x}^j \Phi(\cdot,\cdot) \|_{L^2(\RR \times \RR)}^2
	< \infty  ,
$$
then the random process $W=\Phi B$,
$$
	W (t,x) = \int \Phi(x,x') B(t,x') dx'
$$
is well-defined in $H^k(\RR)$. It is a  zero-mean  Gaussian  process with covariance function given by
\begin{eqnarray*}
	&&\EE[ W(t,x) W(t',x') ] = t\wedge t' \, C(x,x'),\\
	&&C(x,x')=\Phi \Phi^T(x,x') =
	\int \Phi(x,x'')\Phi(x',x'')   dx'' ,
\end{eqnarray*}
where $\Phi^T$ stands for the adjoint of $\Phi$.

As an example, we may think at $\Phi(x,x')=\Phi_0(x) \Phi_1(x-x')$ with $\Phi_0 \in H^k(\RR)$ and 
$\Phi_1 \in H^k(\RR)$. In this case $\Phi_0$ characterizes the spatial support of the additive noise and
$\Phi_1$ characterizes its local correlation function. In the limit $\Phi_0(x)=1$ and 
$\Phi_1(x-x')=\delta(x-x')$, the process $\dot{W}(t,x)$ in (\ref{eq:pert0}) is a space-time white noise.

By adapting the technique used in \cite[Chap. XVI, Sec. 3]{Dautray1992} we obtain the following lemma.

\begin{lemma}
	\label{lem:HS}
	If $\Phi$ is Hilbert-Schmidt from $L^2$ into $L^2$, then $Z(t) = \int_0^t S(t-s) {dW}(s)$ belongs to 
	$L^2([0,T], H^1(\RR))\cap {\cal C}([0,T], L^2(\RR))$ almost surely. If  $\Phi$ is Hilbert-Schmidt from 
	$L^2$ into $H^1$, then $Z(t)$ belongs to $L^2([0,T], H^2(\RR))\cap {\cal C}([0,T], H^1(\RR))$ almost 
	surely.
\end{lemma}

\begin{proof} 
	See Appendix \ref{app:prooflemHS}.
\end{proof}
\section{Large deviation principle}
\label{sec:LDP}

In this section we state a large deviation principle (LDP) for the solution 
$(u^\eps(t,x))_{ t \in [0,T], x\in \RR}$ of the randomly perturbed scalar conservation law (\ref{eq:pert0}). 
It generalizes the classical Freidlin-Wentzell principle for finite-dimensional diffusions. Throughout this 
paper, we assume that the flux $F$ is a $\mathcal{C}^2$ function with bounded first and second derivatives.

\begin{assumption}
	\label{asmp:saturated fluxes}
	In this paper, the flux $F$ is a $\mathcal{C}^2$ function and there exists $C_F<\infty$ such that 
	$\|F'\|_{L^\infty(\RR)}$ and $\|F''\|_{L^\infty(\RR)}$ are bounded by $C_F$.
\end{assumption}

\textbf{Remark.} Assumption \ref{asmp:saturated fluxes} is merely a technical assumption to simplify the 
proofs of Proposition \ref{prop:space of u^eps} and \ref{prop:ldp} and obviously it violates the convexity 
of fluxes in conservation laws. From the physical point of view, for a general flux $F$, we can choose a 
very large constant $M$ and let $F_M(u)=F(u)$ for $|u|\leq M$ and saturate $F_M(u)$ for $|u|>M$. If $[-M,M]$ 
can cover the range of interest of $u$, then $F_M(u)=F(u)$. Without a proof, we point out that this physical 
argument can be proven mathematically: if $u^\eps_t$, $u^\eps_x$ and $u^\eps_{xx}$ in (\ref{eq:pert0}) are 
continuous on $[0,T]\times\RR$ and $\dot{W}$ is bounded on $[0,T]\times\RR$, then the parabolic maximum 
principle implies that $|u^\eps|$ is bounded on $[0,T]\times\RR$ and thus we can find $M<\infty$. The 
technical difficulty is that $W$ is not differentiable in time so we can not apply the parabolic maximum 
principle directly. However, we can consider a modified $\tilde{u}^\eps$ by replacing $W$ by a smooth 
$\tilde{W}$ in (\ref{eq:pert0}) and show that 
$\|\tilde{u}^\eps(t)\|_{L^\infty(\RR)} \to \|u^\eps(t)\|_{L^\infty(\RR)}$ as $\tilde{W}\to W$.

We first describe the functional space to which the solution of the perturbed conservation law belongs.
\begin{proposition}
	\label{prop:space of u^eps}
	If $\Phi$ is Hilbert-Schmidt from $L^2$ to $H^1$, then for any $\eps >0$ there is a unique solution 
	$u^\eps$ to (\ref{mild:ueps}) in  the space ${\cal E}^1$ almost surely, where
	\begin{equation*}
		{\cal E}^1 = \big\{(u(t,x))_{t \in [0,T], x\in \RR}:~ 
		(u(t,x)- U(x))_{t \in [0,T],x\in \RR} \in {\cal C}( [0,T], H^1(\RR) ) \big\}  .
	\end{equation*}
\end{proposition}

\begin{proof}
	See Appendix \ref{pf:space of u^eps}
\end{proof}

The rate function of the LDP  is defined in terms of the mild solution of the control problem
\begin{eqnarray}
	\label{eq:control1a}
	&& u_t + ( F(u) )_x = (Du_x)_x + \Phi h(t,x), \quad \quad x\in \RR, \quad  t \in [0,T] , \\
	&& u(t=0,x)={U}(x) , \quad \quad x\in \RR,
\end{eqnarray}
where $h \in L^2( [0,T], L^2(\RR))$. The solution $u$ to this problem is denoted by ${\cal H}[h]$ and 
${\cal H}$ is a mapping from $ L^2( [0,T], L^2(\RR))$ to ${\cal E}^1$ provided $\Phi$ is Hilbert-Schmidt 
from $L^2$ to $H^1$. The following proposition is an extension of the LDP proved in \cite{Cardon-Weber1999}.

\begin{proposition}
	\label{prop:ldp}
	If $\Phi$ is Hilbert-Schmidt from $L^2$ to $H^1$, then the solutions $u^\eps$ satisfy a large deviation 
	principle in ${\cal E}^1$ with the good rate function
	\begin{equation}
		\label{eq:rate function}
		I(u) = \inf_{ h , u={\cal H}[h]} \frac{1}{2}\int_0^T  \| h(t,\cdot)\|_{L^2}^2 dt,
	\end{equation}
	with the convention $\inf \emptyset = \infty$.
\end{proposition}
\begin{proof} 
	See Appendix \ref{pf:ldp}.
\end{proof}

In fact, $I(u)< \infty$ if and only if $u$ is in the range of ${\cal H}$. The LDP means that, for any 
$A \subset {\cal E}^1$, we have
\begin{equation}
	-{\cal J}(\mathring{A} ) \leq
	\liminf_{\eps \to 0} \eps^2 \log \PP( u^\eps \in A) 
	\leq \limsup_{\eps \to 0} \eps^2 \log \PP( u^\eps \in A) \leq - {\cal J}(\overline{A}) ,
\end{equation}
with
\begin{equation}
	{\cal J}(A) =\inf_{u \in A} I(u) .
\end{equation}
Note that the interior $\mathring{A}$ and closure $\overline{A}$ are taken in the topology associated with 
${\cal E}^1$. 

Although the convexity of the rate function $I$ is unknown, it is possible to show that $I$ satisfies the 
maximum principle: if the set of the rare event $A$ does not contain the exact solution $U(x-\gamma t)$, 
then $\inf_{u\in A}I(u)$ attains its minimum at the boundary of $A$.
\begin{proposition}
	\label{prop:maximum principle}
	If $U(x-\gamma t)\notin A$ and $u\in\mathring{A}$, then there exists a 
	sequence $\{u^n\}$ in $\mathcal{E}^1$ such that $I(u^n)<I(u)$ and $u^n\to u$ in 
	$\mathcal{E}^1$ as $n\to\infty$. As a consequence, any $u\in\mathring{A}$ can not be a 
	local minimizer of $I$.
\end{proposition}
\begin{proof}
	See Appendix \ref{pf:maximum principle}.
\end{proof}
\section{Wave displacement from an elementary point of view}
\label{sec:rare event1}

In this paper we are interested in estimating the probability of large deviations from the deterministic 
path ${U}(x-\gamma t)$. In this section, we first study in a very elementary way how the center of the solution $u^\eps$ at time $T$ 
can deviate from its unperturbed value. The center of a function 
$u \in {\mathcal E}^1$ is defined as
\begin{equation}
	{\cal C}[u](t) = - \frac{ \int_{-\infty}^\infty [u(t,x) - U(x)] dx }{u_- - u_+} ,
\end{equation}
provided that the integral is well-defined.

\begin{proposition}
	\label{prop:1}%
	If the covariance function $C$ is in $L^1(\RR\times \RR)$, then the center of $u^\eps$ is  well-defined 
	for any time $t \in [0,T]$ almost surely and it is given by
	\begin{equation}
		\label{center1}
		{\cal C}[u^\eps](t) = \gamma t - \eps \frac{ \int W (t,x) dx }{u_- - u_+} .
	\end{equation}
	It is a Gaussian process with mean $\gamma t$ and covariance
	\begin{equation}
		\mathbf{Cov} \big( {\cal C}[u^\eps](t) ,{\cal C}[u^\eps](t') \big) 
		= \eps^2  \frac{\iint C(x,x') dx dx'}{(u_- - u_+)^2} \, t \wedge t'  .
	\end{equation}
\end{proposition}

\begin{proof}
	See Appendix \ref{app:proofprop1}.
\end{proof}
 
In the absence of noise, the center of the solution increases linearly as $\gamma t$. In the presence of 
noise we can characterize the probability of the rare event
\begin{equation}
	B = \big\{ u \in {\cal E}^1, \, {\cal C}[u](T) \geq \gamma T + x_0 \big\}  ,
\end{equation}
where $x_0  \in [0,\infty)$.

\begin{proposition}
	If  the covariance function $C$ is in $L^1(\RR\times \RR)$, then 
	\begin{equation}
		\label{eq:Prare1}
		\PP( u^\eps \in B) \sim \exp \Big(-\frac{x_0^2 (u_- - u_+)^2}{2 \eps^2T \iint C(x,x') dx dx'}\Big).
	\end{equation}
\end{proposition}
The approximate equality means that 
$$
\lim_{\eps \to 0} \eps^2 \log \PP( u^\eps \in B) = -\frac{x_0^2 (u_- - u_+)^2}{2  T \iint C(x,x') dx dx'} .
$$
This proposition is a direct corollary of Proposition \ref{prop:1}.

If $x_0=0$ then $U(x-\gamma t) \in B$ and $\PP(u^\eps \in B)=1/2$. If $x_0>0$ then the set $B$ is indeed 
exceptional in that it corresponds to the event in which the center of the profile is anomalously ahead of 
its expected position. Note that the scaling $x_0^2/T$ in (\ref{eq:Prare1})  corresponds to that of the 
exit problem of a Brownian particle.

The LDP for $u^\eps$ stated in Proposition \ref{prop:ldp} is not used here for two reasons:
\begin{enumerate}
	\item It does not give a good result because the interior of $B$ in ${\mathcal E}^1$ is empty (since we 
	can construct a sequence of functions bounded in $H^1$ that blows up in $L^1$).
	
	\item The distribution of the center $ {\cal C}[u^\eps](T) $ is here explicitly known for any $\eps$. 
	This is fortunate and it is not always true. When the rare event is more complex, than the LDP 
	for $u^\eps$ is useful as we will show in the next section.
\end{enumerate}
\section{Wave displacement from the large deviations point of view}
\label{sec:displacement_general}

\subsection{The framework and result}

In this section, we use the large deviation principle to compute the probability of the rare event that the 
perturbed traveling wave is the same profile but with the displacement $x_0$ at time $T$:
\begin{equation}
	\label{eq:rare event of the wave dispacement}
	A = \{u\in\mathcal{E}^1\mbox{ such that } u(0,x)=U(x), \, u(T,x)=U(x-\gamma T - x_0)\}.
\end{equation}

The value of the rate function heavily depends on the kernel $\Phi$ and formally the simplest kernel we can 
have is the identity operator. Of course, the identity operator is not Hilbert-Schmidt, but with the given 
$A$, we can construct a Hilbert-Schmidt $\Phi$ such that $\Phi$ is approximately the identity in the region 
of interest.

\begin{assumption}
	\label{asmp:conditions of Phi}
	\begin{enumerate}
		\item We assume that the center of transition of $U(x)$ in 
		(\ref{eq:rare event of the wave dispacement}) is $0$ and $x_0 + \gamma T \geq 0$.
		
		\item The kernel $\Phi^{l_c}_{L_0}$ has the following form:
		\begin{equation}
			\label{eq:covariance kernel Phi}
			\Phi^{l_c}_{L_0}(x,x') = \sigma \phi_0\Big(\frac{x}{L_0}\Big) 
			\frac{1}{l_c}\phi_1\Big(\frac{x-x'}{l_c}\Big),
		\end{equation}
		where $\sigma$, $L_0$ and $l_c$ are positive constants.
		
		\item $\phi_0$ and $\phi_1$ are $L^1\cap H^1$ functions so that their Fourier transforms are 
		well-defined in $\RR$ and $\Phi$ is Hilbert-Schmidt from $L^2$ to $H^1$ (see Lemma 
		\ref{lma:sifficient condition for Hilbert-Schmidt Phi}). $\phi_0$ and $\phi_1$ are normalized so that
		$\phi_0(0)=\hat{\phi}_1(0)=1$.
		
		\item $\phi_0\in\mathcal{C}^\infty$ is positive valued and $\hat{\phi}_1$ is nonzero. In addition, 
		$1/\phi_0(x)$ and $1/\hat{\phi}_1(\xi)$ have at most polynomial growth at $\pm \infty$.
		
		\item $\phi_0(x)$ is increasing on $x\in(-\infty,-1)$, decreasing on $x\in(1,\infty)$ and 
		identically equal to $1$ on $x\in[-1,1]$. By this setting, the support of the noise kernel $\Phi$ is 
		roughly $[-L_0,L_0]$.
	\end{enumerate}
\end{assumption}
 
The following theorem is the main result of this section, and the proof will be given in 
the next two subsections.
\begin{theorem}
	\label{thm:main}
	Let $A$ be defined by (\ref{eq:rare event of the wave dispacement}). There exists a constant $C_0$ such 
	that for all $\Phi^{l_c}_{L_0}$ satisfying Assumption \ref{asmp:conditions of Phi} with 
	$L_0=\gamma T+x_0+C_0$, the quantity $\mathcal{J}(A)$ of the large deviation problem 
	(\ref{eq:control1a}) generated by $\Phi^{l_c}_{L_0}$ has the following asymptotic (in the sense that 
	$l_c\to 0$) scales:
	\begin{equation}
		\label{eq:upper and lower bounds for J(A), general gamma}
		\mathcal{J}(A) = \inf_{u\in A}I(u) \overset{l_c\to 0}{=}
		\begin{cases}
			\Theta(x_0^2),\quad   &\text{for $|x_0|$ small and any fixed $T$,}\\
			\Theta(x_0^2/T),\quad &\text{for $|x_0|$ small and $T$ small,}\\
			\Theta(|x_0|),\quad   &\text{for $|x_0|$ large and any fixed $T$,}\\
			\Theta(|x_0|/T),\quad &\text{for $|x_0|$ large and $T$ small,}\\
			\Theta(1/T),\quad     &\text{for $T$ small and any fixed $x_0$.}
		\end{cases}
	\end{equation}
	Here $\mathcal{J}(A)\overset{l_c\to 0}{=}\Theta(x_0^2)$ and similar expressions mean that there are constants $C_1$ and 
	$C_2$ with the units such that $C_1 x_0^2$ and $C_2 x_0^2$ are dimensionless and 
	$C_1 x_0^2\leq\mathcal{J}(A)\leq C_2 x_0^2$ as $l_c\to 0$.
\end{theorem}
\begin{proof}
	See Section \ref{sec:bounds for small |x_0|} and \ref{sec:bounds for large |x_0|}.
\end{proof}

\textbf{Remark.} $C_0=\max\{C_v,C_w\}$, where $C_v$ and $C_w$ will be determined in Lemma 
\ref{lma:approximation of phi_0 for small |x_0|} and \ref{lma:approximation of phi_0 for large |x_0|}, 
respectively. $L_0=\gamma T+x_0+C_0$ means that the range of noise covers the region of interest of $A$ so 
that the rare event $A$ is possible and at the same time the range is not too large to cause the waste of 
energy; the small $l_c$ means that the noise in this region is weakly correlated and can be treated as the 
white noise.

We note that there are two occurrences of $T$ in $\mathcal{J}(A)$: in 
(\ref{eq:rate function}), we integrate the $\|h(t,\cdot)\|_{L^2}^2$ from $0$ to $T$, and 
the rare event $A$ is $T$-dependent. These two occurrences of $T$ would make the 
$T$-dependence of $\mathcal{J}(A)$ nontrivial; if $T$ is small, the $T$-dependence in 
(\ref{eq:rare event of the wave dispacement}) is negligible and we can have the scale in 
$T$. Therefore, if the traveling wave is stationary ($\gamma=0$), then $A$ has no 
$T$-dependence so the sharper bounds can be obtained.

\begin{corollary}
	If $\gamma=0$, then (\ref{eq:upper and lower bounds for J(A), general gamma}) can 
	be refined:
	\begin{equation}
		\label{eq:upper and lower bounds for J(A), gamma=0}
		\mathcal{J}(A) = \inf_{u\in A}I(u) \overset{l_c\to 0}{=}
		\begin{cases}
			\Theta(x_0^2/T),\quad &\text{for $|x_0|$ small,}\\
			\Theta(|x_0|),\quad   &\text{for $|x_0|$ large and any fixed $T$,}\\
			\Theta(|x_0|/T),\quad &\text{for $|x_0|$ large and $T$ small,}\\
			\Theta(1/T),\quad     &\text{for any fixed $x_0$.}
		\end{cases}
	\end{equation}
\end{corollary}
\begin{proof}
	See Section \ref{sec:bounds for small |x_0|} and \ref{sec:bounds for large |x_0|}.
\end{proof}

\textbf{Remark.} The case that both $|x_0|$ and $T$ are large is not clear even if 
$\gamma=0$. This is because $\mathcal{J}(A)=\Theta(|x_0|/T)$ for $|x_0|\gg T$, and  
$\mathcal{J}(A)=\Theta(x_0^2/T)$ for $|x_0|\ll T$; when both $x_0$ and $T$ are large, 
$\mathcal{J}(A)$ is the mixture of $\Theta(|x_0|/T)$ and $\Theta(x_0^2/T)$ so it is 
difficult to estimate the bounds.

In the rest of this subsection we discuss two relevant issues about this framework. We first show that the 
given kernel $\Phi^{l_c}_{L_0}$ in Assumption \ref{asmp:conditions of Phi} is indeed Hilbert-Schmidt from 
$L^2$ to $H^1$.
\begin{lemma}
	\label{lma:sifficient condition for Hilbert-Schmidt Phi}
	If $\phi_0$ and $\phi_1$ are both in $L^1\cap H^k$ for $k\geq 1$, then the kernel $\Phi^{l_c}_{L_0}$ of 
	the form (\ref{eq:covariance kernel Phi}) is Hilbert-Schmidt from $L^2$ to $H^k$.
\end{lemma}
\begin{proof}
	See Appendix \ref{pf:sifficient condition for Hilbert-Schmidt Phi}.
\end{proof}

The other issue is that $A$ has no interior and therefore $\mathcal{J}(\mathring{A})=\infty$, which gives a 
trivial lower bound in the large deviation framework; to avoid this triviality, a rigorous way is  to 
consider instead the (closed) rare event $A_\delta$ with a small $\delta>0$:
\begin{equation}
	A_\delta = \{u\in\mathcal{E}^1\mbox{ such that } u(0,\cdot)=U,\,  
	\|u(T,\cdot)-U(\cdot-\gamma T - x_0)\|_{H^1}\leq\delta\}.
\end{equation}
By Proposition \ref{prop:ldp}, we have 
\begin{equation}
	-\mathcal{J}(\mathring{A_\delta})
	\leq \liminf_{\eps\to 0} \eps^2 \log\PP(u^\eps\in A_\delta)
	\leq \limsup_{\eps\to 0} \eps^2 \log\PP(u^\eps\in A_\delta)
	\leq -\mathcal{J}(A_\delta).
\end{equation}
In addition, the following lemma shows that $\mathcal{J}(A_\delta)$ converges to $\mathcal{J}(A)$ as 
$\delta\to 0$.
\begin{lemma}
	\label{lma:lower bound of J(A_delta)}
	By definition $\mathcal{J}(A_\delta)$ is a decreasing function with $\delta$ and 
	bounded from above by $\mathcal{J}(A)$. In addition, 
	\[
		\lim_{\delta\to 0}\mathcal{J}(A_\delta) = \mathcal{J}(A).
	\]
\end{lemma}
\begin{proof}
	See Appendix \ref{pf:lower bound of J(A_delta)}.
\end{proof}

By Lemma \ref{lma:lower bound of J(A_delta)} and the fact that 
$\mathcal{J}(\mathring{A_\delta})\leq\mathcal{J}(A)$,
\begin{align*}
	-\mathcal{J}(A) \leq -\mathcal{J}(\mathring{A_\delta})
	&\leq \liminf_{\eps\to 0} \eps^2 \log\PP(u^\eps\in A_\delta)\\
	&\leq \limsup_{\eps\to 0} \eps^2 \log\PP(u^\eps\in A_\delta)
	\leq -\mathcal{J}(A_\delta)
	\overset{\delta\to 0}{\to} -\mathcal{J}(A).
\end{align*}
Namely, for $\eps$ and $\delta$ small, we have 
\[
	\PP(u^\eps\in A_\delta) \sim \exp\left(-\frac{1}{\eps^2}\mathcal{J}(A)\right),
\]
and it thus suffices to consider the bounds for $\mathcal{J}(A)$.

{\bf Remark.}
We can compare the results stated in Theorem \ref{thm:main}, valid under 
Assumption \ref{asmp:conditions of Phi}, and the ones stated in Proposition \ref{eq:Prare1}. First we can 
check that $A \subset B$. Second we can anticipate that in fact $B$ is not much larger than $A$ because it 
seems reasonable that the most likely paths that achieve ${\mathcal C}[u](T) \geq \gamma T + x_0$ should be 
of the stable form $U(x-\gamma T-x_0)$ at time $t=T$. This conjecture can indeed be verified as we now show.
Note that
$$
	\iint C(x,x') dx dx' = \int \Big[\int \Phi(x,x') dx \Big]^2 dx'   .
$$
If the kernel $\Phi^{l_c}_{L_0}$ is of the form (\ref{eq:covariance kernel Phi}), then we find
$$
	\iint C(x,x') dx dx' 
	= \frac{\sigma^2L_0}{2\pi} \int |\hat{\phi}_1(\frac{l_c}{L_0} \xi)|^2 |\hat{\phi}_0(\xi)|^2  d\xi  .
$$
By assuming $l_c \ll L_0$ (and $\hat{\phi}_1(0)=1$) as in Assumption \ref{asmp:conditions of Phi}, then this 
simplifies to
$$
	\iint C(x,x') dx dx' \simeq  \sigma^2L_0  \int   \phi_0(x)^2  dx  .
$$
Therefore, from Theorem \ref{thm:main} (in which $L_0=\gamma T+x_0+C_0$), we get 
$$
	 \iint C(x,x') dx dx' 
	=\begin{cases}
		\Theta(x_0),\quad &\text{for $x_0$ large and $T$ small or fixed,}\\
		\Theta(1),\quad     &\text{for $x_0$ and $T$ small or fixed,}
	\end{cases}
$$
which gives with (\ref{eq:Prare1})
$$
	\PP( u^\eps \in B) \sim \exp \left( - \frac{1}{\eps^2} {\mathcal J}(B) \right)
$$
with
$$
	{\mathcal J}(B)
	=\begin{cases}
		\Theta(x_0^2),\quad   &\text{for $|x_0|$ small and any fixed $T$,}\\
		\Theta(x_0^2/T),\quad &\text{for $|x_0|$ small and $T$ small,}\\
		\Theta(|x_0|),\quad   &\text{for $|x_0|$ large and any fixed $T$,}\\
		\Theta(|x_0|/T),\quad &\text{for $|x_0|$ large and $T$ small,}\\
		\Theta(1/T),\quad     &\text{for $T$ small and any fixed $x_0$,}
	\end{cases}
$$
as in (\ref{eq:upper and lower bounds for J(A), general gamma}). We therefore recover the same asymptotic 
scales as in Theorem \ref{thm:main}.
\subsection{Proof of Theorem \ref{thm:main} for small $|x_0|$}
\label{sec:bounds for small |x_0|}

Here we prove the bounds for $\mathcal{J}(A)$ in (\ref{eq:upper and lower bounds for J(A), general gamma}) 
and (\ref{eq:upper and lower bounds for J(A), gamma=0}) when $|x_0|$ is small and the bound $\Theta(1/T)$ 
for fixed $x_0$. We first consider the upper bounds. By definition, $\mathcal{J}(A)\leq I(u)$ for any 
$u\in A$; therefore the idea to obtain the upper bound is to find good test functions $u$. 

For $|x_0|$ small, we use the linear shifted profile as the test function:
\[
	 v(t,x)=U(x- (t/T)(\gamma T+x_0)) .
\]
If the kernel $\Phi^{l_c}_{L_0}\approx\sigma^{-2}$ in the region of interest, then 
\begin{equation}
	\label{eq:approximate I(v)}
	\mathcal{J}(A)=\inf_{u\in A}I(u)\leq I(v) 
	\lesssim \frac{1}{2\sigma^2}\int_0^T \|v_t + F(v)_x - Dv_{xx}\|_{L^2}^2 dt.
\end{equation}
We show (\ref{eq:approximate I(v)}) rigorously by the following lemmas:
\begin{lemma}
	\label{lma:approximation of phi_0 for small |x_0|}
	Given the traveling wave $U(x)$ in (\ref{eq:rare event of the wave dispacement}), there exists 
	$C_v\geq 1$ such  that, uniformly in $x_0$ and $T$ and for all $t\in[0,T]$,
	\begin{equation}
		\label{eq:L^2 approximation of phi_0 for small |x_0|}
		\|[\phi_0^{-1}(\cdot/L_0)-1]U_x(\cdot-(t/T)(\gamma T+x_0))\|_{L^2}^2 \leq 1.
	\end{equation}
	where $L_0 = \gamma T + x_0 + C_v$.
\end{lemma}
\begin{proof}
	See Appendix \ref{pf:approximation of phi_0 for small |x_0|}.
\end{proof}

\begin{lemma}
	\label{lma:approximation of phi_1 for small |x_0|}
	Let $v(t,x)=U(x-(t/T)(\gamma T+x_0))$. Given $L_0=x_0+\gamma T + C_v$ in Lemma 
	\ref{lma:approximation of phi_0 for small |x_0|}, $\Phi^{l_c}_{L_0}$ satisfying Assumption 
	\ref{asmp:conditions of Phi}, and $h^{l_c}_{L_0}$ defined by 
	\begin{equation}
		\label{eq:driving force expression of v}
		v_t + F(v)_x -Dv_{xx} = \Phi^{l_c}_{L_0} h^{l_c}_{L_0},
	\end{equation}
	then 
	\begin{equation}
		\label{eq:L^2 approximation of phi_1 for small |x_0|}
		\int_0^T \|h^{l_c}_{L_0}(t,\cdot)\|_{L^2}^2 dt \overset{l_c\to 0}{\to}
		\frac{1}{\sigma^2} \int_0^T \|(v_t+F(v)_x-Dv_{xx})/\phi_0(\cdot /L_0)\|_{L^2}^2 dt.
	\end{equation}
\end{lemma}
\begin{proof}
	See Appendix \ref{pf:approximation of phi_1 for small |x_0|}.
\end{proof}

Now we are ready to prove (\ref{eq:approximate I(v)}). 
By Lemma \ref{lma:approximation of phi_1 for small |x_0|}, we have
\begin{equation*}
	I(v) \leq \frac{1}{2} \int_0^T \|h^{l_c}_{L_0}(t,\cdot)\|_{L^2}^2 dt
	\overset{l_c\to 0}\to\frac{1}{2\sigma^2} \int_0^T \|(v_t+F(v)_x-Dv_{xx})/\phi_0(\cdot /L_0)\|_{L^2}^2 dt.
\end{equation*}
Then the rest is to compute $\int_0^T\|(v_t+F(v)_x-Dv_{xx})/\phi_0(\cdot /L_0)\|_{L^2}^2 dt$. Note that
\[
	\big(v_t+F(v)_x-Dv_{xx}\big)(t,x) = \frac{x_0}{T} U_x(x-(t/T)(\gamma T+x_0)).
\]
With $L_0 = \gamma T + x_0 + C_v$ given in Lemma \ref{lma:approximation of phi_0 for small |x_0|},
we have 
$$
\|U_x(\cdot-(t/T)(\gamma T+x_0))/\phi_0(\cdot /L_0)\|_{L^2}^2  \leq 2 \|U_x(\cdot-(t/T)(\gamma T+x_0))\|_{L^2}^2 + 2,
$$
and therefore
\begin{align*}
	&\int_0^T \|(v_t+F(v)_x-Dv_{xx})/\phi_0(\cdot /L_0)\|_{L^2}^2 dt \\
	&\quad =\frac{x_0^2}{T^2}\int_0^T\|U_x(\cdot-(t/T)(\gamma T+x_0))/\phi_0(\cdot /L_0)\|_{L^2}^2dt\\
	&\quad \leq\frac{x_0^2}{T^2}\int_0^T\big( 2 \|U_x(\cdot-(t/T)(\gamma T+x_0))\|_{L^2}^2 + 2\big)dt
	= \frac{x_0^2}{T} \big(2 \|U_x\|_{L^2}^2 + 2\big).
\end{align*}
Therefore, we have the asymptotic upper bounds for $\mathcal{J}(A)$:
\begin{equation}
	\label{eq:upper bounds for J(A) by using I(v)}
	\mathcal{J}(A) \leq I(v) 
	\overset{l_c\to 0}{\leq} \frac{1}{2\sigma^2}\frac{x_0^2}{T}\big( 2\|U_x\|_{L^2}^2 + 2\big).
\end{equation}

Now we find the lower bounds for $\mathcal{J}(A)$. 
Let us first denote by $1$ the function identically equal to $1$, 
assume that $(\Phi^{l_c}_{L_0})^T 1\in L^2$, 
and find the general form of the lower bound.
\begin{proposition}
	\label{prop:general form of lower bounds}
	Given $(\Phi^{l_c}_{L_0})^T 1\in L^2$, for any $u\in A$,
	\begin{equation}
		\label{eq:general form of lower bounds}
		I(u)\geq \frac{1}{2}T^{-1}\|(\Phi^{l_c}_{L_0})^T 1\|_{L^2}^{-2}\left(\int[U(x-x_0)-U(x)]dx\right)^2.
	\end{equation}
\end{proposition}
\begin{proof}
	See Appendix \ref{pf:general form of lower bounds}.
\end{proof}

Then we show that $(\Phi^{l_c}_{L_0})^T 1$ is indeed in $L^2$.
\begin{lemma}
	\label{lma:Phi^T 1}
	Let $\Phi^{l_c}_{L_0}$ satisfy Assumption \ref{asmp:conditions of Phi}. 
	Then $(\Phi^{l_c}_{L_0})^T 1(x)$ is in $L^2(\RR)$, and 
	$\|(\Phi^{l_c}_{L_0})^T 1\|_{L^2}^2\to \sigma^2 L_0 \|\phi_0\|_{L^2}^2$ as $l_c\to 0$.
\end{lemma}
\begin{proof}
	See Appendix \ref{pf:Phi^T 1}.
\end{proof}

From Lemma \ref{lma:approximation of phi_0 for small |x_0|}, $L_0 = \gamma T + x_0 + C_v$, where $C_v$ is 
uniform in $x_0$ and $T$. Then for $l_c$ small the general lower bound 
(\ref{eq:general form of lower bounds}) becomes:
\begin{align}
	\label{eq:aymptotic general lower bounds for J(A) by letting l_c to 0}
	\mathcal{J}(A) 
	&\geq \frac{1}{2}T^{-1}\|(\Phi^{l_c}_{L_0})^T 1\|_{L^2}^{-2}
	\left(\int[U(x-x_0)-U(x)]dx\right)^2 \\
	&\overset{l_c\to 0}{\to}\frac{1}{2\sigma^2}\|\phi_0\|_{L^2}^{-2}T^{-1}(x_0+\gamma T+C_v)^{-1}
	\left(\int[U(x-x_0)-U(x)]dx\right)^2.\notag
\end{align}
For $|x_0|$ small, $U(x-x_0)=U(x)-x_0U_x(x)+\mathcal{O}(x_0^2)$. Then
\begin{align}
	\label{eq:lower bound with small |x_0|}
	\mathcal{J}(A) &\overset{l_c\to 0}{\geq} 
	\frac{1}{2\sigma^2}\|\phi_0\|_{L^2}^{-2}T^{-1}(x_0+\gamma T+C_v)^{-1}
	\left(\int[-x_0U_x(x)+\mathcal{O}(x_0^2)]dx\right)^2\\
	&=\frac{1}{2\sigma^2}\|\phi_0\|_{L^2}^{-2}T^{-1}(x_0+\gamma T+C_v)^{-1}
	[x_0^2(u_+-u_-)^2 + \mathcal{O}(x_0^3)]. \notag
\end{align}

To summarize this subsection, given $A$ in (\ref{eq:rare event of the wave dispacement}), we choose the 
kernel $\Phi^{l_c}_{L_0}$ with $L_0=\gamma T+x_0+C_v$, which is the support of the noise covering the region 
of interest of $A$, and then by (\ref{eq:upper bounds for J(A) by using I(v)}) and 
(\ref{eq:lower bound with small |x_0|}), $\mathcal{J}(A)$ has the following asymptotic (in the sense that 
$l_c\to 0$) bounds:
\begin{equation*}
	\mathcal{J}(A) \overset{l_c\to 0}{=}
	\begin{cases}
		\Theta(x_0^2),\quad   &\text{for $|x_0|$ small and any fixed $T$,}\\
		\Theta(x_0^2/T),\quad &\text{for $|x_0|$ small and $T$ small,}\\
		\Theta(1/T),\quad     &\text{for $T$ small and any fixed $x_0$,}
	\end{cases}
\end{equation*}
for nonzero $\gamma$, and for $\gamma=0$.
\begin{equation*}
	\mathcal{J}(A) \overset{l_c\to 0}{=}
	\begin{cases}
		\Theta(x_0^2/T),\quad &\text{for $|x_0|$ small,}\\
		\Theta(1/T),\quad     &\text{for any fixed $x_0$.}
	\end{cases}
\end{equation*}
\subsection{Proof of Theorem \ref{thm:main} for large $|x_0|$}
\label{sec:bounds for large |x_0|}

Now we consider the bounds for $\mathcal{J}(A)$ in (\ref{eq:upper and lower bounds for J(A), general gamma}) 
and (\ref{eq:upper and lower bounds for J(A), gamma=0}) when $|x_0|$ is large. We use the test function 
\begin{equation}
	\label{eq:linear interpolation of the profiles}
	w(t,x)=(1-t/T)U(x)+(t/T)U(x-\gamma T-x_0),
\end{equation}
the linear interpolation of the two profiles, and show that 
\begin{equation}
	\label{eq:approximate I(w)}
	\mathcal{J}(A)=\inf_{u\in A}I(u)\leq I(w) 
	\lesssim \frac{1}{2\sigma^2}\int_0^T \|w_t + F(w)_x - Dw_{xx}\|_{L^2}^2 dt.
\end{equation}

We also show (\ref{eq:approximate I(w)}) by two similar technical lemmas.
\begin{lemma}
	\label{lma:approximation of phi_0 for large |x_0|}
	Given the traveling wave $U(x)$ in (\ref{eq:rare event of the wave dispacement}) and $w(t,x)$ in 
	(\ref{eq:linear interpolation of the profiles}), there exists $C_w\geq 1$ such that, uniformly in $x_0$ and $T$ 
	and for all $t\in[0,T]$,
	\begin{align}
		\label{eq:L^2 approximation of phi_0 for large |x_0|}
		&\|[\phi_0^{-1}(\cdot/L_0)-1][U(\cdot-\gamma T-x_0)-U]\|_{L^2}^2 \leq 1,\\
		&\|[\phi_0^{-1}(\cdot/L_0)-1][F(w)_x -Dw_{xx}]\|_{L^2}^2 \leq 1, \notag
	\end{align}
	where $L_0 = \gamma T + x_0 + C_w$.
\end{lemma}
\begin{proof}
	See Appendix \ref{pf:approximation of phi_0 for large |x_0|}.
\end{proof}

Because $w$ has the same property that $w_t+F(w)_x-Dw_{xx}\in\mathcal{C}([0,T],\mathcal{S}(\RR))$, we have 
the following lemma by replacing $v$ by $w$ in Lemma \ref{lma:approximation of phi_1 for small |x_0|}.
\begin{lemma}
	\label{lma:approximation of phi_1 for large |x_0|}
	Let $w$ in (\ref{eq:linear interpolation of the profiles}). Given $L_0=x_0+\gamma T + C_w$ in Lemma 
	\ref{lma:approximation of phi_0 for large |x_0|}, $\Phi^{l_c}_{L_0}$ satisfying Assumption 
	\ref{asmp:conditions of Phi}, and $h^{l_c}_{L_0}$ defined by 
	$w_t + F(w)_x -Dw_{xx} = \Phi^{l_c}_{L_0} h^{l_c}_{L_0}$,
	then 
	\begin{equation*}
		I(w) \leq \frac{1}{2} \int_0^T \|h^{l_c}_{L_0}(t,\cdot)\|_{L^2}^2 dt \overset{l_c\to 0}{\to}
		\frac{1}{2\sigma^2} \int_0^T \|(w_t+F(w)_x-Dw_{xx})/\phi_0(\cdot/L_0)\|_{L^2}^2 dt.
	\end{equation*}
\end{lemma}

The rest is to compute $\int_0^T\|(w_t+F(w)_x-Dw_{xx})/\phi_0(\cdot/L_0)\|_{L^2}^2 dt$. Note that
\[
	w_t+F(w)_x-Dw_{xx} = \frac{1}{T}[U(x-\gamma T-x_0)-U(x)] + F(w)_x - Dw_{xx}.
\]
With $L_0 = \gamma T + x_0 + C_w$ given in Lemma \ref{lma:approximation of phi_0 for large |x_0|},
\begin{align*}
	\lefteqn{\int_0^T \|(w_t+F(w)_x-Dw_{xx})/\phi_0(\cdot/L_0)\|_{L^2}^2 dt}\\
	&\leq \frac{2}{T^2}\int_0^T\|[U(\cdot-\gamma T-x_0)-U]/\phi_0(\cdot/L_0)\|_{L^2}^2dt\\
	&\quad + 2\int_0^T\|[F(w)_x - Dw_{xx}]/\phi_0(\cdot/L_0)\|_{L^2}^2dt\\
	&\leq \frac{2}{T^2}\int_0^T\big(2 \|[U(\cdot-\gamma T-x_0)-U]\|_{L^2}^2 + 2\big)dt\\
	&\quad +  2\int_0^T\big( 2\|F(w)_x - Dw_{xx}\|_{L^2}^2 + 2 \big)dt.
\end{align*}
We note that for $|x_0|$ large, $\|U(\cdot-\gamma T-x_0)-U\|^2_{L^2}=\Theta(|x_0|)$, and 
$\|F(w)_x-Dw_{xx}\|^2_{L^2}=\mathcal{O}(1)$. Then for $L_0 = \gamma T + x_0 + C_w$ given in Lemma 
\ref{lma:approximation of phi_0 for large |x_0|}, we have the upper bounds for $\mathcal{J}(A)$ for $|x_0|$ 
large:
\begin{equation}
	\label{eq:upper bounds for J(A) by using I(w)}
	\mathcal{J}(A)\leq I(w) \overset{l_c\to 0}{\leq} \Theta(|x_0|/T) + \Theta(T).
\end{equation}

Now we consider the lower bound. We know that 
(\ref{eq:aymptotic general lower bounds for J(A) by letting l_c to 0}) still holds for $|x_0|$ large and 
with $C_w$ given in Lemma \ref{lma:approximation of phi_0 for large |x_0|}. For $|x_0|$ large, 
$\int[U(x-x_0)-U(x)]dx=\Theta(|x_0|)$ so
\begin{equation}
	\label{eq:lower bound with large |x_0|}
	\mathcal{J}(A) \overset{l_c\to 0}{\geq} 
	\frac{1}{2\sigma^2}\|\phi_0\|_{L^2}^{-2}T^{-1}(x_0+\gamma T+C_w)^{-1}\Theta(x_0^2).
\end{equation}
Consequently, by combining (\ref{eq:upper bounds for J(A) by using I(w)}) and 
(\ref{eq:lower bound with large |x_0|}), we have the bounds in 
(\ref{eq:upper and lower bounds for J(A), general gamma}) and 
(\ref{eq:upper and lower bounds for J(A), gamma=0}) for $|x_0|$ large:
\begin{equation*}
	\mathcal{J}(A) \overset{l_c\to 0}{=}
	\begin{cases}
		\Theta(|x_0|),\quad   &\text{for $|x_0|$ large and any fixed $T$,}\\
		\Theta(|x_0|/T),\quad &\text{for $|x_0|$ large and $T$ small,}
	\end{cases}
\end{equation*}
\section{Large deviations for discrete conservation laws}
\label{sec:discrete LDP}

Conservation laws can only be solved numerically, in general, so we need to consider space-time
discretizations. For the calculation of small probabilities of large deviations, we may pass directly to the 
calculation of the infimum of the rate function, which we know analytically. First we discretize it in 
space and time and then use a suitable optimization method to find the approximate minimizer and the 
approximate value of the rate function. This way of calculating probabilities of large deviations has been 
carried out previously \cite{E2004} for different stochastically driven partial differential equations.
More involved methods that use adaptive meshes are discussed in \cite{Zhou2008}.

In the next subsection we show briefly that the rate function for large deviations of discrete conservation 
laws using Euler schemes is, as expected, the corresponding discretization of the continuum rate function.

\subsection{Large deviations with Euler schemes}
\label{sec:LD for Euler}

To formulate the discrete problem, we discretize the space and time domains with uniform grids, with 
$L=x_0<\cdots<x_M=R$, $M\Delta x=R-L$ and $0=t_0<\cdots<t_N=T$, $N\Delta t=T$. Let $Q_m^n$ denote the 
average of $u$ over the $m$-th cell $[x_{m-1},x_{m}]$ (whose center is $x_{m-1/2} =(x_{m-1}+x_m)/2$) at time $n\Delta t$, and $Q^n$ denote the vector 
$(Q_1^n,\ldots,Q_M^n)^T$. The Euler scheme for the conservation law is 
\begin{equation}
	\label{eq:discrete conservation laws}
	Q_m^{n+1} 
	= Q_m^n - \frac{\Delta t}{\Delta x}(F_{m+\frac{1}{2}}^n-F_{m-\frac{1}{2}}^n)
	+ D\frac{\Delta t}{\left(\Delta x\right)^2}(Q_{m+1}^n-2Q_m^n+Q_{m-1}^n)
	+ \eps\Delta W_m^n,
\end{equation}
for $m=2,\ldots,M-1$, $n=0,\ldots,N-1$,
where $F_{m\pm 1/2}^n(Q^n)$ are numerical fluxes constructed by standard finite volume methods such as 
Godunov or local-Lax-Friedrichs (LLF), and possibly with higher order ENO (Essentially Non-Oscillatory) 
reconstructions. The initial conditions $(Q^0_m)_{m=1,\ldots,M}$ are given by (\ref{eq:icburgers}),
the boundary conditions $(Q^n_1)_{n=1,\ldots,N}$  and $(Q^n_M)_{n=1,\ldots,N}$ are given by (\ref{eq:bcburgers}), and the
fluxes are given explicitly in (\ref{eq:Burgers fluxes by Godunov}), in the next 
subsection, for Burgers' equation. We let 
\[
	b(Q^n)=(b_2(Q^n),\ldots,b_{M-1}(Q^n))^T,
\]
where
\[
	b_m(Q^n) = -\frac{1}{\Delta x}(F_{m+\frac{1}{2}}^n-F_{m-\frac{1}{2}}^n)
	+D\frac{1}{(\Delta x)^2}(Q_{m+1}^n-2Q_m^n+Q_{m-1}^n).
\]
Let $(\Delta W_m^n)_{m=2,\ldots,M-1,n=1,\ldots,N} $ be Gaussian random variables with mean $0$ and covariance 
\[
	\mathbb{E}(\Delta W_{m_1}^{n_1}\Delta W_{m_2}^{n_2})
	=\begin{cases}
		\frac{\Delta t}{\Delta x}C_{m_1 m_2},\quad &n_1=n_2,\\
		0,\quad                                    &\text{otherwise.}
	\end{cases}
\]
The matrix $C=(C_{ij})_{i,j=2,\ldots,M-1}$ is symmetric and non-negative definite. For simplicity, we assume that $C$ is positive 
definite, and then $C=\Phi\Phi^T$ for an invertible matrix $\Phi$. 

By the Markov property, the joint density function $f_{(Q^0,\ldots,Q^N)}$ is 
\[
	f_{(Q^0,\ldots,Q^N)}(q^1,\ldots,q^N) = \prod_{n=0}^{N-1}f_{Q^{n+1}|Q^n}(q^{n+1};q^n),
\]
where 
\begin{align*}
	\lefteqn{f_{Q^{n+1}|Q^n}(q^{n+1};q^n)}\\
	&\quad = \frac{1}{Z}\exp\left[ -\frac{\Delta t \Delta x}{2\eps^2}
	\left(\frac{1}{\Delta t}(q^{n+1}-q^n)-b(q^n)\right)^T C^{-1}
	\left(\frac{1}{\Delta t}(q^{n+1}-q^n)-b(q^n)\right)\right]\\
	&\quad = \frac{1}{Z}\exp\left[ -\frac{\Delta t \Delta x}{2\eps^2} \left\Vert 
	\Phi^{-1}\left(\frac{1}{\Delta t}(q^{n+1}-q^n)-b(q^n)\right)\right\Vert _2^2\right] ,
\end{align*}
where $Z^2= (2 \pi \eps^2 \Delta t/\Delta x)^{M-2} \det C$ and the $l^2$-norm is here 
$$
\| q\|_2^2 = \sum_{j=2}^{M-1} q_j^2 .
$$
The exact probability that $Q=(Q^1,\ldots,Q^N)\in A$ with the given $Q^0$ is therefore
\begin{multline*}
	\mathbb{P}(Q\in A)\\
	=\int_{A}\frac{1}{Z^N}\exp\left[-\frac{\Delta t \Delta x}{2\eps^2}
	\sum_{n=0}^{N-1}\left\Vert \Phi^{-1}\left(\frac{1}{\Delta t}
	(q^{n+1}-q^n)-b(q^n)\right)\right\Vert_2^2\right] dq^1\cdots dq^N.
\end{multline*}
Because the problem is finite-dimensional, the LDP is a form of Laplace's method for the asymptotic evaluation 
of integrals \cite{Varadhan1966} and the rate function for $q=(q^1,\ldots,q^N)$ is 
\begin{equation}
	\label{eq:discrete rate function}
	I(q )=\frac{\Delta t \Delta x}{2}\sum_{n=0}^{N-1} \left\Vert 
	\Phi^{-1}\left(\frac{1}{\Delta t}(q^{n+1}-q^n)-b(q^n)\right)\right\Vert_2^2.
\end{equation}

From this expression it is clear that the rate function of the continuous conservation law is the limit of 
$\eqref{eq:discrete rate function}$ as $\Delta t$, $\Delta x\to 0$. However, the LDP of the discrete 
conservation law does not immediately imply that of the continuous case. The limit of $\eps$ and the 
limit of $\Delta t$, $\Delta x$ need to be interchangeable. In the language of large deviations, the law of 
the discrete conservation law has to be exponentially equivalent to that of the continuous one 
\cite{Dembo2010}. Without going into the proof of this, we can only say that 
(\ref{eq:discrete rate function}) is a discretization of the rate function of the continuous conservation 
law. 

\subsection{Numerical calculation of rate functions for changes in traveling waves}

In the numerical simulations we consider the rare event that a traveling wave at time $0$ becomes a 
different traveling wave at time $T$ due to the small random perturbations.

We carry out numerical calculations with Burgers' equation as a simple but representative conservation 
law. For convenience, given a traveling wave $U_0$ with the speed $\gamma$, we consider Burgers' equation 
in moving coordinates:
\begin{equation}
	\label{eq:Burgers eqn in the moving coordinate}
	u_t + \Big(\frac{1}{2}(u-\gamma)^2\Big)_x = (Du_x)_x + \eps\dot{W}.
\end{equation}
Thus $U_0$ is a traveling wave of (\ref{eq:Burgers eqn in the moving coordinate}) with speed zero. We are 
interested in the rare event that $u(0,x)=U_0(x)$ and $u(T,x)=U_T(x)$ because of $\eps \dot{W}$, where $U_0$ 
and $U_T$ are two different traveling waves.

We therefore minimize the rate function (\ref{eq:discrete rate function}) to compute the asymptotic 
probability. The initial conditions
\begin{equation}
\label{eq:icburgers}
q_m^0=U_0(x_{m-1/2}) \mbox{ for all } m=1,\ldots,M
\end{equation}
and the  terminal conditions 
\begin{equation}
\label{eq:tcburgers}
q_m^N=U_T(x_{m-1/2})\mbox{ for all } m=1,\ldots,M 
\end{equation}
are simply the values of the functions at the centers of 
the cells.
For $m=2,\ldots,M-1$ 
we let $F_{m\pm 1/2}^n$ be the numerical fluxes for $(u-\gamma)^2/2$ at $x_{m\pm 1/2}$. We use Godunov's method 
\cite{LeVeque2002} to construct $F_{m\pm 1/2}^n$:
\begin{equation}
\label{eq:Burgers fluxes by Godunov}
	F_{m-1/2}^n
	=\begin{cases}
		\min_{q_{m-1}^n\leq q\leq q_m^n}\frac{1}{2}(q-\gamma)^2,\quad
		&q_{m-1}^n\leq q_m^n,\\
		\max_{q_m^n\leq q\leq q_{m-1}^n}\frac{1}{2}(q-\gamma)^2,\quad 
		&q_m^n\leq q_{m-1}^n.
	\end{cases}
\end{equation}

When the final traveling wave solution $U_T$ is the shifted initial profile $U_0(x-x_0)$ for some $x_0$, 
then we impose the boundary conditions 
\begin{equation}
\label{eq:bcburgers}
q_1^n= u_- \mbox{  and } q_M ^n= u_+\mbox{ for all } n=1,\ldots,N ,
\end{equation}
where 
$u_\pm = \lim_{x \to \pm \infty} U_0(x)$. We will discuss the boundary conditions imposed in the other cases 
later on.

The covariance matrix of the random coefficients $(\Delta W_i^n)_{i=2,\ldots,M-1}$ is set to be equal to  
$C_{ij}=\delta_{ij}$ and thus $\Phi=I_{M-2\times M-2}$. In other words, $(\Delta W_i^n)_{i=2,\ldots,M-1}$ are i.i.d. 
Gaussian random variables. Since the discrete problem is finite dimensional, $\Phi$ is clearly a 
Hilbert-Schmidt matrix.

The objective is to minimize the rate function (\ref{eq:discrete rate function}). Because the initial, 
terminal and boundary conditions are easily integrated into the definition of the discrete rate function, we 
can minimize $I(q)$ by unconstrained optimization methods. The BFGS quasi-Newton method 
\cite{Nocedal2006} is our optimization algorithm.

An important issue in numerical optimization is that any gradient-based method, for example the BFGS method 
that we use, only gives a local optimum unless the objective function is convex. In our case, it is not 
clear that the discrete rate function (\ref{eq:discrete rate function}) is convex or not. However, based on 
our numerical simulations we note the following.
\begin{enumerate}
	\item Our numerical results of the optimal shifted profiles coincide with the analytical predictions 
	(\ref{eq:upper and lower bounds for J(A), gamma=0}).

	\item Instead of using a good initial guess for the minimizer in the numerical optimization, we have 
	also numerically verified that completely random initial guesses give essentially the same result.
	 
	\item We have checked numerically to see if  the rate function (\ref{eq:discrete rate function}) is 
	convex. We randomly pick two close by test paths to see if the midpoint convexity of
	(\ref{eq:discrete rate function}) is satisfied.  We find that in $99.98\%$ out of $10^6$ pairs the 
	discrete rate function passes the convexity test. It is not known what causes the $0.02\%$ failures. We 
	conclude that based on numerical calculations (\ref{eq:discrete rate function}) is essentially convex, 
	if it is not fully convex. This explains the observed robustness of the numerical optimization.	
\end{enumerate}

\subsection{The numerical setup}

We use different $U_0$ and $U_T$ to calculate probabilities of several rare events. Our main interest is 
the anomalous wave displacement,  which is theoretically analyzed in the previous sections. The other cases are 
the wave speed change, the transition from a strong shock to a weak shock, and the transition from a strong 
shock to a weak shock. We have not carried out an analysis of the last three cases. However, the 
probabilities of these rare events can be calculated numerically and show how unlikely such events are 
compared to anomalous wave displacement.

In each configuration we consider the high viscosity case ($D=1$) and the low  viscosity case ($D=0.01$). 
$T=1$, $\Delta x=0.2$ and $\Delta t=0.02$ in all simulations and the linear interpolation of $U_0$ and $U_T$ 
is the initial guess in the numerical optimization. As we noted before, a random initial guess gives 
essentially the same result, but we use the linear interpolation to speed up the optimization.

\subsection{Anomalous wave displacements}

In this case, we let $U_0 $ be the traveling wave solution for Burgers' equation 
(\ref{eq:Burgers eqn in the moving coordinate}), and $U_T=U_0(x-x_0)$ represent a shifted traveling wave. 
This is the discrete version of (\ref{eq:rare event of the wave dispacement}), and we find that the 
numerical results are consistent with the analytical result 
(\ref{eq:upper and lower bounds for J(A), gamma=0}).

We first consider Fig.\ref{fig:the plot of optimal paths, shifted profiles with high D} and 
Fig.\ref{fig: I versus x_0 and T, shifted profiles with high D} with $D=1$. The optimal path is close to the 
linear shift when $x_0$ is small and it looks like the linear interpolation when $x_0$ is large. This also 
motivates us to choose the test functions $v$ and $w$ for the upper bounds of $\mathcal{J}(A)$ in Section 
\ref{sec:bounds for small |x_0|} and \ref{sec:bounds for large |x_0|}. Further, Fig.\ref{fig: I versus x_0 
and T, shifted profiles with high D} shows that the optimal $I$ is quadratic near $x_0=0$ and is linear for 
$x_0$ large, and is of order $1/T$ in $T$. These observations confirms the analysis 
(\ref{eq:upper and lower bounds for J(A), gamma=0}).

We consider next Fig.\ref{fig:the plot of optimal paths, shifted profiles with low D} and 
Fig.\ref{fig: I versus x_0 and T, shifted profiles with low D} with $D=0.01$. As the transition regions are 
very narrow and separate very quickly in the low viscosity case, the optimal path is nearly the linear 
interpolation, and therefore the $I$ versus $x_0$ plot is almost linear except around $x_0=0$. Moreover, the 
optimal rate function is still of order $1/T$ in $T$, which is also seen in the analysis 
(\ref{eq:upper and lower bounds for J(A), gamma=0}).

\begin{figure}
	\includegraphics[width=0.49\textwidth]
	{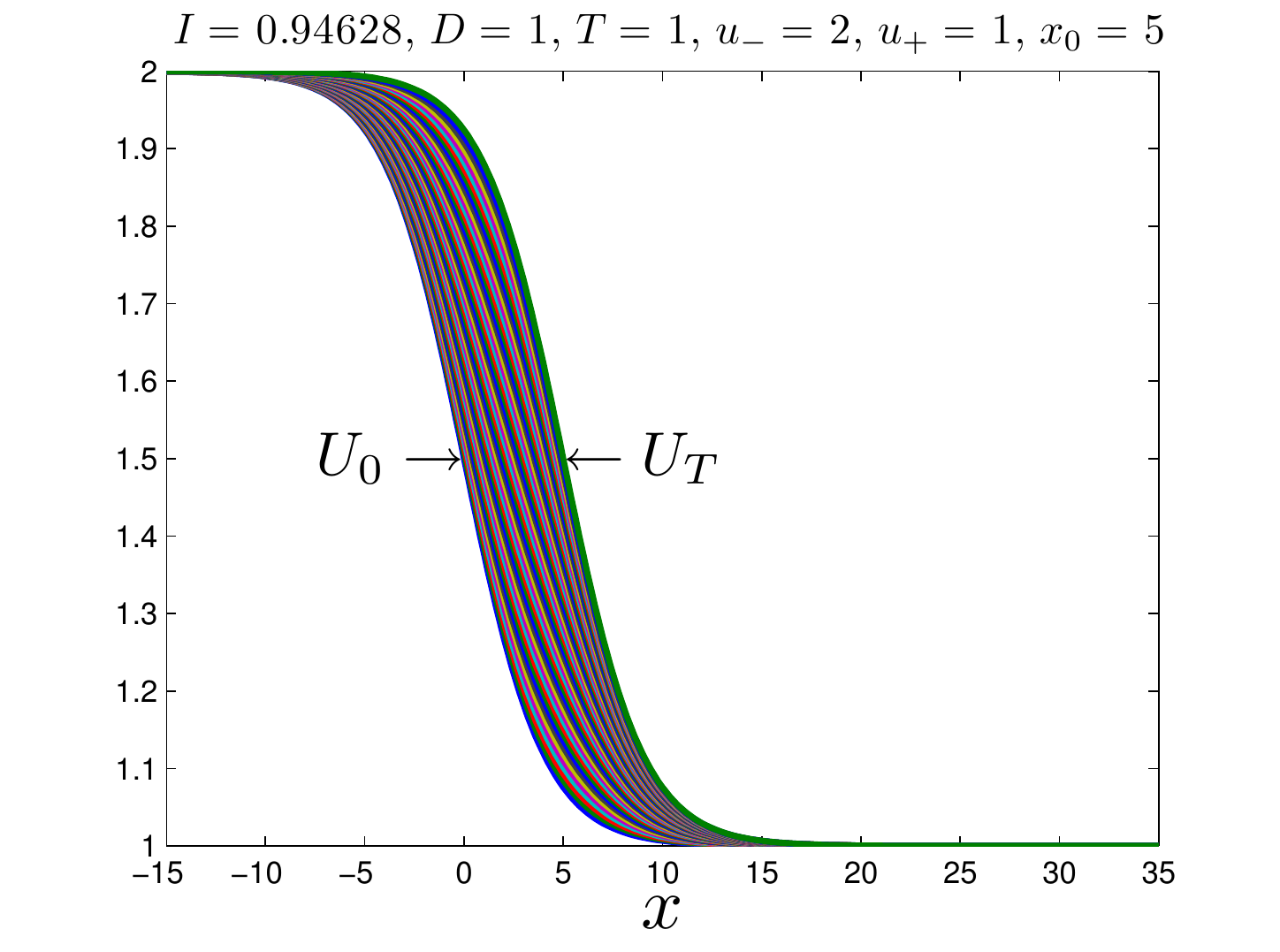}
	\includegraphics[width=0.49\textwidth]
	{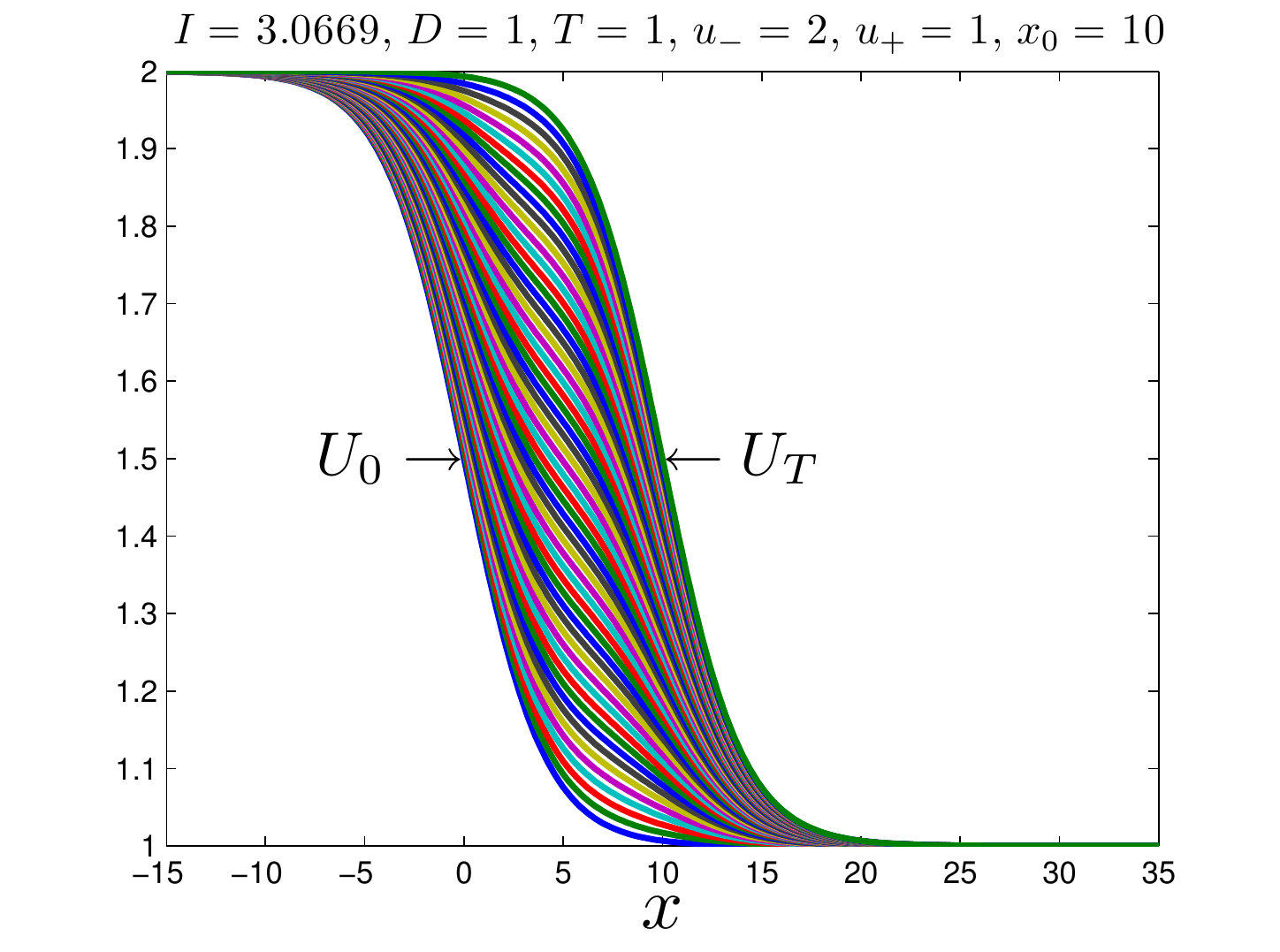}

	\includegraphics[width=0.49\textwidth]
	{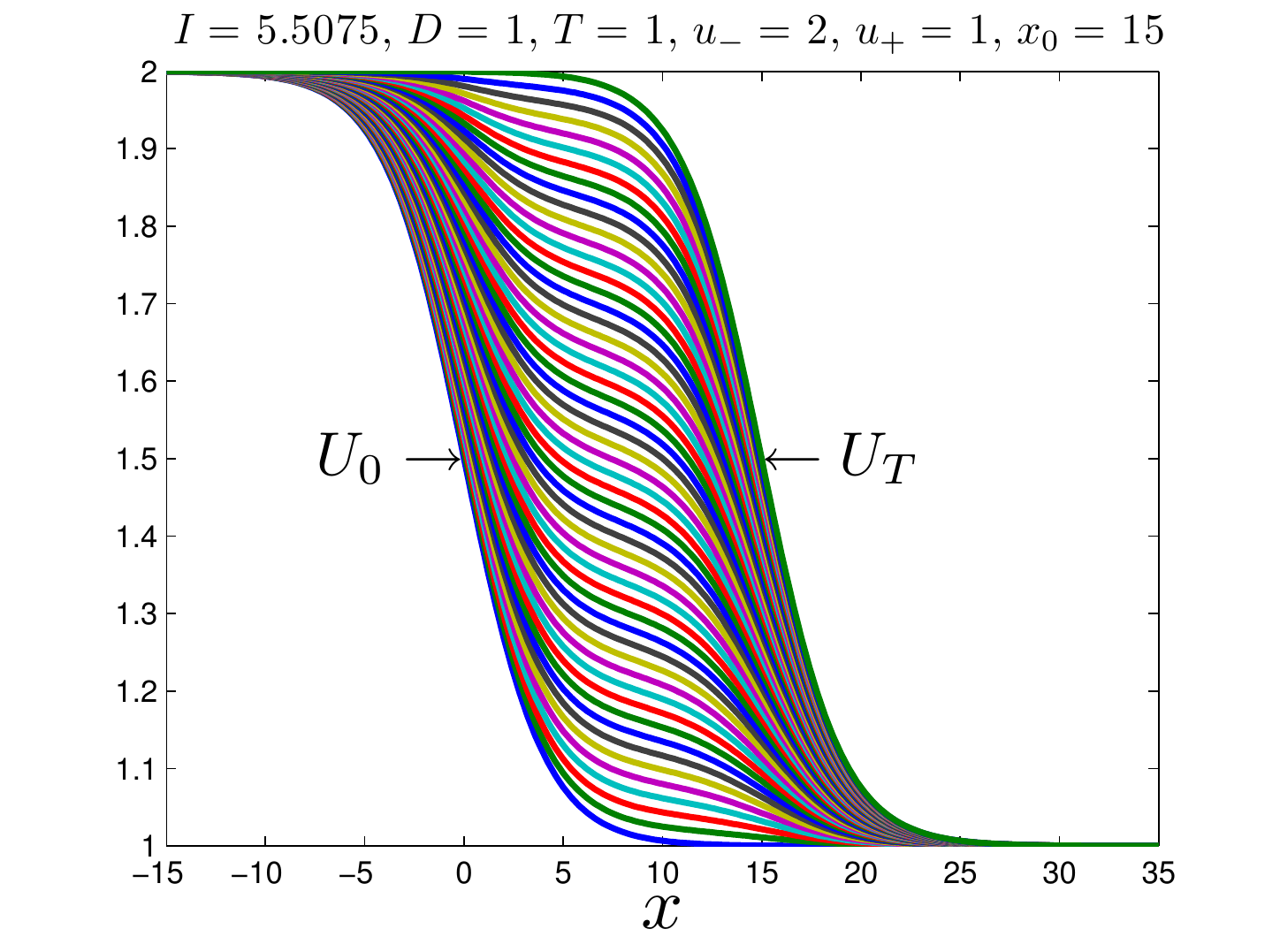}
	\includegraphics[width=0.49\textwidth]
	{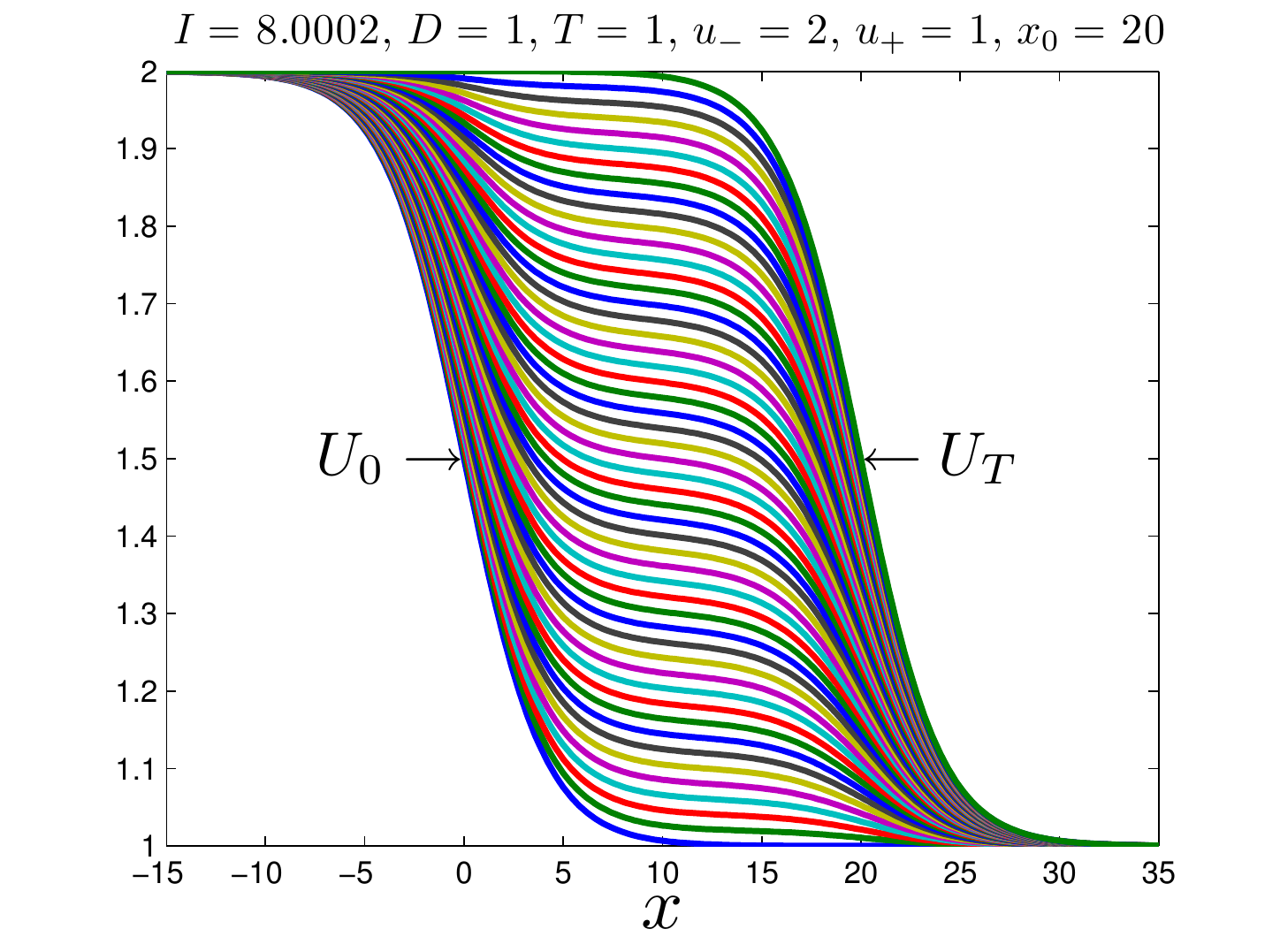}
	
	\caption{
	\label{fig:the plot of optimal paths, shifted profiles with high D}
	The optimal paths and their values of $I$ in (\ref{eq:discrete rate function}) in the case that 
	$U_T(x)=U_0(x-x_0)$ for different $x_0$. In each figure, we plot the curves indicating the optimal path 
	at time $0,\Delta t,2\Delta t,\ldots,T=1$. We can see that for $x_0$ small, the optimal path is close to 
	the linear shift traveling wave and for $x_0$ large, the optimal path looks like the linear 
	interpolation of $U_0$ and $U_T$. These observations are consistent with the test function $v$ and 
	$w$ we use for the upper bounds in Section \ref{sec:bounds for small |x_0|} and 
	\ref{sec:bounds for large |x_0|}.}
\end{figure}

\begin{figure}
	\includegraphics[width=0.49\textwidth]
	{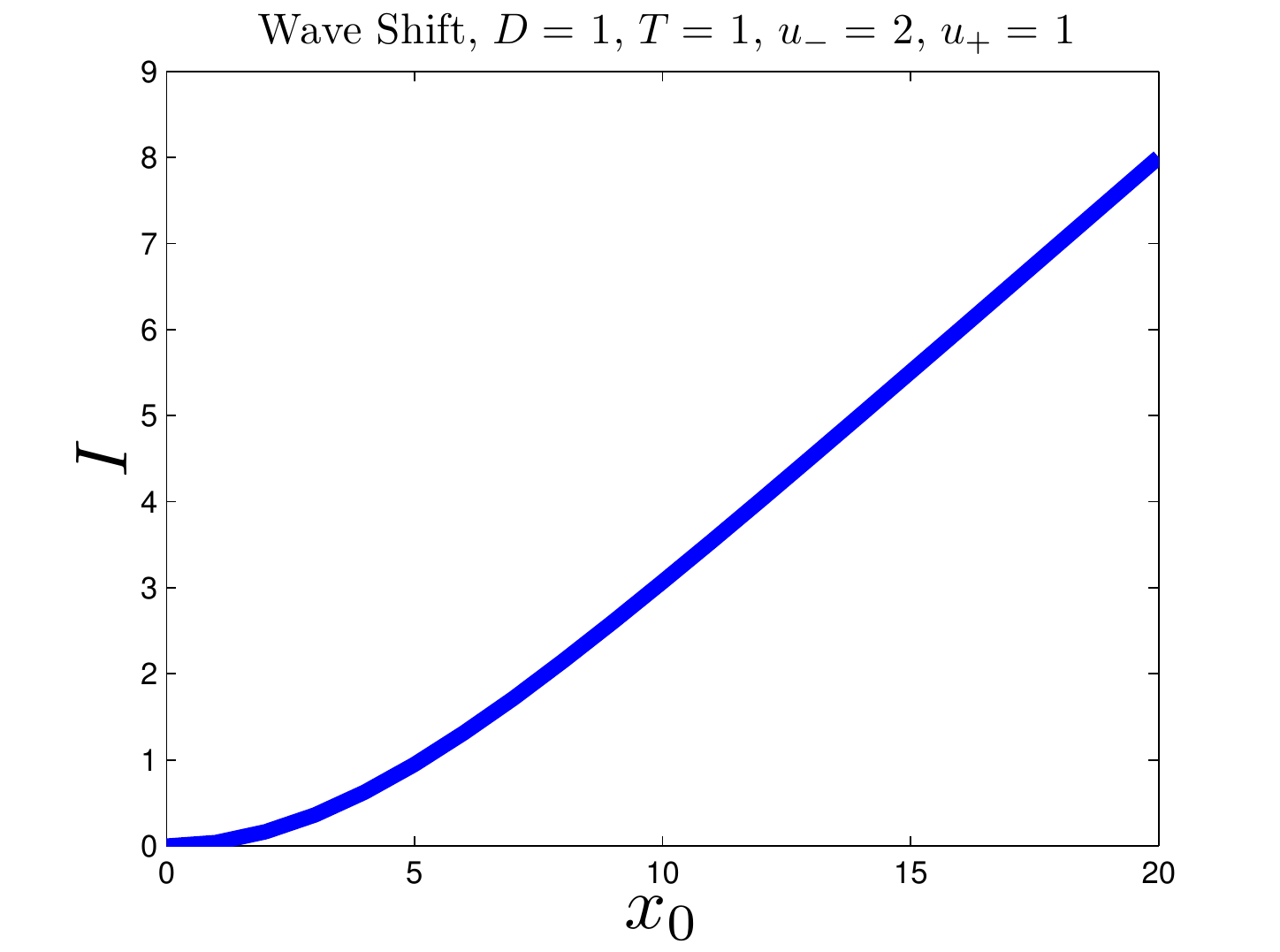}
	\includegraphics[width=0.49\textwidth]
	{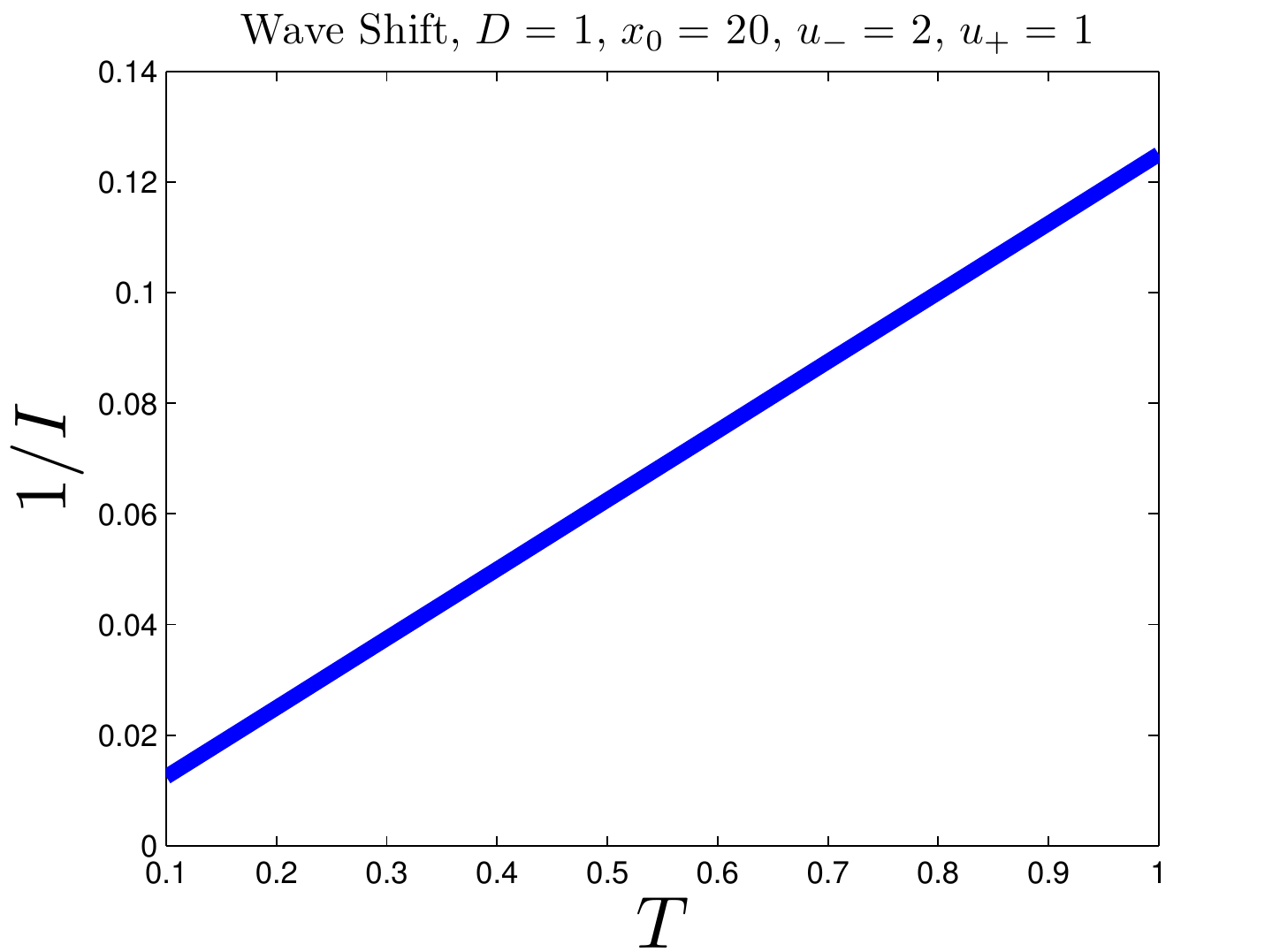}
	
	\caption{
	\label{fig: I versus x_0 and T, shifted profiles with high D}
	\textbf{Left}: The optimal values of $I$ in (\ref{eq:discrete rate function}) versus $x_0$ in the case 
	that $U_T(x)=U_0(x-x_0)$ for $x_0=0,1,2,\ldots,20$ with the same setting in Fig.\ref{fig:the plot of 
	optimal paths, shifted profiles with high D}. We see that the curve is quadratic for $x_0$ small and 
	linear for $x_0$ large (see (\ref{eq:upper and lower bounds for J(A), gamma=0})). \textbf{Right}: The 
	reciprocal of the optimal values of $I$ in (\ref{eq:discrete rate function}) versus $T$ in the case that 
	$U_T(x)=U_0(x-x_0)$ for $x_0=20$ and $T=0.1,0.2,\ldots,1$ with the same setting in 
	Fig.\ref{fig:the plot of optimal paths, shifted profiles with high D}. We see that the curve is linear 
	in $T$, which means the optimal $I$ is of $\Theta(1/T)$ (see also 
	(\ref{eq:upper and lower bounds for J(A), gamma=0})).}
\end{figure}

\begin{figure}
	\includegraphics[width=0.49\textwidth]
	{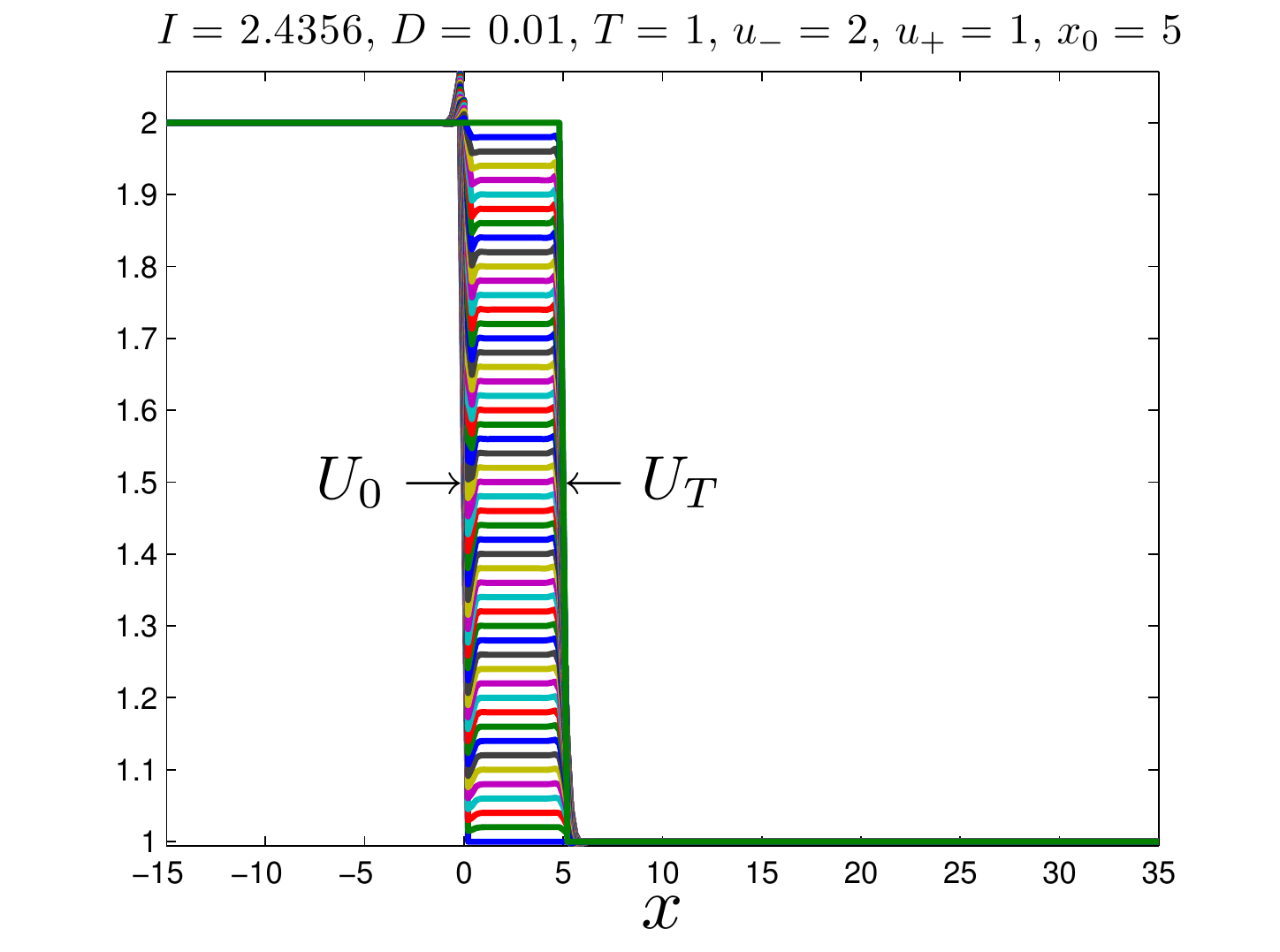}
	\includegraphics[width=0.49\textwidth]
	{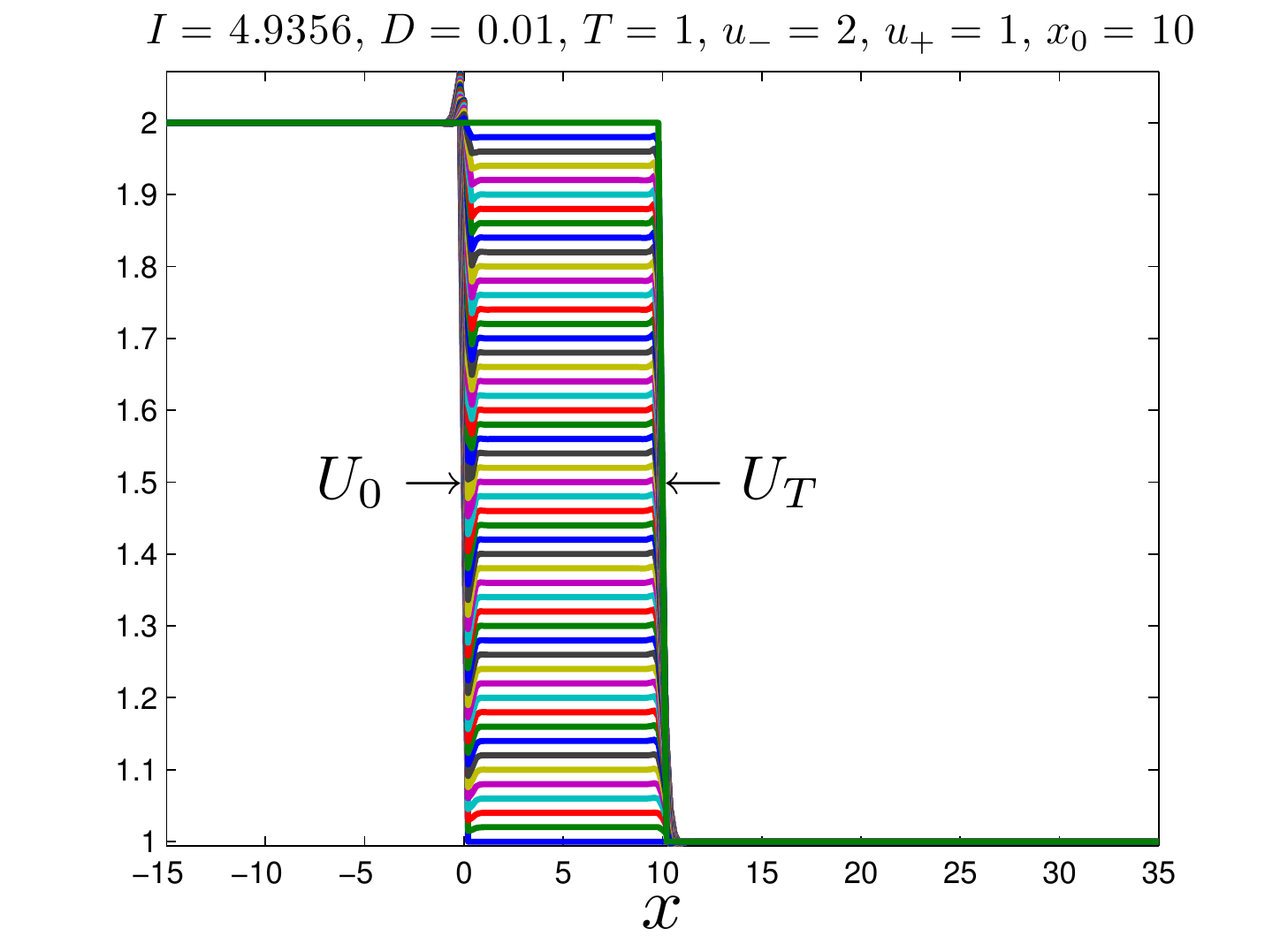}

	\includegraphics[width=0.49\textwidth]
	{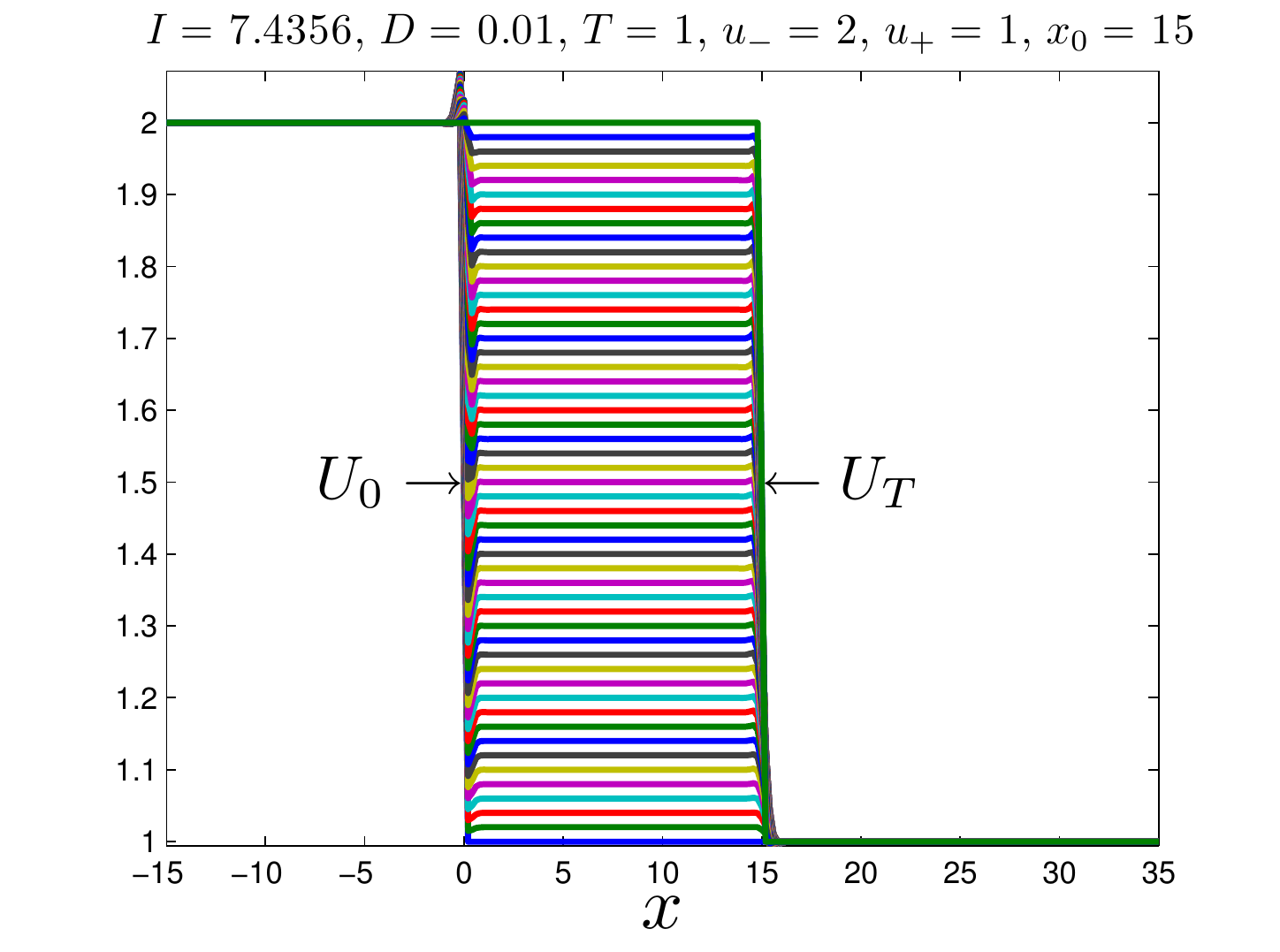}
	\includegraphics[width=0.49\textwidth]
	{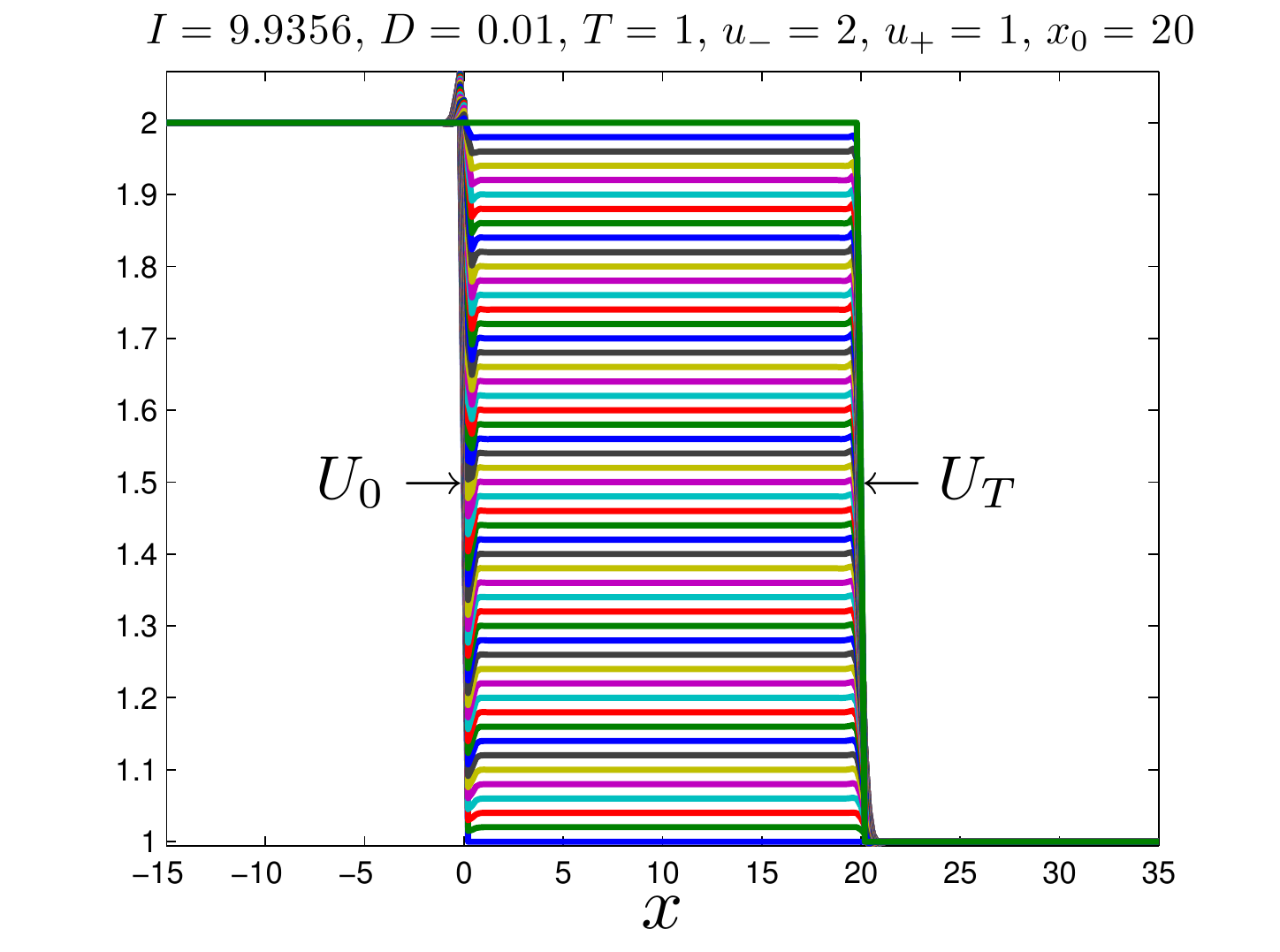}
	
	\caption{
	\label{fig:the plot of optimal paths, shifted profiles with low D}
	The optimal paths and their values of $I$ in (\ref{eq:discrete rate function}) in the case that 
	$U_T(x)=U_0(x-x_0)$ for different $x_0$. In each figure, we plot the curves indicating the optimal path 
	at time $0,\Delta t,2\Delta t,\ldots,T=1$. Different from 
	Fig.\ref{fig:the plot of optimal paths, shifted profiles with high D}, as the transition region is very 
	small, we see nearly no shifted part but only the linear interpolation.}
\end{figure}

\begin{figure}
	\includegraphics[width=0.49\textwidth]
	{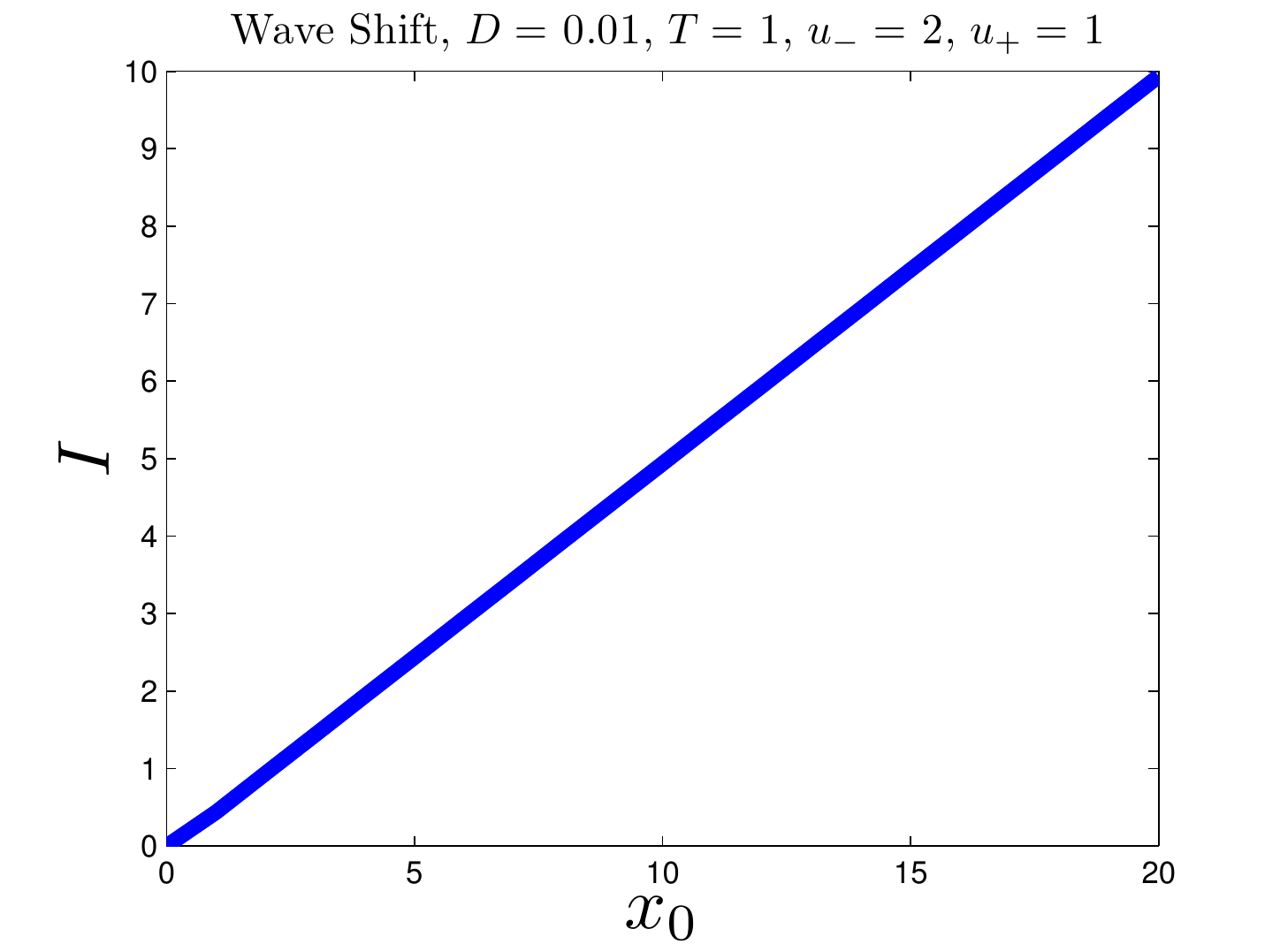}
	\includegraphics[width=0.49\textwidth]
	{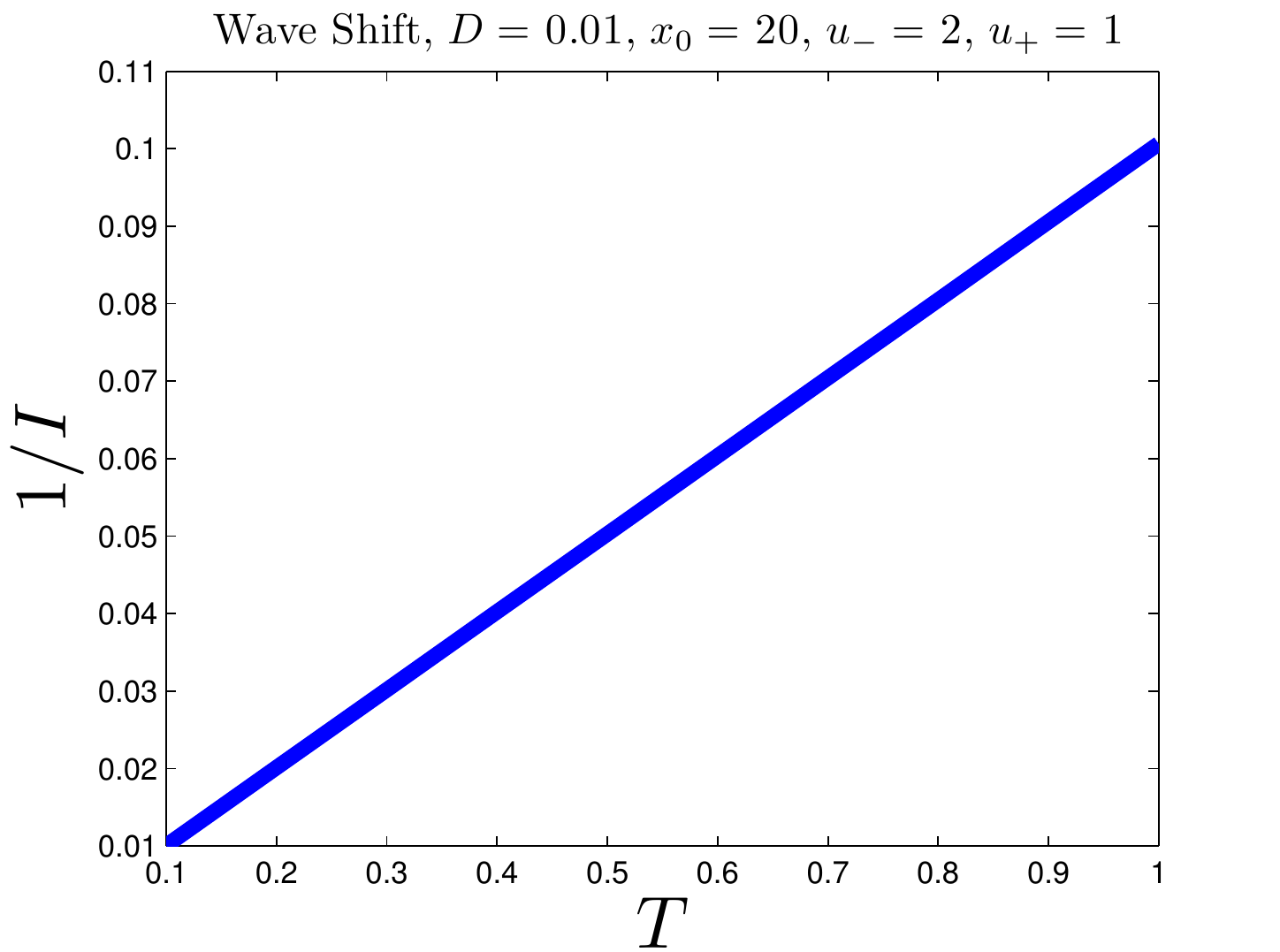}

	\caption{
	\label{fig: I versus x_0 and T, shifted profiles with low D}
	\textbf{Left}: The optimal values of $I$ in (\ref{eq:discrete rate function}) versus $x_0$ in the case 
	that $U_T(x)=U_0(x-x_0)$ for $x_0=0,1,2,\ldots,20$ with the same setting in Fig.\ref{fig:the plot of 
	optimal paths, shifted profiles with low D}. As what we see	in Fig.\ref{fig:the plot of optimal paths, 
	shifted profiles with low D}, the optimal path is the linear interpolation and thus the curves is almost 
	linear except a small perturbation around $x_0=0$. \textbf{Right}: The reciprocal of the optimal values 
	of $I$ in (\ref{eq:discrete rate function}) versus $T$ in the case that $U_T(x)=U_0(x-x_0)$ for 
	$T=0.1,0.2,\ldots,1$ with the same setting in 
	Fig.\ref{fig:the plot of optimal paths, shifted profiles with low D}. We see that the curve is linear in 
	$T$, which means the optimal $I$ is $\mathcal{O}(1/T)$.}
\end{figure}

\subsection{Change of wave speeds}

In this subsection we consider the rare event in which $u(0,x)=U_0(x)$ and $u(T,x)=U_T(x):=U_1(x)$ 
because of $\eps \dot{W}$, where $U_1(x)$ is a traveling wave with $u_{1\pm}:=\lim_{x\to \pm\infty} U_1(x)$ 
which are different from $u_{\pm}= \lim_{x\to \pm\infty}U_0(x)$, and such that 
$\gamma_1= \frac{F(u_{1+})-F(u_{1-})}{u_{1+}-u_{1-}}$ is different from 
$\gamma=  \frac{F(u_{+})-F(u_{-})}{u_{+}-u_{-}}$. This case does not belong to the class of problems 
addressed in the previous sections of this paper, in which the boundary conditions are fixed. But it is 
still possible to look for the optimal paths going from $U_0$ to $U_T$ and that minimize the rate function 
$I$.

The boundary conditions are more delicate to implement in this case. We know that $U_0$ and $U_T$ are very 
close to constants when $L$ and $R$ are far from their transition regions. In this case, the LDP implies 
that the optimal path around the boundaries should be very close to the linear interpolations of $U_0$ and 
$U_T$ in time. Therefore we let: 
\[
	q_1^n=(1-\frac{n}{N})q_1^0+\frac{n}{N}q_1^N,\quad
	q_M^n=(1-\frac{n}{N})q_M^0+\frac{n}{N}q_M^N.
\]
It is possible not to set the boundary conditions and to optimize the boundary cells as well. Our numerical 
simulations show, however, that the solution in both cases is basically the same except for some 
oscillations near the boundaries. The oscillations come from the inappropriate discretization at the 
boundaries, and this is a limitation of the numerical discretization. We can have a few extra boundary 
conditions to reduce the unwanted oscillations at the boundaries. For example, we can additionally set
\[
	q_2^n=(1-\frac{n}{N})q_2^0+\frac{n}{N}q_2^N,\quad 
	q_{M-1}^n=(1-\frac{n}{N})q_{M-1}^0+\frac{n}{N}q_{M-1}^N.
\]

The results are shown in Fig.\ref{fig:the plot of optimal paths, change of speed}. We let $\gamma=0$ and 
$\gamma_1=3.5$ while we keep $u_{-}-u_{+}=u_{1-}-u_{1+}=1$ to indicate that $U_0$ and $U_T$ have roughly the 
same transition magnitude. We see that the values of the rate function are much larger than that of the 
anomalous wave displacements. This means that it is very unlikely to have changes of wave speeds compared to 
wave displacements.

\begin{figure}
	\includegraphics[width=0.49\textwidth]
	{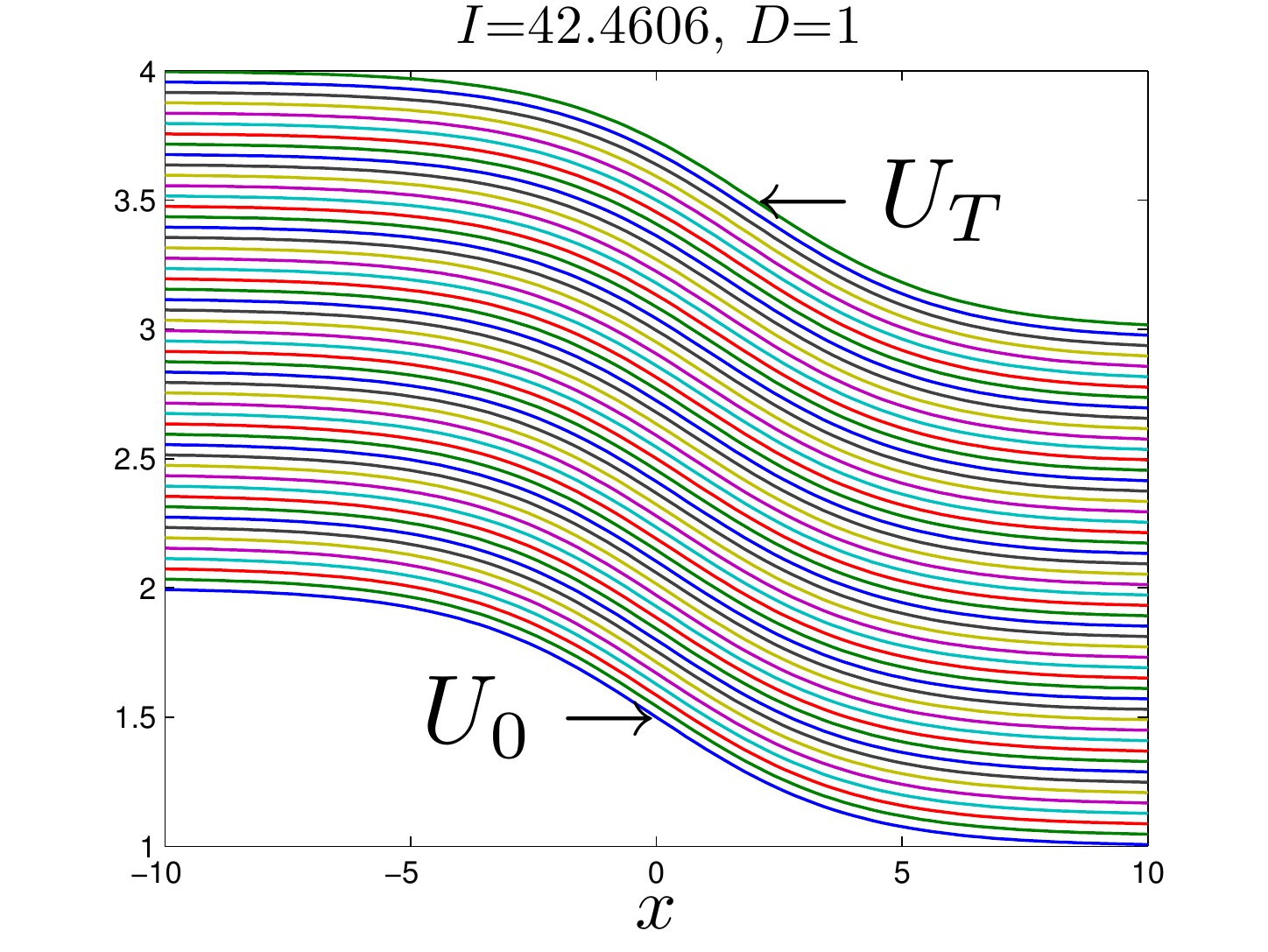}
	\includegraphics[width=0.49\textwidth]
	{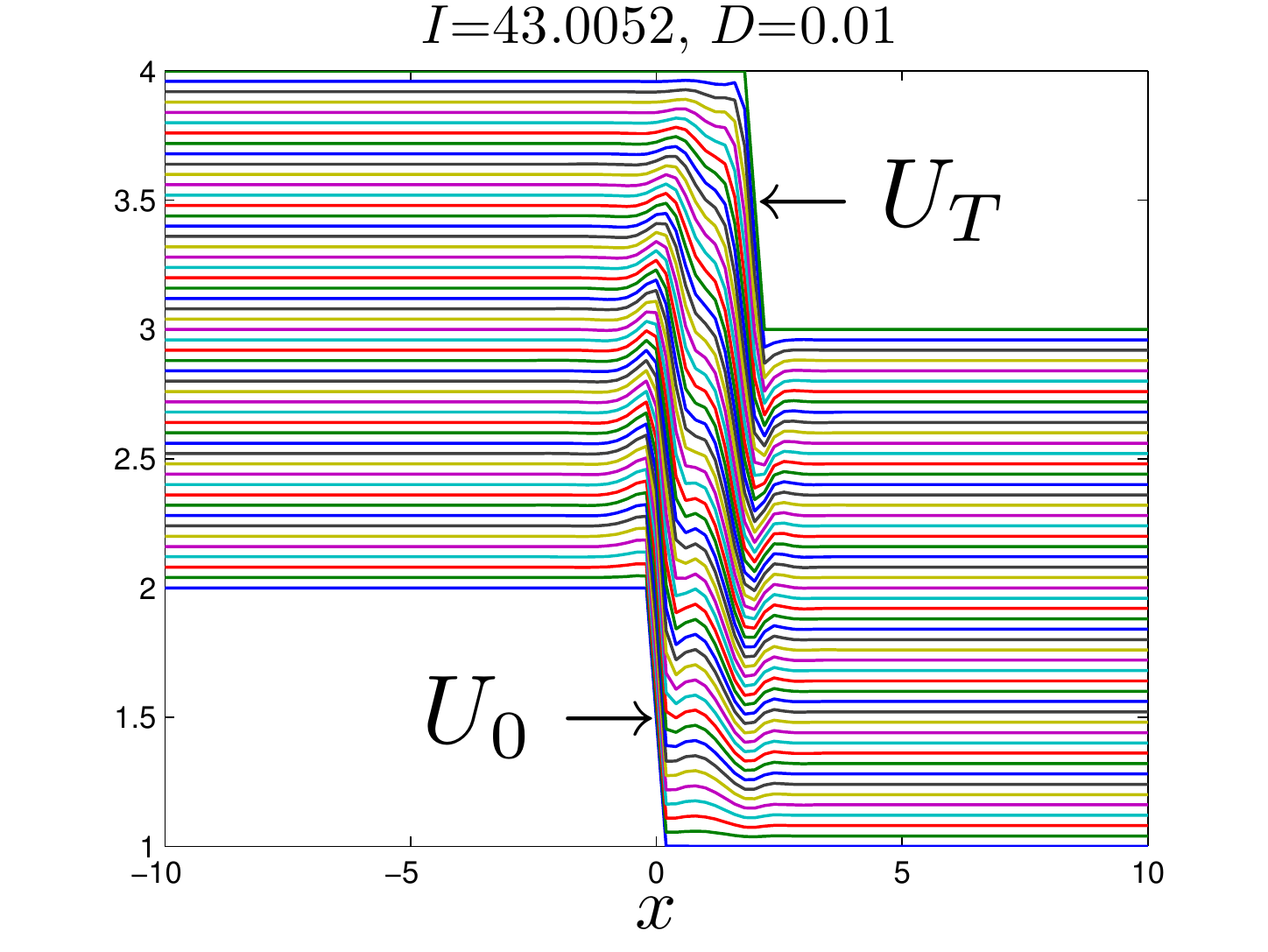}
	\caption{
	\label{fig:the plot of optimal paths, change of speed}
	The optimal paths and their values of $I$ in (\ref{eq:discrete rate function}) in the case that $U_0$ 
	and $U_T$ have different speeds. Here we let $\gamma=1.5$ in Burgers' equation 
	(\ref{eq:Burgers eqn in the moving coordinate}), and therefore the speed of $U_0$ is zero and the speed 
	of $U_T$ is $3.5$. The differences between the left and right boundary values are kept the same so that 
	$U_0$ and $U_T$ have the same transition magnitude. In each figure, we plot the curves indicating the 
	optimal path at time $0,\Delta t,2\Delta t,\ldots,T=1$.}
\end{figure}

\subsection{Weak shocks to strong shocks and strong shocks to weak shocks}

We also consider the case that $U_0$ is a weak (strong) shock while $U_T$ is a strong (weak) shock. By a 
strong shock we mean that the difference between $u_-$ and $u_+$ is large. This case is also not in the 
range of our analytical framework, but we can still compute the rate function after we impose the suitable 
boundary conditions. We use the same boundary conditions as the ones in the previous subsection:
\begin{align*}
	&q_1^n=(1-\frac{n}{N})q_1^0+\frac{n}{N}q_1^N,\quad
	q_M^n=(1-\frac{n}{N})q_M^0+\frac{n}{N}q_M^N,\\
	&q_2^n=(1-\frac{n}{N})q_2^0+\frac{n}{N}q_2^N,\quad 
	q_{M-1}^n=(1-\frac{n}{N})q_{M-1}^0+\frac{n}{N}q_{M-1}^N.
\end{align*}
From Fig.\ref{fig:the plot of optimal paths, weak to strong} and 
Fig.\ref{fig:the plot of optimal paths, strong to weak} we see that the optimal path of weak to strong and 
the one of strong to weak are significantly different, even if the reference strong and weak shocks are 
fixed. We note the very large value of the rate function compared to anomalous displacement, and how it 
depends on $D$. This confirms quantitatively the expectation that shock profiles are very stable and they 
are not easily perturbed except for displacements.

\begin{figure}
	\includegraphics[width=0.49\textwidth]
	{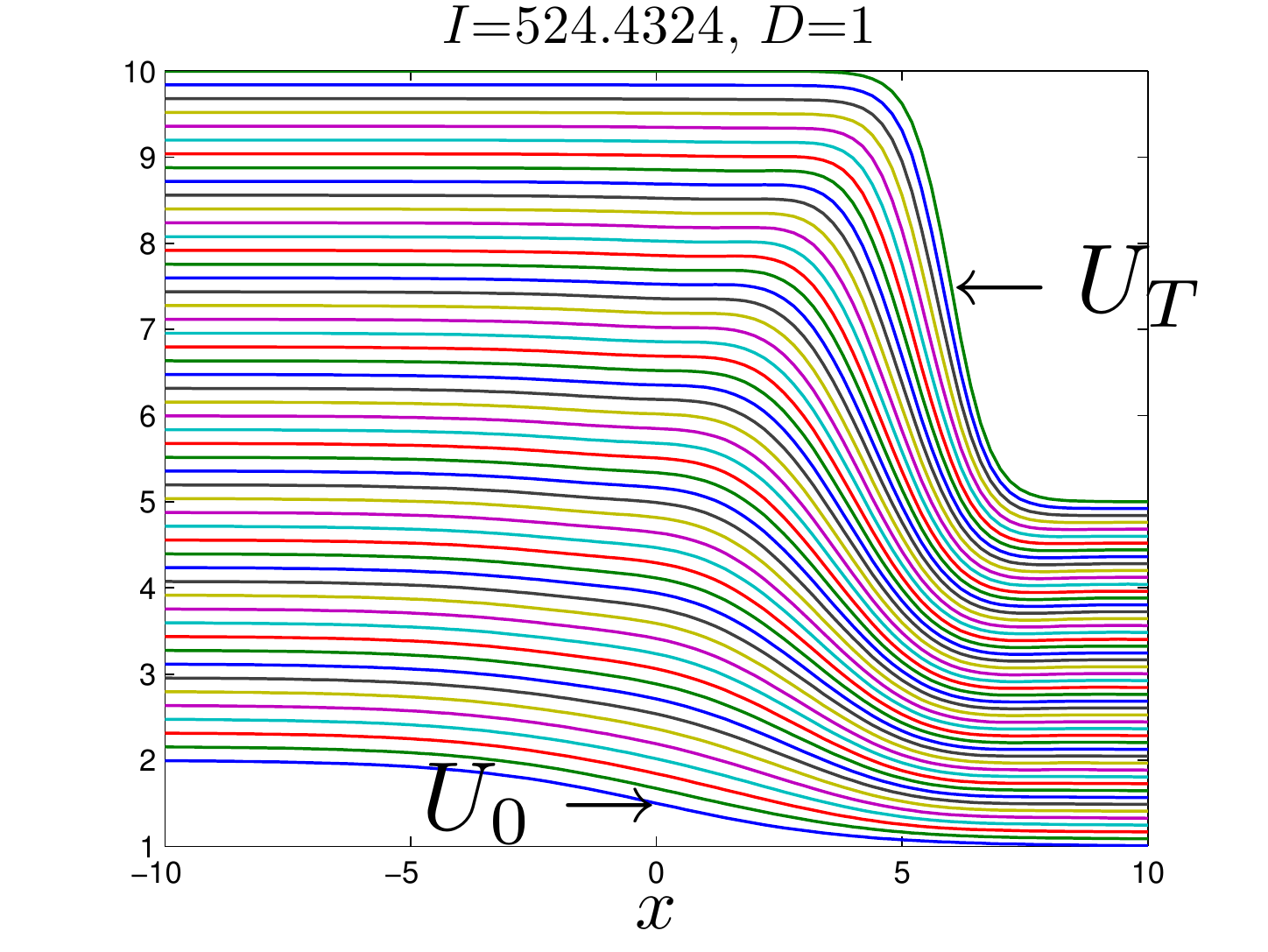}
	\includegraphics[width=0.49\textwidth]
	{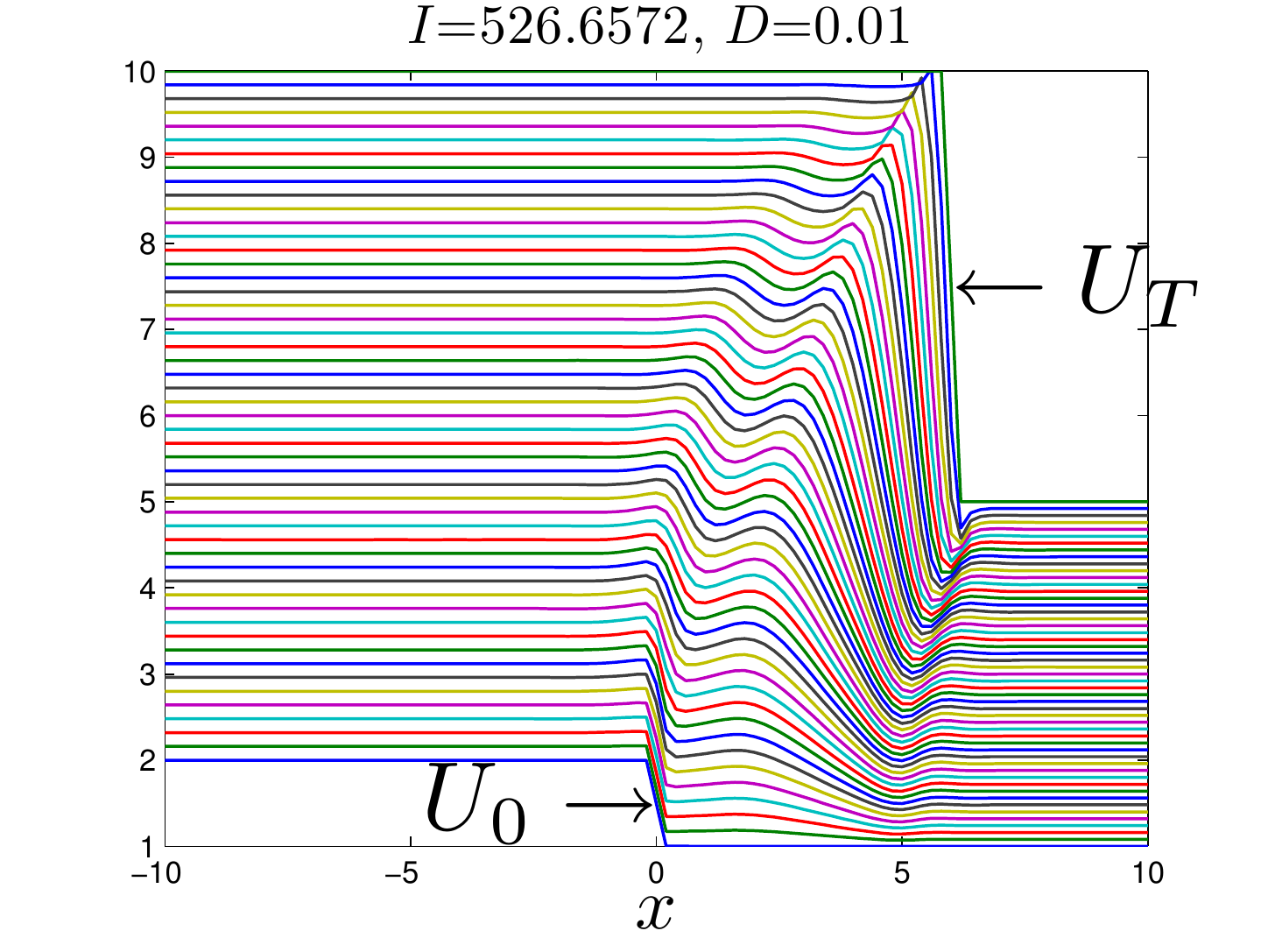}
	\caption{
	\label{fig:the plot of optimal paths, weak to strong}
	The optimal paths and their values of $I$ in (\ref{eq:discrete rate function}) in the case that $U_0$ is 
	a weak shock while $U_T$ is a strong shock. Here we let $U_0$ has a small difference in the boundary 
	values while $U_T$ has a large one. In each figure, we plot the curves indicating the optimal path at 
	time $0,\Delta t,2\Delta t,\ldots,T=1$.}
\end{figure}

\begin{figure}
	\includegraphics[width=0.49\textwidth]
	{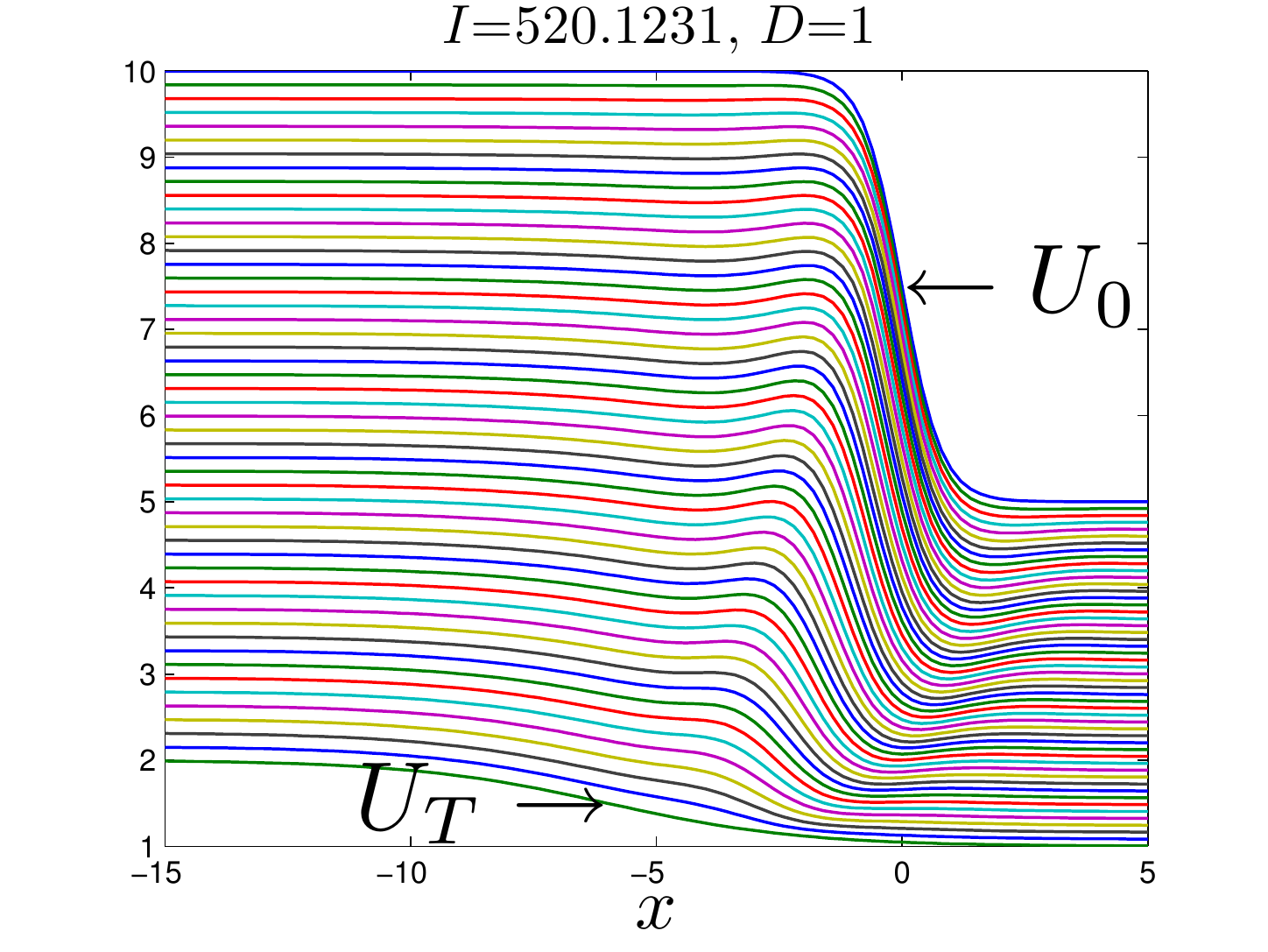}
	\includegraphics[width=0.49\textwidth]
	{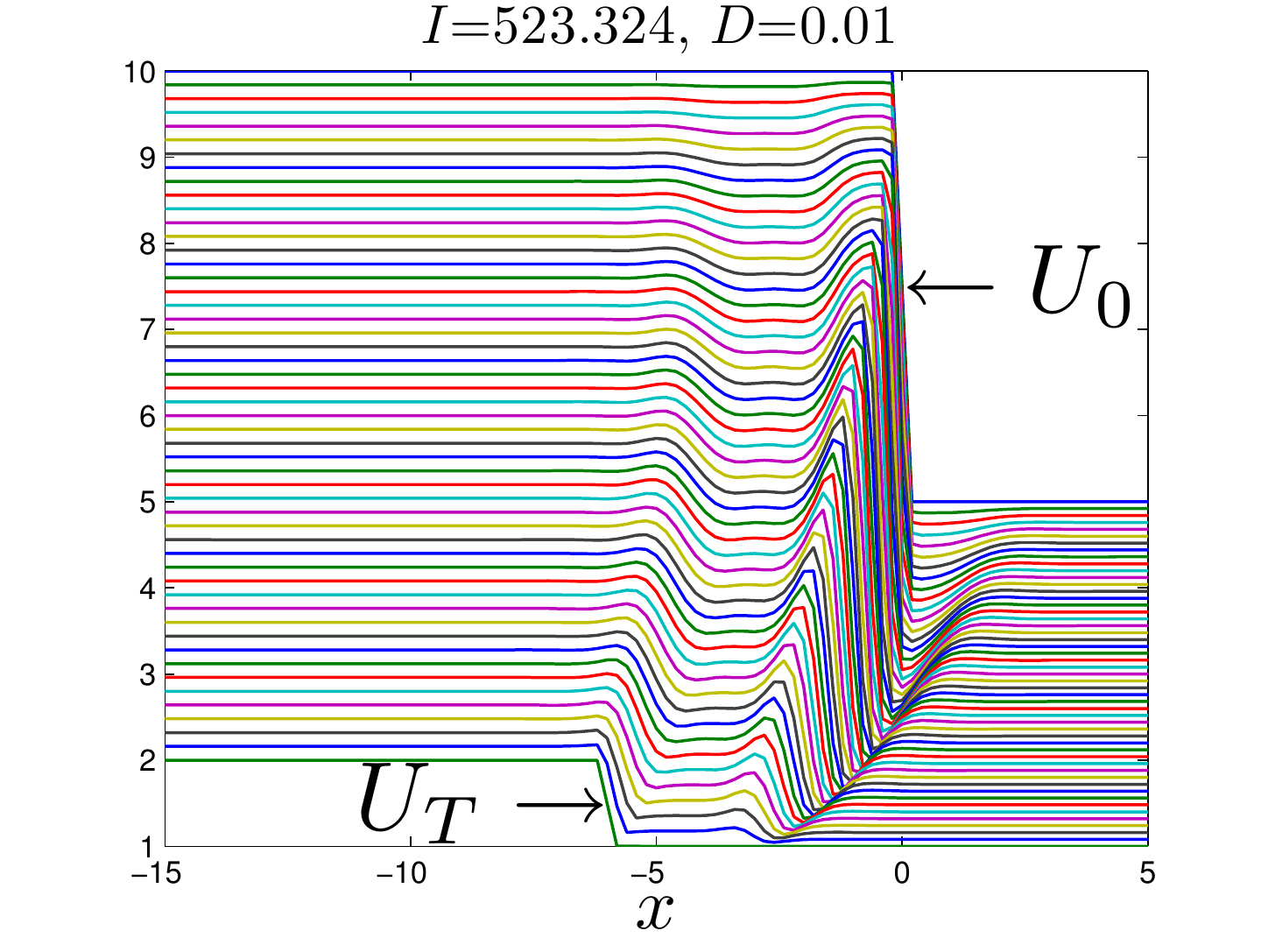}
	\caption{
	\label{fig:the plot of optimal paths, strong to weak}
	The optimal paths and their values of $I$ in (\ref{eq:discrete rate function}) in the case that $U_0$ is 
	a strong shock while $U_T$ is a weak shock. Here we let $U_0$ has a large difference in the boundary 
	values while $U_T$ has a small one. In each figure, we plot the curves indicating the optimal path at 
	time $0,\Delta t,2\Delta t,\ldots,T=1$.}
\end{figure}
\section{Direct numerical simulations with importance sampling}
\label{sec:importance sampling}

The large deviations probabilities calculated in Sections \ref{sec:displacement_general} and 
\ref{sec:discrete LDP} are only the exponential decay rates of the probabilities but not the actual 
probabilities. In this section we use Monte Carlo methods to compute the actual probabilities numerically.

\subsection{Burgers' equation with spatially correlated random perturbations}

We reformulate the discretized problem for Burgers' equation when we have spatially correlated random
perturbations. Given a traveling wave solution $U_0(x-\gamma t)$ of Burgers' equation with 
$\lim_{x\to\pm\infty}U_0(x)=u_\pm$, we transform to moving coordinates
\begin{equation}
	\label{eq:Burgers eqn in the moving coordinate for Monte Carlo}
	u_t + \Big(\frac{1}{2}(u-\gamma)^2\Big)_x = (Du_x)_x + \eps\dot{W}.
\end{equation}
Then $U_0(x)$ is a stationary traveling wave of 
(\ref{eq:Burgers eqn in the moving coordinate for Monte Carlo}). The rare event we consider is 
\[
	A_\delta = \{u\in {\cal E}^1  \mbox{ such that } u(0,\cdot)=U_0,\, \|u(T,\cdot)-U_0(\cdot-x_0)\|_{L^2}\leq \delta\}.
\]
Although for a discrete conservation law it is also possible to consider the other cases in Section 
\ref{sec:discrete LDP}, their probabilities are too small to compute by the basic Monte Carlo method so we 
omit them.

To compute $\mathbb{P}(u\in A_\delta)$ numerically, 
we discretize the space and time domains uniformly as in Section \ref{sec:LD for Euler}: 
$L=x_0<\cdots<x_M=R$, $M\Delta x=R-L$ and $0=t_0<\cdots<t_N=T$, $N\Delta t=T$. Here $Q_m^n$ denotes the average 
of $u$ over the $m$-th cell at time $n\Delta t$, and evolves by the Euler method:
\begin{equation}
	\label{eq:Euler method for Monte Carlo}
	Q_m^{n+1} 
	= Q_m^n - \frac{\Delta t}{\Delta x}(F_{m+\frac{1}{2}}^n-F_{m-\frac{1}{2}}^n)
	+ D\frac{\Delta t}{\left(\Delta x\right)^2}(Q_{m+1}^n-2Q_m^n+Q_{m-1}^n)
	+ \eps\Delta W_m^n,
\end{equation}
for $m=2,\ldots,M-1$, $ n=1,\ldots,N$,
where $F_{m\pm 1/2}^n(Q^n)$ are numerical fluxes of $(u-\gamma)^2/2$ constructed by Godunov's 
method and $(\Delta W_m^n)_{m=2,\ldots,M-1,n=1,\ldots,N}$ are Gaussian random variables with mean $0$ and covariance 
\[
	\mathbb{E}(\Delta W_{m_1}^{n_1}\Delta W_{m_2}^{n_2})
	=\begin{cases}
		\frac{\Delta t}{\Delta x}C_{m_1 m_2},	\quad &n_1=n_2,\\
		0,										\quad &\text{otherwise.}
	\end{cases}
\]
In order to make the problem more realistic, we assume that the variables $\Delta W_m^n$ are spatially correlated: 
$C_{m_1 m_2}=\sigma^2\exp(-\frac{1}{l_c}|x_{m_1}-x_{m_2}|)$ and $C=(C_{ij})_{i,j=2}^{M-1}=\Phi\Phi^T$. 
Finally we impose the initial and boundary conditions: $Q_m^0=U_0(x_{m-1/2})$,  $Q_1^n=u_-$ and $Q_M^n=u_+$.

\subsection{Introduction to importance sampling}

To estimate $\PP(u \in A_\delta)$, we may use the basic Monte Carlo method. The Monte Carlo 
 strategy is as follows. We generate $K$ independent samples 
$\Delta W^{(k)} = (\Delta W_m^{n,(k)})_{m=2,\ldots,M-1,n=1,\ldots,N}$, $k=1,\ldots,K$,
of the Gaussian vector $(\Delta W_m^n)_{m=2,\ldots,M-1,n=1,\ldots,N}$, which give $K$ independent samples 
$Q^{(k)} = (Q_m^{n,(k)})_{m=1,\ldots,M,n=1,\ldots,N}$, $k=1,\ldots,K$ of the random vector $Q= (Q_m^{n})_{m=1,\ldots,M,n=1,\ldots,N}$. The basic Monte Carlo estimator is 
\begin{equation}
	\label{eq:plain Monte Carlo estimator}
	\hat{P}^{MC}=\frac{1}{K}\sum_{k=1}^K 1_{A_\delta}(Q^{(k)})
\end{equation}
where 
\begin{equation}
Q \in A_\delta \mbox{  if and only if }
\Delta x\sum_{m=1}^M [Q_m^{N} - U_0(x_{m-\frac{1}{2}}-x_0)]^2 \leq \delta^2.
\end{equation} 
In other words, 
$\hat{P}^{MC}$ is the empirical frequency that $Q^{(k)}\in A_\delta$. It is an unbiased estimator
$\EE [ \hat{P}^{MC} ] =\PP( Q\in A_\delta)$. By the law of large numbers, it is strongly convergent 
$\hat{P}^{MC}\to \PP(Q \in A_\delta)$ almost surely as $K\to\infty$.
Its variance is given by
\begin{align*}
 \mathbf{Var} (\hat{P}^{MC} ) = \frac{1}{K}  \mathbf{Var} ( 1_{A_\delta}(Q) )= 
 \frac{1}{K} \big (\PP(Q\in A_\delta)-\PP(Q\in A_\delta)^2 \big).
\end{align*}
In order to have a meaningful estimation, the standard deviation 
of the estimator and $\PP(Q \in A_\delta)$  should 
be of the same order. Namely, the relative error
\[
\frac{ \mathbf{Var}^{\frac{1}{2}}  (\hat{P}^{MC} )}{\PP(Q \in A_\delta)}= 
 \frac{1}{\sqrt{K}}\left(\frac{1}{\PP(Q\in A_\delta)}-1\right)^\frac{1}{2}
\]
should be of order one (or smaller).
This means that the number $K$ of Monte Carlo samples should be at least of the order of the reciprocal
of the probability $\PP(Q \in A_\delta)$.
We note that for $\eps$ small, $\PP(Q\in A_\delta)$ decreases exponentially and so $K$ should be 
increased exponentially; the exponential growth of $K$ makes the basic Monte Carlo method computationally 
infeasible.

The well established way to overcome the difficulty of calculating rare event probabilities is to use 
importance sampling. The problem with basic Monte Carlo is that for small $\eps$ there are very few samples 
in $A_\delta$ under the original measure $\PP$ so the estimator is inaccurate. In importance sampling we 
change the original measure so that there is a significant fraction of $Q^{k}$ in $A_\delta$ under this 
new measure $\QQ$, even for small $\eps$. Since we use the biased measure $\QQ$ to generate $Q^{(k)}$, 
it is necessary to weight the simulation outputs in order to get an unbiased estimator of $\PP(Q \in A_\delta)$.
The correct weight is the  likelihood ratio since we have:
\[
	\PP(Q\in A_\delta)
	=\EE_\PP[1_{A_\delta}(Q)]
	=\EE_\QQ\big[1_{A_\delta}(Q)\frac{d\PP}{d\QQ}(Q)\big].
\]
Then the importance sampling estimator is
\begin{equation}
	\label{eq:general importance sampling estimator}
	\hat{P}^{IS}=\frac{1}{K}\sum_{k=1}^K 1_{A_\delta}(Q^{(k)})\frac{d\PP}{d\QQ}(Q^{(k)}),
\end{equation}
where $Q^{(k)}$ are generated under $\QQ$. The estimator $\hat{P}^{IS}$ is unbiased 
$\EE_\QQ[\hat{P}^{IS}]=\PP(Q\in A_\delta)$ and its 
variance is
\begin{equation*}
\mathbf{Var}_\QQ (\hat{P}^{IS})
=\frac{1}{K}\mathbf{Var}_\QQ \big(1_{A_\delta}(Q)\frac{d\PP}{d\QQ}(Q)\big)
\end{equation*}
The main issue in importance sampling is how to choose a good $\QQ$ to have a low 
$\mathbf{Var}_\QQ[1_{A_\delta}(Q )\frac{d\PP}{d\QQ}(Q )]$. In many cases (see for example, 
\cite{Siegmund1976,Sadowsky1990,Chen1993,Sadowsky1996}), it can be shown that the change of measure 
suggested by the most probable path of the LDP is asymptotically optimal as $\eps\to 0$. However, it is also 
well-known that in some cases (see \cite{Glasserman1997}), the estimator by this strategy is worse than the 
basic Monte Carlo, and may even have infinite variance. However, because the rare event $A_\delta$ is convex and 
the discrete rate function $I$ is (numerically tested) essentially convex, the importance sampling estimator 
by this strategy is expected to be asymptotically optimal and we will see that it indeed works very well.

\subsection{Importance sampling based on the most probable path}

In this subsection we implement the importance sampling by using a biased distribution
centered on the most probable path obtained in Section 
\ref{sec:discrete LDP}. From Section \ref{sec:LD for Euler}, $\bar{Q} :=\arg\inf_{Q \in A_\delta}I(Q )$ 
is the most probable path as $\eps\to 0$ by the large deviation principle. We choose $\bar{h}=(\bar{h}^n)_{n=1,\ldots,N} = (\bar{h}^n_m)_{m=2,\ldots,M-1,n=1,\ldots,N} $ 
such that
\begin{align}
	\label{eq:residual of the optimal path}
	\bar{Q}_m^{n+1} 
	&= \bar{Q}_m^n - \frac{\Delta t}{\Delta x}(F_{m+\frac{1}{2}}(\bar{Q}^n)-F_{m-\frac{1}{2}}(\bar{Q}^n))\\
	&\quad + D\frac{\Delta t}{\left(\Delta x\right)^2}(\bar{Q}_{m+1}^n-2\bar{Q}_m^n+\bar{Q}_{m-1}^n) 
	+ (\Phi \bar{h}^n)_m, \notag
\end{align}
for $m=2,\ldots,M-1$, $n=1,\ldots,N$.

Assume that under the probability $\QQ$, the vector $\Delta W^n:=(\Delta W^n_m)_{m=2,\ldots,M-1}$ 
in (\ref{eq:Euler method for Monte Carlo}) is multivariate Gaussian 
$\mathcal{N}(\eps^{-1}\Phi \bar{h}^n, \frac{\Delta t}{\Delta x}C)$ and 
$\mathbf{Cov}(\Delta W_{m_1}^{n_1},\Delta W_{m_2}^{n_2})=0$ if $n_1\neq n_2$. 
Denoting $\Delta W= (\Delta W^n)_{n=1,\ldots,N}$, the likelihood ratio 
$d\PP/d\QQ$ can be computed explicitly:
\begin{equation}
	\label{eq:change of measure}
	\frac{d\PP}{d\QQ}(\Delta W ) = \exp\left( -\frac{1}{2} \sum_{n=1}^N 
	[\|\Phi^{-1}\Delta W^n\|_2^2 - \|\Phi^{-1}\Delta W^n - \eps^{-1}\bar{h}^n\|_2^2  \right) .
\end{equation}
Note that under $\QQ$, 
(\ref{eq:Euler method for Monte Carlo}) can be written as
\begin{equation}
	\label{eq:Euler scheme under the biased measure}
	Q_m^{n+1} 
	= Q_m^n - \frac{\Delta t}{\Delta x}(F_{m+\frac{1}{2}}^n-F_{m-\frac{1}{2}}^n)
	+ D\frac{\Delta t}{\left(\Delta x\right)^2}(Q_{m+1}^n-2Q_m^n+Q_{m-1}^n)
	+ (\Phi \bar{h}^n)_m + \eps\Delta \tilde{W}_m^n,
\end{equation}
where $\Delta \tilde{W}= (\Delta \tilde{W}_m^n)_{m=2,\ldots,M-1,n=1,\ldots,N}$ is zero-mean, Gaussian with the spatial covariance $\frac{\Delta t}{\Delta x}C$, and white in 
time. Then (\ref{eq:change of measure}) can be written as 
\begin{equation}
	\label{eq:change of measure under the biased measure}
	\frac{d\PP}{d\QQ}(\Delta \tilde{W} ) = \exp\left( -\frac{1}{2} \sum_{n=1}^N 
	 \|\Phi^{-1}\Delta \tilde{W} ^n + \eps^{-1}\bar{h}^n\|_2^2 - \|\Phi^{-1}\Delta \tilde{W} ^n\|_2^2  \right).
\end{equation}
In summary, the importance sampling Monte Carlo strategy  is implemented as follows:
\begin{enumerate}
	\item Compute the optimal path $\bar{Q}=\arg\inf_{Q \in A_\delta}I(Q)$ and its residual 
	$\bar{h}$ from (\ref{eq:residual of the optimal path}).
	
	\item Sample $K$  independent $\Delta \tilde{W}^{(k)}$ with the zero-mean, Gaussian distribution 
	with the covariance $\frac{\Delta t}{\Delta x}C$ in space and white in time. Compute the corresponding $Q^{(k)}$ by 
	(\ref{eq:Euler scheme under the biased measure}).
	
	\item The importance sampling estimator is 
	\[
		\hat{P}^{IS}=\frac{1}{K}\sum_{k=1}^K 1_{A_\delta}(Q^{(k)})
		\frac{d\PP}{d\QQ}(\Delta \tilde{W}^{(k)}),
	\]
	where $\frac{d\PP}{d\QQ}(\Delta \tilde{W})$ is defined in 
	(\ref{eq:change of measure under the biased measure}).
\end{enumerate}

\subsection{Simulations with importance sampling}

We consider three estimators: the basic Monte estimator $\hat{P}^{MC}$ and two importance sampling 
estimators: $\hat{P}^{IS}_0$ and $\hat{P}^{IS}_\delta$, where $\hat{P}^{IS}_0$ uses $\bar{h}$ in 
(\ref{eq:residual of the optimal path}) with $\bar{Q} =\arg\inf_{Q \in A_0}I(Q_m^n)$ while 
$\hat{P}^{IS}_\delta$ uses $\bar{h}$ in (\ref{eq:residual of the optimal path}) with 
$\bar{Q} =\arg\inf_{Q \in A_\delta}I(Q )$. The parameters for the simulations are listed in Table 
\ref{tab:parameter of Monte Carlo}.
\begin{table}
	\begin{center}
		\begin{tabular}{|c|c|c|c|c|c|c|}
			\hline
			$\Delta x=0.5$ & $\Delta t=0.05$ & $L=-15$ & $R=20$ & $T=1$ & $u_{-}=2$ & $u_{+}=1$\\
			\hline
			$\gamma=1.5$ & $D=1$ & $x_{0}=5$ & $\delta=\sqrt{0.5}$ & $\sigma=1$ & $l_{c}=5$ & $K=10^{4}$\\
			\hline
		\end{tabular}
	\end{center}
	\caption{
	\label{tab:parameter of Monte Carlo}
	The parameters for the Monte Carlo simulations.}
\end{table}

As we mention before, we only test the wave displacement with $x_0=5$ and $D=1$ because the probabilities of 
the other cases in Section \ref{sec:discrete LDP} are too small for $\hat{P}^{MC}$ to have meaningful 
samples.

First we find the optimal paths for $\inf_{Q \in A_0}I(Q )$ and $\inf_{Q \in A_\delta}I(Q )$. 
As before, $\inf_{Q \in A_0}I(Q )$ can be modeled as an unconstrained optimization problem and we 
solve it by the BFGS method while we solve $\inf_{Q \in A_\delta}I(Q )$ by  
sequential quadratic programming (SQP) \cite{Nocedal2006}.

From Fig.\ref{fig:optimal paths for inf_A I and inf_A_delta I} we note that $\inf_{Q \in A_0}I(Q )$ is 
much smaller than the corresponding one with spatially white noise. This is because with the correlated 
noise, it is easier to have simultaneous increments. In addition, $\inf_{Q \in A_\delta}I(Q )$ is 
significantly different from $\inf_{Q \in A_0}I(Q)$ because $\delta=\sqrt{0.5}$ is not very small. 
we will see that this difference significantly affects the performances of $\hat{P}^{IS}_0$ and 
$\hat{P}^{IS}_\delta$.

\begin{figure}
	\centering
	\includegraphics[width=0.49\textwidth]{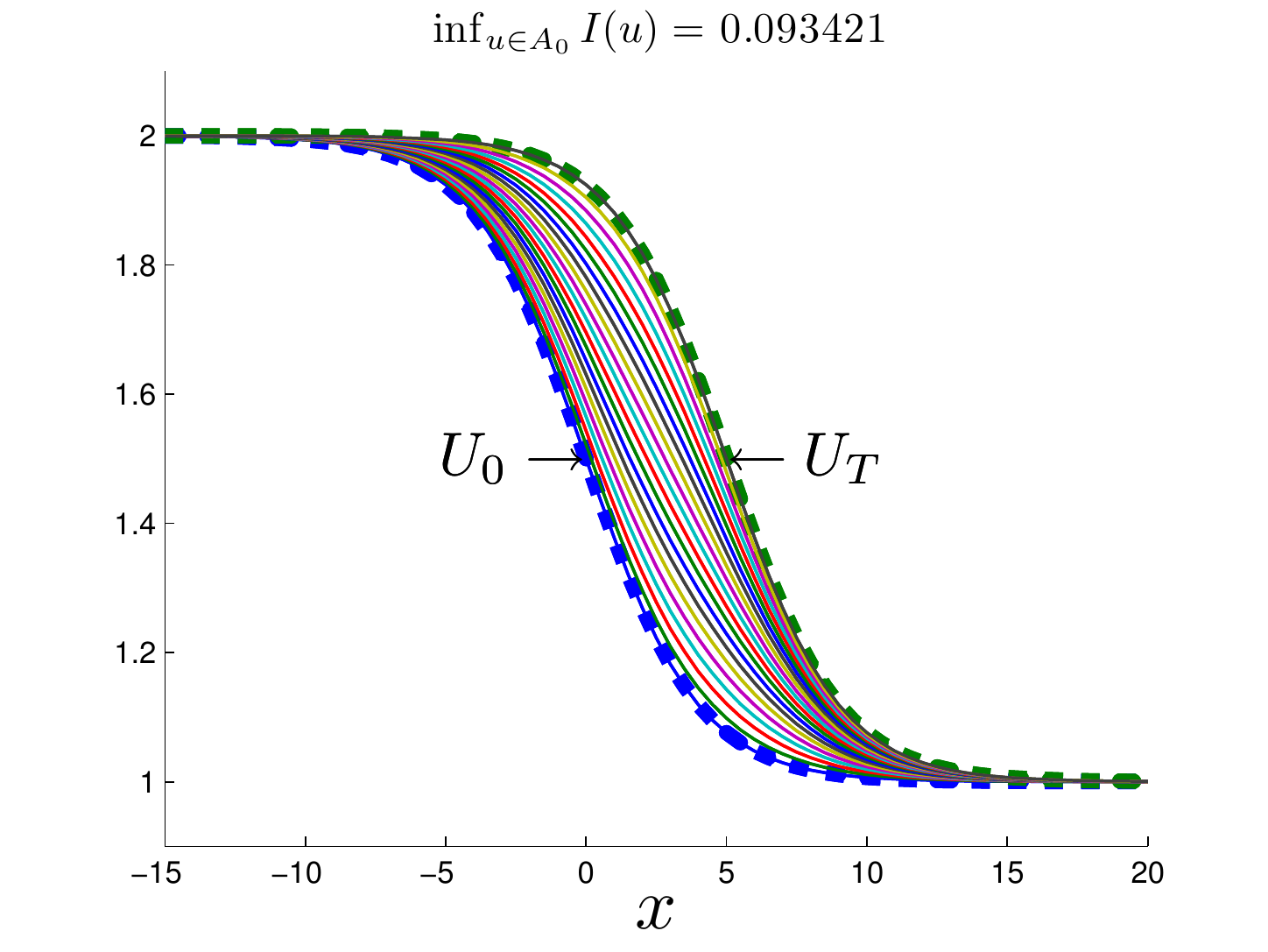}
	\includegraphics[width=0.49\textwidth]{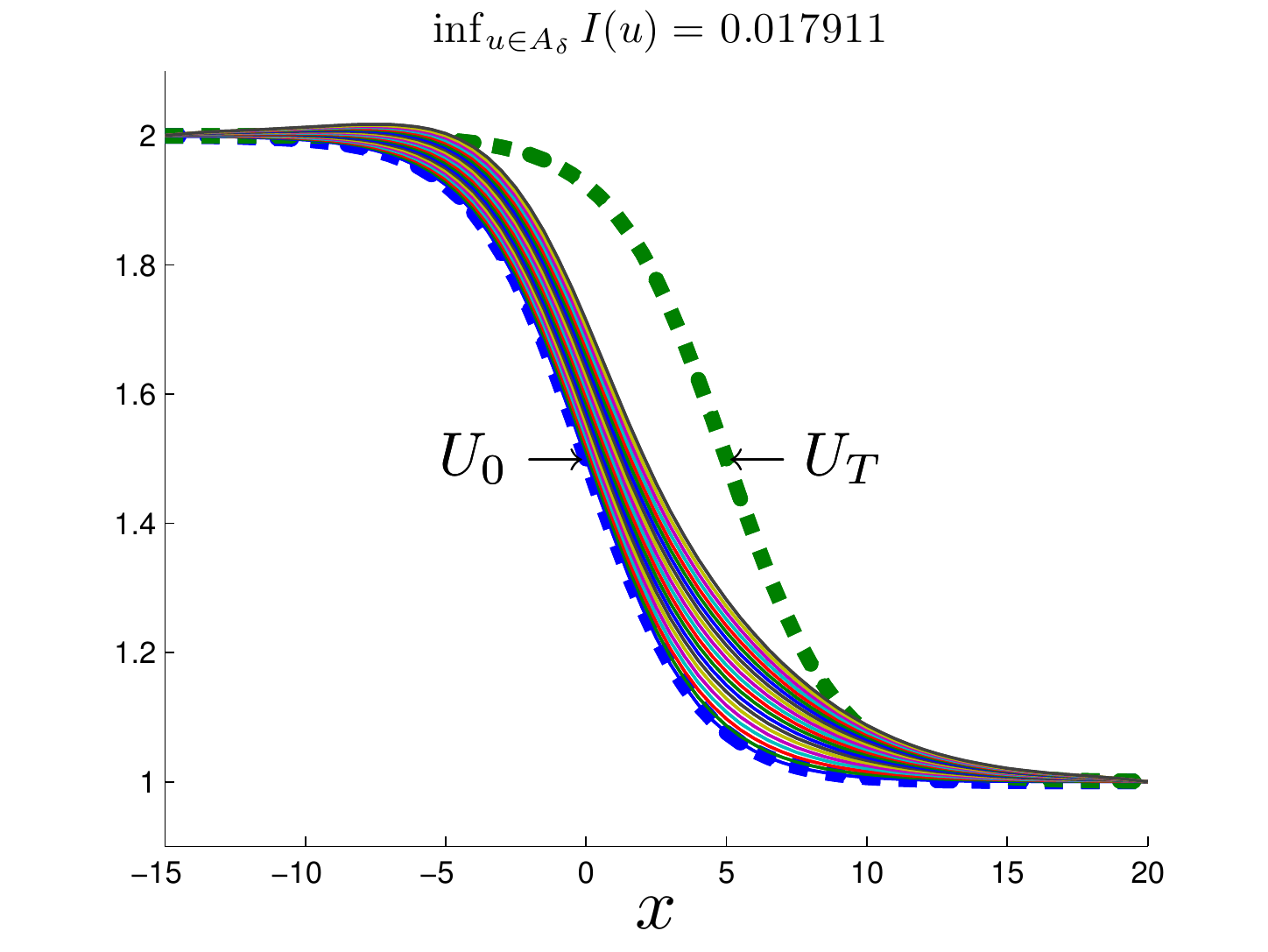}
	\caption{
	\label{fig:optimal paths for inf_A I and inf_A_delta I}
	The optimal paths for $\inf_{Q\in A_0}I(Q)$ and $\inf_{Q \in A_\delta}I(Q)$. With the 
	spatially correlated noise, $\inf_{Q\in A_0}I(Q)$ is much lower than the one with the spatially  
	white noise (see Fig.\ref{fig:the plot of optimal paths, shifted profiles with high D}).}
\end{figure}

Once $\arg\inf_{Q \in A_0}I(Q)$ and $\arg\inf_{Q \in A_\delta}I(Q)$ are obtained, we can 
construct $\hat{P}^{IS}_0$ and $\hat{P}^{IS}_\delta$. We estimate $\PP(Q\in A_\delta)$
 by these three estimators for $100$ different $\eps$ taken uniformly in $[0.01, 0.2]$. For each $\eps$ 
and estimator, we use $K=10^4$ samples. 
The (numerical) $99\%$ confidence 
intervals are defined by
$
\big[ {\rm Mean}_K -2.6\, {\rm Std}_K, {\rm Mean}_K +2.6\, {\rm Std}_K\big]
$
where the (numerical) mean and standard deviation are 
$$
{\rm Mean}_K =\frac{1}{K} \sum_{k=1}^K p^{(k)},Ê\quad \quad
 {\rm Std}_K^2 =\frac{1}{K} \sum_{k=1}^K (p^{(k)})^2 - {\rm Mean}_K^2,
$$
with $p^{(k)}={\bf 1}_{A_\delta}(Q^{(k)})$ for $\hat{P}^{MC}$ and 
$p^{(k)}=d\PP/d\QQ(\Delta \tilde{W}^{(k)}){\bf 1}_{A_\delta}(Q^{(k)})$ for $\hat{P}^{IS}$.
We find that $\hat{P}^{IS}_\delta$ has the best performance and 
$\hat{P}^{IS}_0$ also has the good performance for $0.1\leq\eps\leq 0.2$. 
For $\epsilon < 0.1$, because $\arg\inf_{Q \in A_0}I(Q )$ is not the optimal path, $\hat{P}^{IS}_0$ is even worse 
than $\hat{P}^{MC}$ due to the inappropriate change of measure.
For $\epsilon < 0.1$, because $\arg\inf_{Q \in A_\delta}I(Q )$ is the optimal path, $\hat{P}^{IS}_\delta$ 
dramatically outperforms $\hat{P}^{MC}$.
This shows that large-deviations-driven importance sampling strategies can be efficient
to estimate rare event probabilities in the context of perturbed scalar conservation laws.

\begin{figure}
	\centering
	\includegraphics[width=0.32\textwidth]{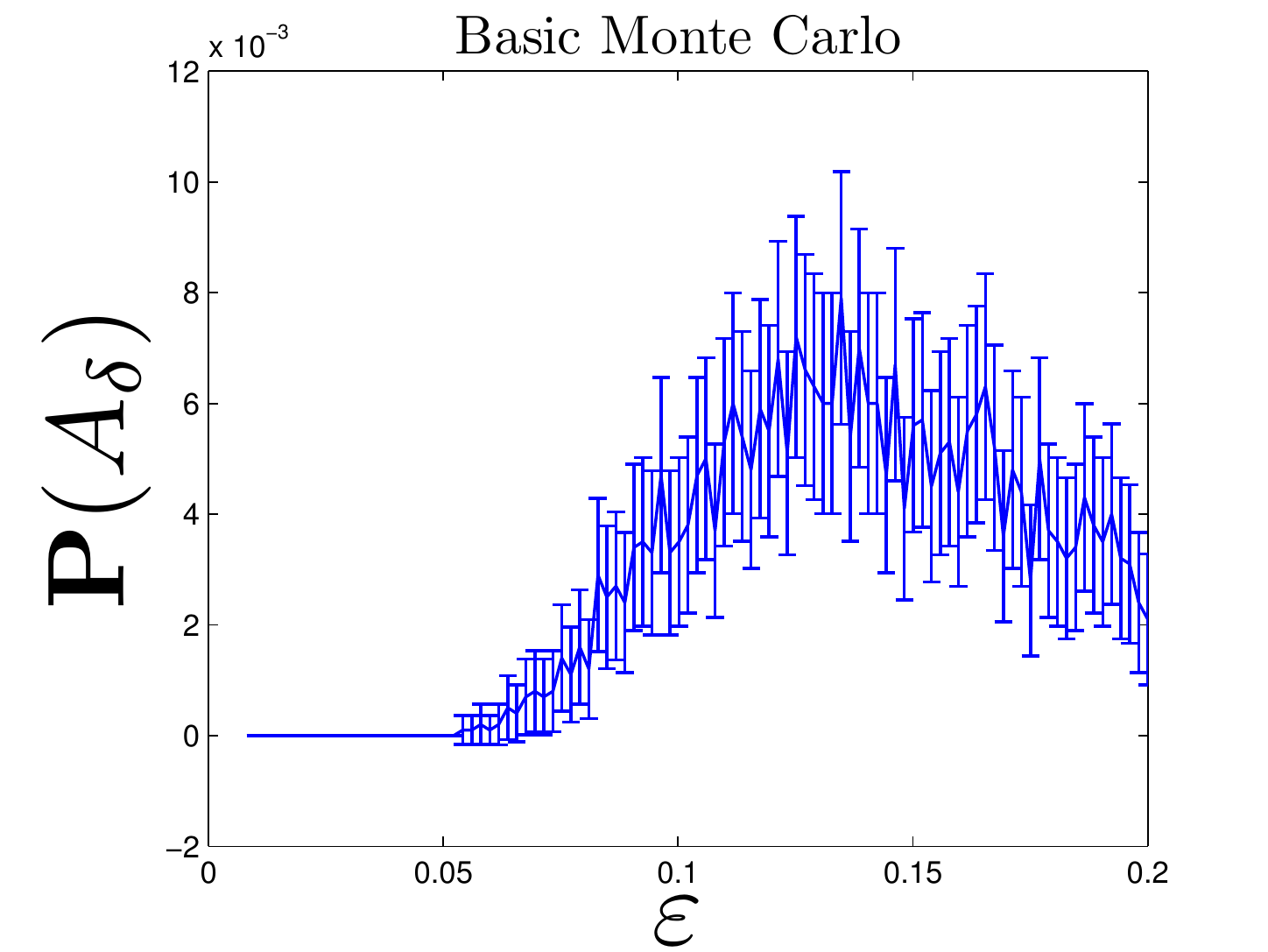}
	\includegraphics[width=0.32\textwidth]{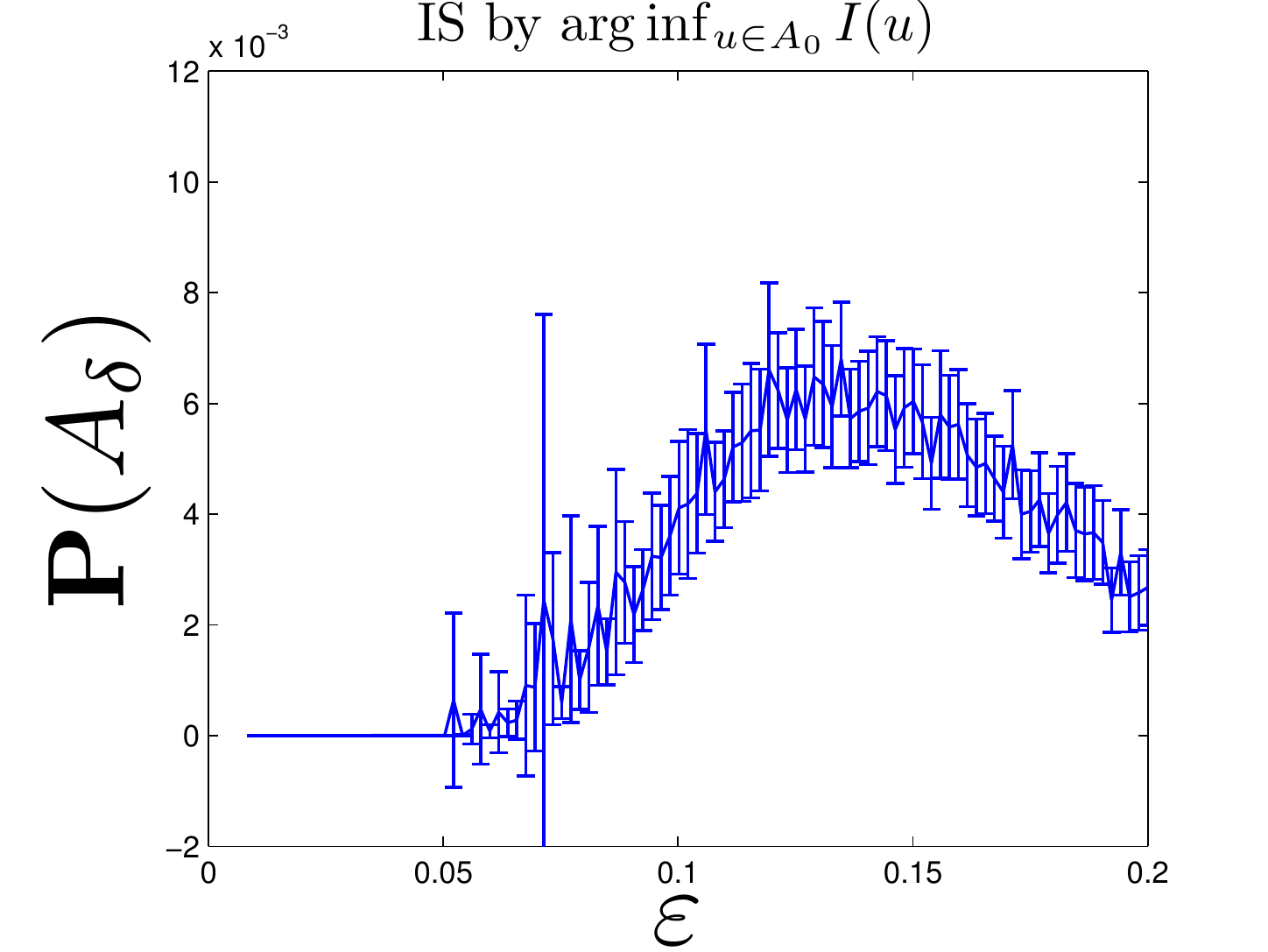}
	\includegraphics[width=0.32\textwidth]{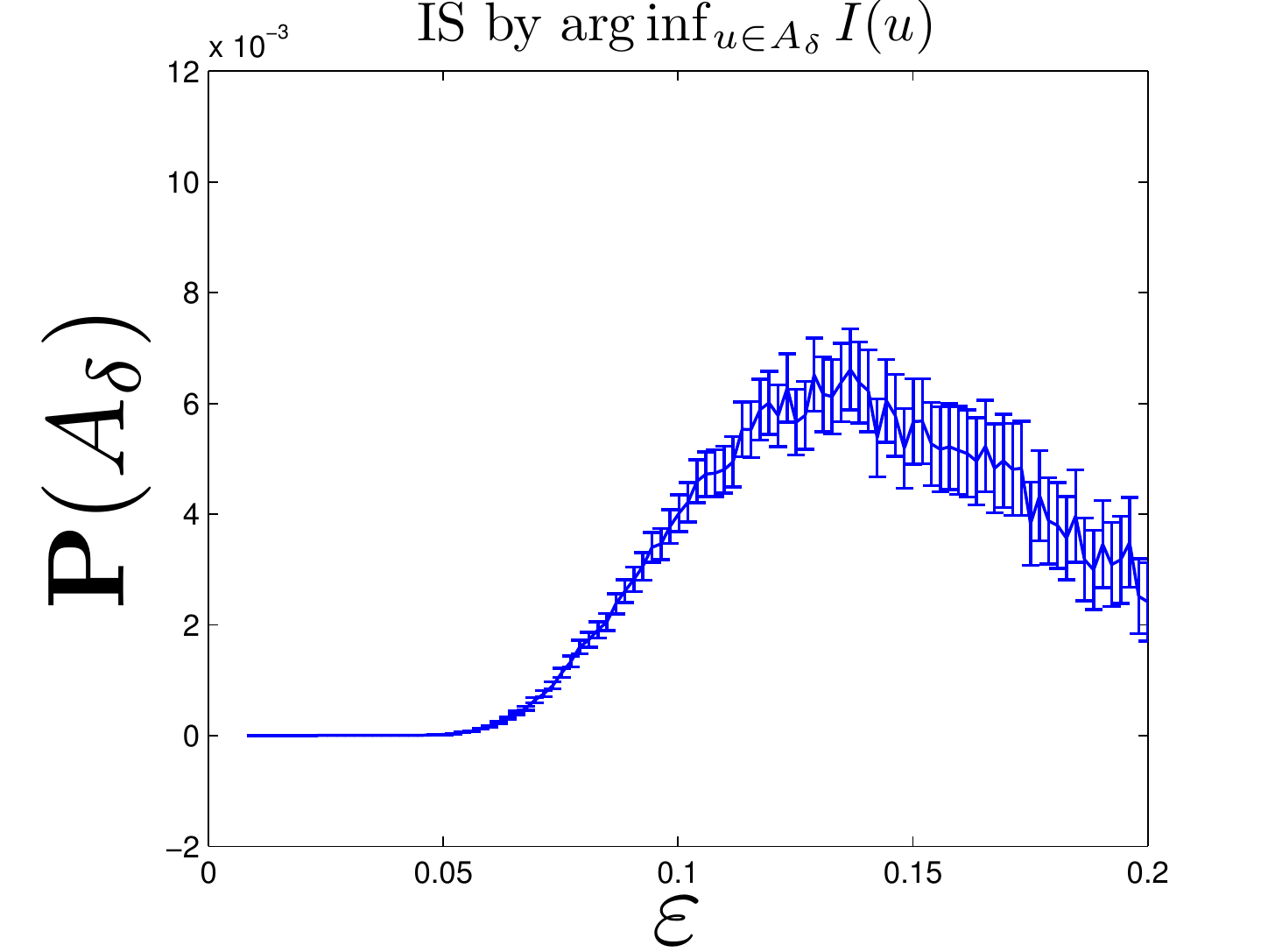}
	\caption{
	\label{fig:confidence interval of the estimators}
	The estimated probabilities and their $99\%$ confidence intervals by $\hat{P}^{MC}$, $\hat{P}^{IS}_0$ 
	and $\hat{P}^{IS}_\delta$. $\hat{P}^{IS}_\delta$ has the best performance and $\hat{P}^{IS}_0$ also has 
	the good performance for $0.1\leq\eps\leq 0.2$. However $\hat{P}^{IS}_0$ works poorly for 
	$\epsilon<0.1$.}
\end{figure}

\begin{figure}
	\centering
	\includegraphics[width=0.32\textwidth]{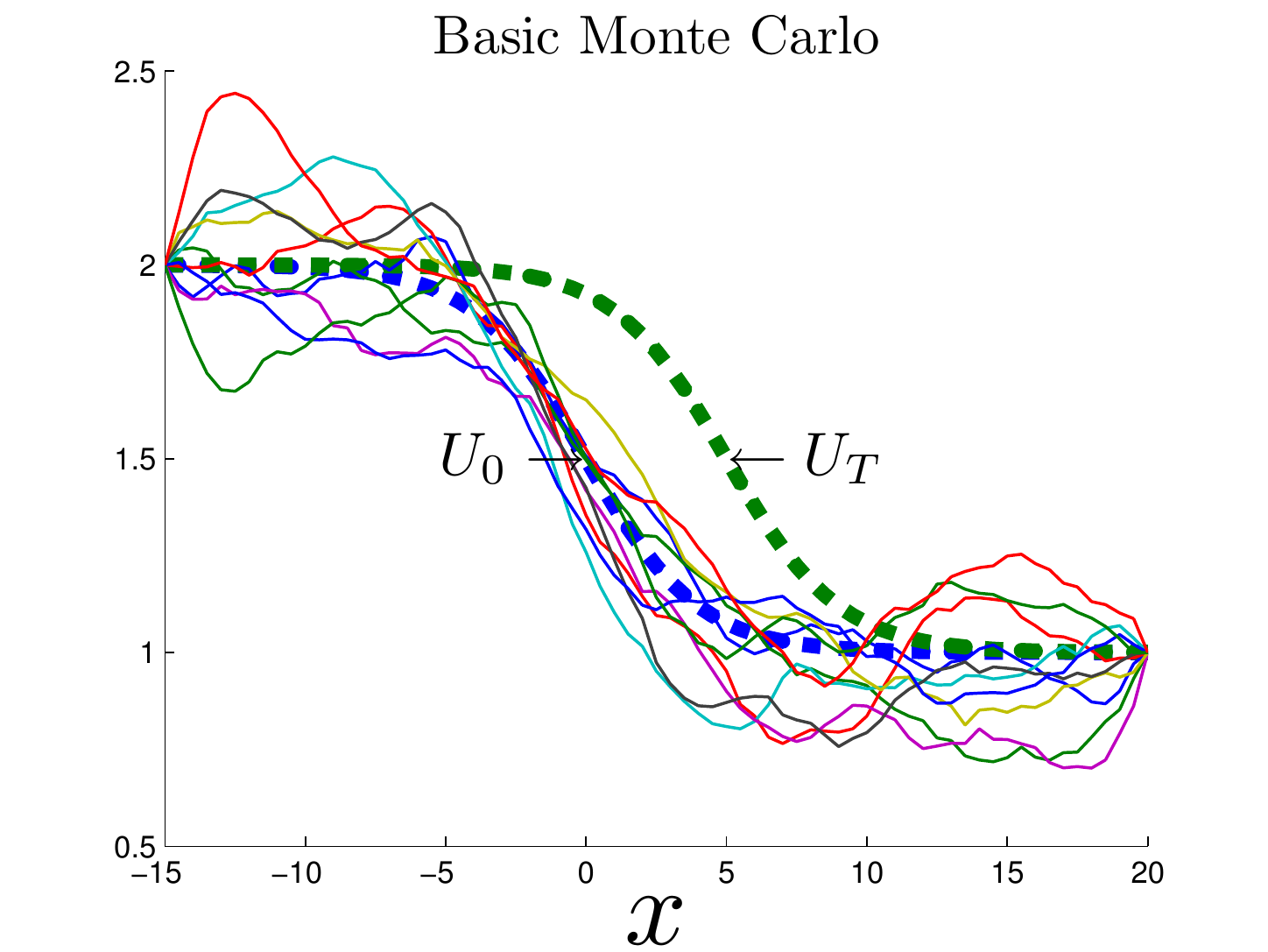}
	\includegraphics[width=0.32\textwidth]{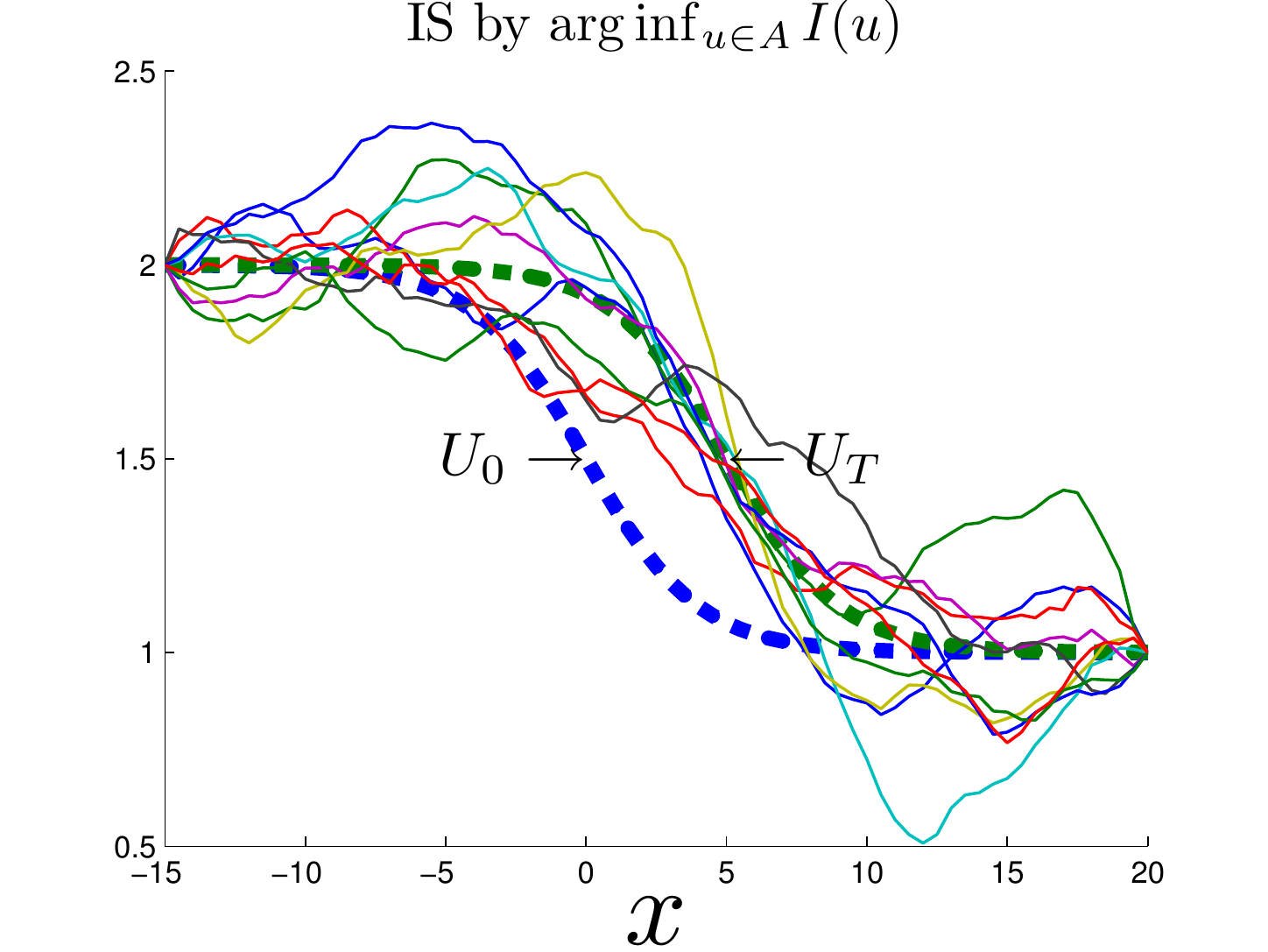}
	\includegraphics[width=0.32\textwidth]{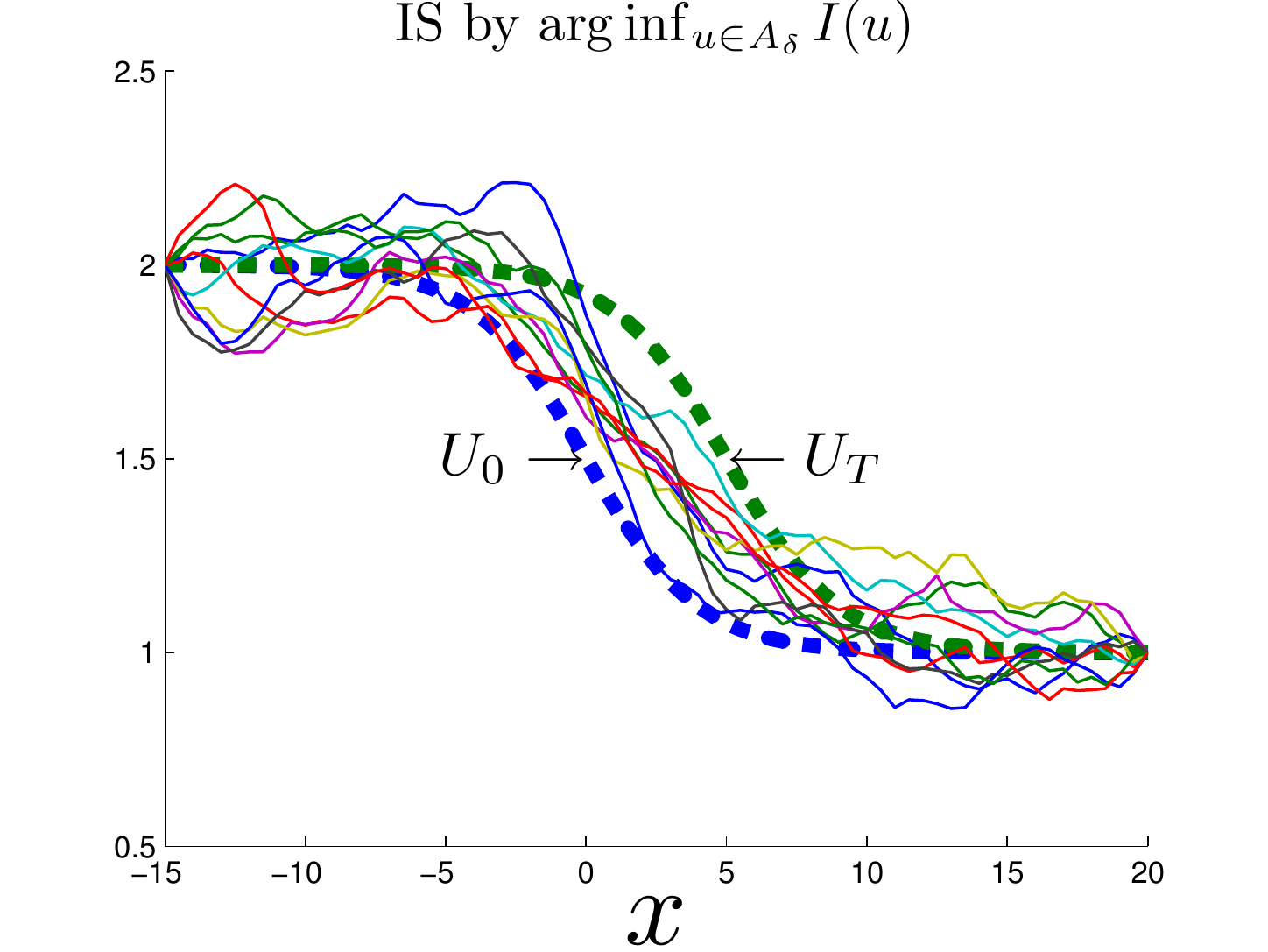}
	\caption{
	\label{fig:sample path of different measures}
	Sample paths under these three probability measures with $\eps=0.1$.}
\end{figure}

We plot the estimated probabilities and the relative error, the ratio of the (numerical) standard 
deviation to the (numerical) mean, in the log scale. Note that in the extreme case, the estimated 
probability is dominated by the value $p^{(k_0)}$ of one realization over $K$ ($p^{(k_0)}=1$ for 
$\hat{P}^{MC}$ and $p^{(k_0)}=d\PP/d\QQ(\Delta \tilde{W}^{(k_0)})$ for $\hat{P}^{IS}$), and the numerical 
variance is approximately $(p^{(k_0)})^2/K$. Therefore the relative error is about $\sqrt{K}=100$. This is 
why the curves of the relative errors in Fig.\ref{fig:P(A_delta) and the relative error in the log scale} 
saturate at $100$, and it also tells us that when the relative error reaches $\sqrt{K}$, the estimator is in 
the extreme case and so is not reliable.

\begin{figure}
	\centering
	\includegraphics[width=0.49\textwidth]{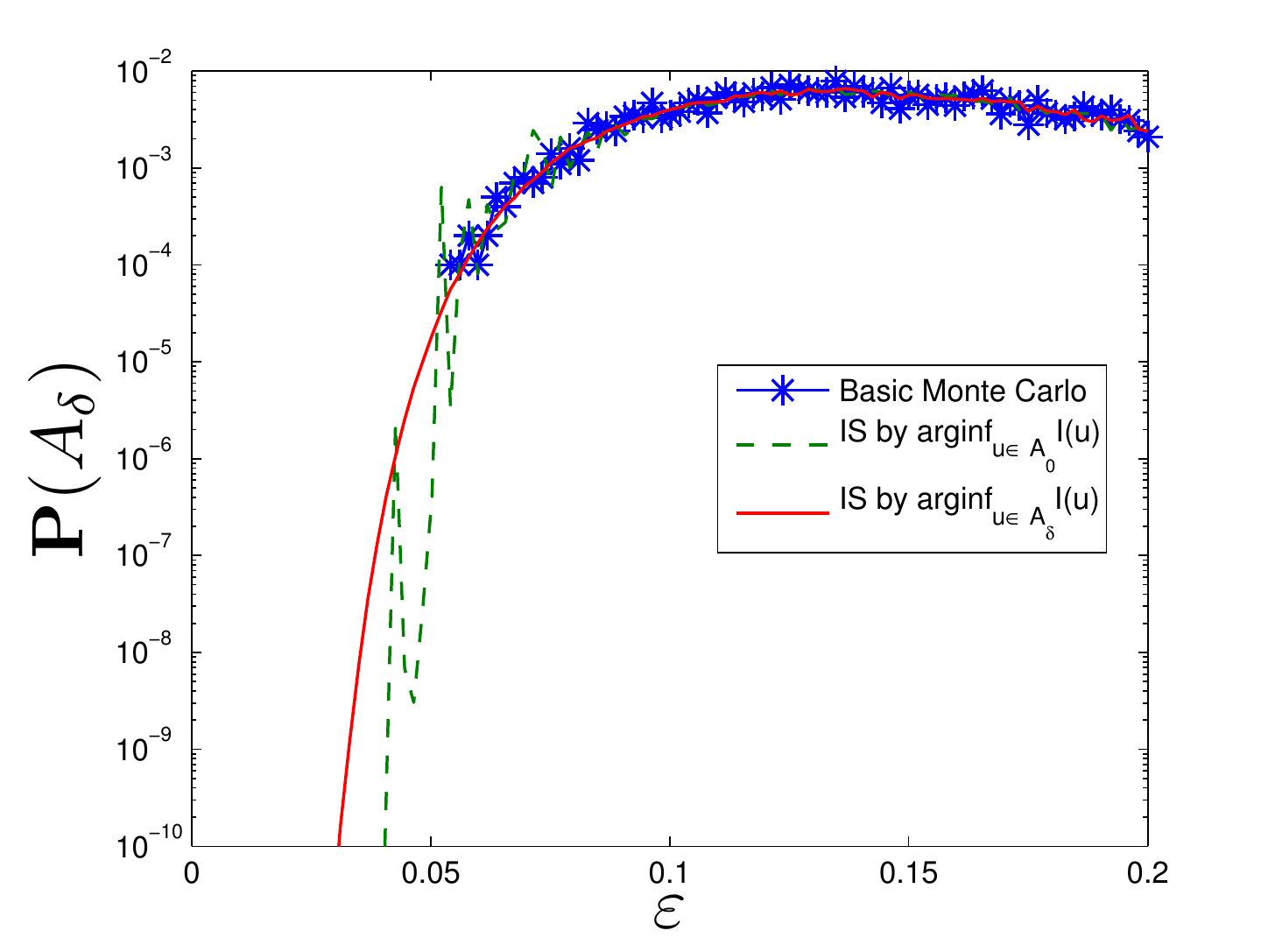}
	\includegraphics[width=0.49\textwidth]{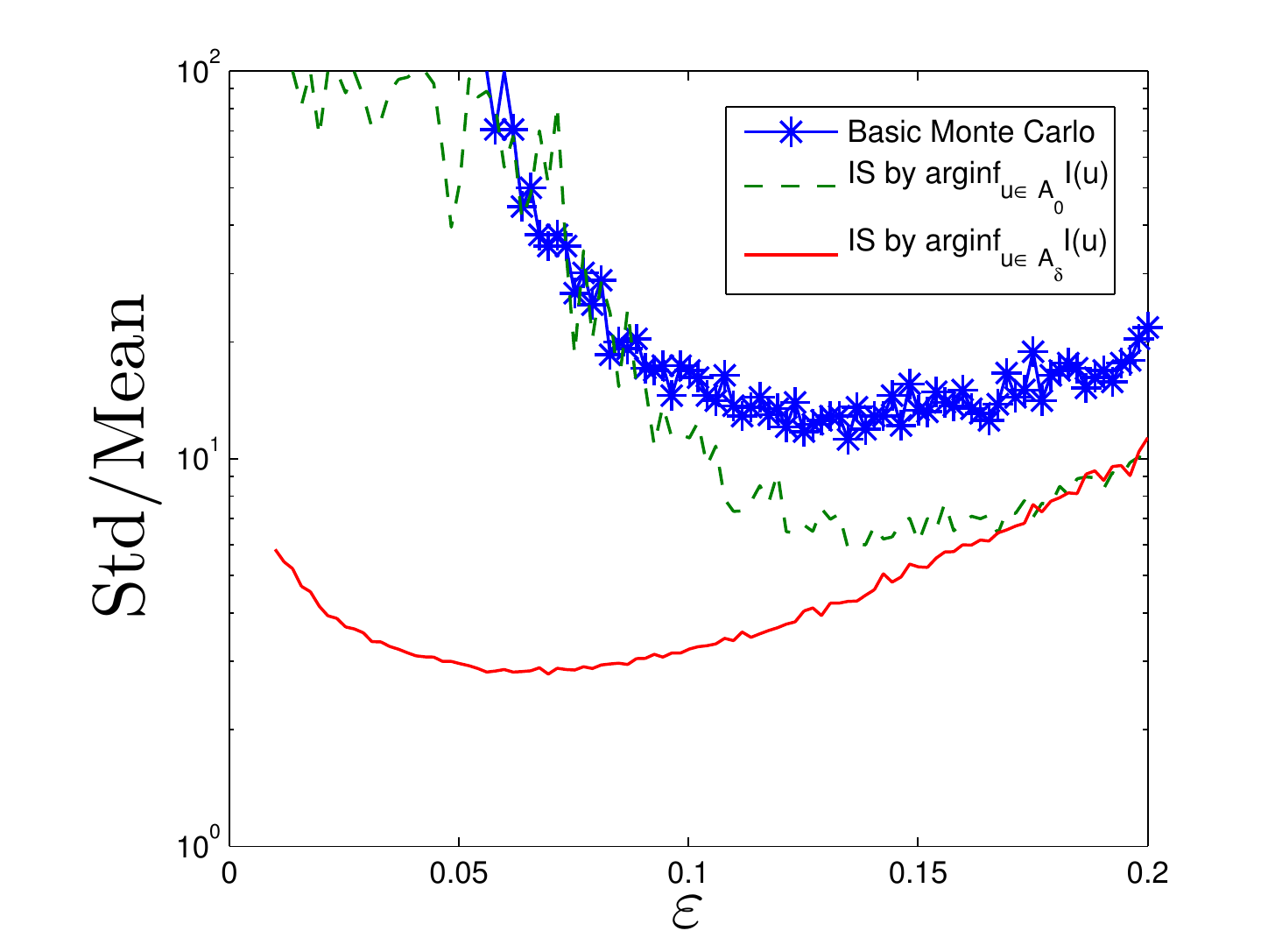}
	\caption{
	\label{fig:P(A_delta) and the relative error in the log scale}
	The semilog plots of the estimated probabilities (left) and the relative errors (right).}
\end{figure}
\section{Conclusion and open problems}

We have analyzed here the small probabilities of anomalous shock profile displacements
due to random perturbations using the theory of large deviations. We have obtained analytically upper
and lower bounds for the exponential rate of decay of these probabilities and we have verified the 
accuracy of these bounds with numerical simulations. We have also used Monte Carlo simulations with 
importance sampling based on the analytically known rate function, which is efficient and gives very good 
results for rare event probabilities.

%
%
%

\section*{Acknowledgment}

This work is partly supported by the Department of Energy [National Nuclear Security Administration] under 
Award Number NA28614, and partly by AFOSR grant FA9550-11-1-0266.

\appendix
\section{Proof of Lemma \ref{lem:HS}}	
\label{app:prooflemHS}%
Taking a Fourier transform in $x$ we have
$$
	\hat{Z}(t,\xi) = \int_0^t e^{-D \xi^2 (t-s)} d\hat{W}(s,\xi)  ,
$$
where $\hat{W}(t,\xi) = \int W(t,x) e^{-i\xi x} dx$ is a complex Gaussian process with mean zero and 
covariance
$$
	\EE \big[ \hat{W}(t,\xi) \overline{ \hat{W}(t',\xi')} \big] = t \wedge t' \, \hat{C}(\xi,\xi')   ,
$$
with $\hat{C}(\xi,\xi') = \int e^{-i\xi x +i\xi'x'} C(x,x') dx dx'$. We find that
$$
	\EE [ | \hat{Z}(t,\xi) |^2 ]  = \frac{1-e^{-2 D \xi^2 t}}{2D \xi^2}  \hat{C}(\xi,\xi)  .
$$
On the one hand, the fact that $\Phi$ is Hilbert-Schmidt from $L^2$ into $L^2$ implies that
$$
	\int \hat{C}(\xi,\xi) d\xi = 2\pi \|\Phi( \cdot, \cdot) \|_{L^2(\RR\times \RR)}^2 
$$
is finite. On the other hand we have $(1-e^{-s})/s \leq 1$ uniformly with respect to $s \in (0,\infty)$, so 
we get that for any $t \in [0,T]$
$$
	\int (1+\xi^2 ) \EE [ | \hat{Z}(t,\xi) |^2 ]d\xi \leq 2\pi  \big(T +\frac{1}{2D} \big) 
	\|\Phi( \cdot, \cdot) \|_{L^2(\RR\times \RR)}^2  , 
$$
which is equivalent by Parseval relation to
$$
	\EE [ \|{Z}(t) \|_{H^1}^2 ] \leq  \big(T +\frac{1}{2D} \big) 
	\|\Phi( \cdot, \cdot) \|_{L^2(\RR\times \RR)}^2  , 
$$
which gives that $Z(t) \in L^2( [0,T], H^1(\RR))$ almost surely. Similarly we find
\begin{eqnarray*}
	\EE \big[ |\hat{Z}(t,\xi) -\hat{Z}(t',\xi) |^2 \big]  
	&=& \big[ 1 - e^{-D \xi^2|t-t'|}\big] \big[ 2 - e^{- 2 D \xi^2 t \wedge t'} +  e^{-D \xi^2(t+t')}\big]  
	\frac{ \hat{C}(\xi,\xi) }{2D \xi^2} \\
	&\leq &  \big[ 1 - e^{-D \xi^2|t-t'|}\big] \frac{ \hat{C}(\xi,\xi) }{D \xi^2}
	\leq |t-t'| \hat{C}(\xi,\xi) ,
\end{eqnarray*}
which gives 
\begin{align*}
	\lefteqn{\EE \big[ \|Z(t) -Z(t') \|_{L^2}^4 \big] 
	= \iint \EE\big[ |Z(t,x)-Z(t',x) |^2  |Z(t,x')-Z(t',x') |^2 \big]  dx dx'}   \\
	&\leq  \iint \EE\big[ |Z(t,x)-Z(t',x) |^4\big]^{1/2} \EE\big[  |Z(t,x')-Z(t',x') |^4 \big]^{1/2} dx dx'\\
	&= 3 \iint  \EE\big[ |Z(t,x)-Z(t',x) |^2\big]  \EE\big[  |Z(t,x')-Z(t',x') |^2 \big]  dx dx'\\
	&=3 \Big( \int \EE\big[ |Z(t,x)-Z(t',x) |^2\big]    dx \Big)^2 
	\leq 3 (2\pi)^2  |t-t'|^2 \|\Phi( \cdot, \cdot) \|_{L^2(\RR\times \RR)}^4  ,
\end{align*}
where we have used the fact that ${Z}(t,x)-Z(t',x)$ is a Gaussian random variable and Parseval equality.
By Kolmogorov's continuity criterium for Hilbert valued stochastic processes \cite[Theorem 3.3]{DaPrato1992}
we find that that $Z(t)\in {\cal C}([0,T], L^2(\RR))$ almost surely.

If $\Phi$ is  Hilbert-Schmidt from $L^2$ into $H^1$, then
$$
	\int \xi^2 \hat{C}(\xi,\xi) d\xi = 2\pi \| \partial_x \Phi( \cdot, \cdot) \|_{L^2(\RR\times \RR)}^2 
$$
is finite, and we can repeat the same arguments to show the final result.
\section{Proofs in Section \ref{sec:LDP}}

\subsection{Proof of Propsition \ref{prop:space of u^eps}}
\label{pf:space of u^eps}

From Lemma \ref{lem:HS}, $Z(t)=\int_0^t S(t-s)dW(s)$ is in $\mathcal{C}([0,T],H^1(\RR))$ 
almost surely. We want to prove the existence and uniqueness in ${\cal E}^1$ of the solution to the equation 
\begin{equation}
	\label{eq:mild solution of one realization}
	u(t) = S(t)U - \int_0^t S(t-s) [F(u(s))]_x ds + \eps Z(t),
\end{equation}
where $F$ is a $\mathcal{C}^2$ function with $\max\{\|F'\|_{L^\infty}, \|F''\|_{L^\infty}\}\leq C_F <\infty$.

To prove the existence we use the Picard iteration scheme: we define $u^0(t)\equiv U$ and 
\begin{equation*}
	u^{n+1}(t) = S(t)U  - \int_0^t S(t-s) [F(u^n(s))]_x ds + \eps Z(t).
\end{equation*}
Therefore,
\begin{align}
	\label{eq:mild solution of ueps-U}
	u^{n+1}(t) - U &= S(t)U-U  - \int_0^t S(t-s) [F'(u^n(s))(u_x^n(s)-U_x)] ds\\
	&\quad - \int_0^t S(t-s) [F'(u^n(s))U_x] ds + \eps Z(t).\notag 
\end{align}
It is easy to see that $u^n(t) - U  \in \mathcal{C}([0,T],H^1(\RR))$ for all $n$, and we first show that 
$\sup_{t\in[0,T]}\|u^n(t)-U \|_{H^1(\RR)}$ are uniformly bounded in $n$.

\begin{lemma}
	\label{lma:bound for a_n}
	Let $a_n(t)=\|u^n(t)-U\|_{H^1}$. Then there exists a constant $C_u$ such that $a_n(t)\leq C_u$ for 
	all $n$ and $t\in[0,T]$. As a consequence, we can choose sufficiently large $C_u$ such that 
	$\|u^n_x(t)\|_{L^2}\leq C_u$ for all $n$ and $t\in[0,T]$.
\end{lemma}
\begin{proof}
	Note that $\sup_{t\in[0,T]}\|S(t)U-U\|_{H^1}$ and $\sup_{t\in[0,T]}\|\eps Z(t)\|_{H^1}$ are 
	bounded. In addition, by \cite[Chap. 15, Sec. 1]{Taylor2011}, 
	$\|S(t-s)\|_{\mathcal{L}(L^2,H^1)}=\mathcal{O}((t-s)^{-1/2})$, then 
	\begin{align*}
		\left\|\int_0^t S(t-s) [F'(u^n(s))U_x] ds\right\|_{H^1}
		&\leq \int_0^t \|S(t-s)\|_{\mathcal{L}(L^2,H^1)} \|F'(u^n(s))U_x\|_{L^2} ds\\
		&\leq C_F \|U_x\|_{L^2} \int_0^t \|S(t-s)\|_{\mathcal{L}(L^2,H^1)}ds,
	\end{align*}
	are uniformly bounded in $n$, and
	\begin{align*}
		&\left\|\int_0^t S(t-s) [F'(u^n(s))(u_x^n(s)-U_x)] ds\right\|_{H^1}\\
		&\quad\leq\int_0^t\|S(t-s)\|_{\mathcal{L}(L^2,H^1)} \|F'(u^n(s))(u_x^n(s)-U_x)\|_{L^2}ds\\
		&\quad \leq C_F \int_0^t \|S(t-s)\|_{\mathcal{L}(L^2,H^1)} \|u_x^n(s)-U_x\|_{L^2} ds\\
		&\quad \leq C_F \int_0^t \|S(t-s)\|_{\mathcal{L}(L^2,H^1)} \|u^n(s)-U\|_{H^1} ds\\
		&\quad \leq C_F \left[\int_0^t \|S(t-s)\|_{\mathcal{L}(L^2,H^1)}^p ds\right]^{1/p}
		\left[\int_0^t\|u^n(s)-U\|_{H^1}^q ds\right]^{1/q}.
	\end{align*}
	By letting $1/p+1/q=1$ with $1<p<2$ and $q>2$, we can have the following recursive inequality with a 
	sufficiently large $C$:
	\[
		a_{n+1}(t) \leq \frac{C}{2} + \frac{C}{2} \left[\int_0^t a_n^q(s) ds\right]^{1/q}.
	\]
	By the convexity of $x\mapsto x^q$,
	\begin{equation}
		\label{eq:recursive inequality for a_n^q}
		a_{n+1}^q(t) \leq \left(\frac{C}{2} + \frac{C}{2} \left[\int_0^t a_n^q(s) ds\right]^{1/q}\right)^q
		\leq \frac{C^q}{2} + \frac{C^q}{2} \int_0^t a_n^q(s) ds.
	\end{equation}
	By noting that $a_0(t)=0$ and (\ref{eq:recursive inequality for a_n^q}), it easy to see that $a_n^q(t)$ 
	are uniformly bounded in $n$ and $t\in[0,T]$ and so are $a_n(t)$.
\end{proof}

To prove the convergence of $u^n(t)-U$ in $\mathcal{C}([0,T],H^1(\RR))$, it suffices to prove that 
$\sum_{n=0}^\infty \sup_{t\in[0,T]}\|u^{n+1}(t)-u^n(t)\|_{H^1} < \infty$.

\begin{lemma}
	\label{lma:existence of the mild solution with saturated fluxes}
	Let $b_n(t)=\|u^n(t)-u^{n-1}(t)\|_{H^1}$. Then $\sum_{n=0}^\infty \sup_{t\in[0,T]}b_n(t) < \infty$. 
	As a consequence, $u^n(t)-U$ converges to $u(t)-U$ in $\mathcal{C}([0,T],H^1(\RR))$ as $n\to\infty$ 
	and $u(t)$ solves (\ref{eq:mild solution of one realization}).
\end{lemma}
\begin{proof}
	By noting (\ref{eq:mild solution of one realization}), Lemma \ref{lma:bound for a_n} and 
	$\|\cdot\|_{L^\infty}\leq C_S\|\cdot\|_{H^1}$ (the Sobolev embedding), we have	
	\begin{align*}
		&\|u^{n+1}(t)-u^n(t)\|_{H^1}\\
		&\leq \int_0^t \|S(t-s)\|_{\mathcal{L}(L^2,H^1)} 
		\|F'(u^n(s))[u_x^n(s)-u_x^{n-1}(s)]\|_{L^2} ds\\
		&\quad+\int_0^t\|S(t-s)\|_{\mathcal{L}(L^2,H^1)}
		\|u_x^{n-1}(s)[F'(u^n(s))-F'(u^{n-1}(s))]\|_{L^2}ds\\
		&\leq C_F \int_0^t \|S(t-s)\|_{\mathcal{L}(L^2,H^1)} \|u^n(s)-u^{n-1}(s)\|_{H^1} ds\\
		&\quad+\int_0^t\|S(t-s)\|_{\mathcal{L}(L^2,H^1)}
		\|u_x^{n-1}(s)\|_{L^2}\|F'(u^n(s))-F'(u^{n-1}(s))\|_{L^\infty}ds\\
		&\leq C_F \int_0^t \|S(t-s)\|_{\mathcal{L}(L^2,H^1)} \|u^n(s)-u^{n-1}(s)\|_{H^1} ds\\
		&\quad+ C_F C_S C_u	\int_0^t\|S(t-s)\|_{\mathcal{L}(L^2,H^1)}\|u^n(s)-u^{n-1}(s)\|_{H^1}ds\\
		&\leq C_F(1+C_S C_u)\left[\int_0^t\|S(t-s)\|_{\mathcal{L}(L^2,H^1)}^p ds\right]^{1/p}
		\left[\int_0^t \|u^n(s)-u^{n-1}(s)\|_{H^1}^q ds\right]^{1/q},
	\end{align*}
	where $1/p+1/q=1$ with $1<p<2$ and $q>2$. Then there exists a constant $C$ such that 
	\[
		b_{n+1}^q(t) \leq C \int_0^t b_n^q(s) ds.
	\]
	Then $b_n^q(t)\leq b_0^q(t) C^nT^n/n!$ so $b_n(t)\leq b_0(t) (C^nT^n/n!)^{1/q}$ and it is easy to see 
	that $\sum_{n=0}^\infty \sup_{t\in[0,T]}b_n(t) < \infty$.
\end{proof}

Finally we show that (\ref{eq:mild solution of one realization}) has a unique solution in ${\cal E}^1$.
\begin{lemma}
	If $u , v \in {\cal E}^1$ solve 
	(\ref{eq:mild solution of one realization}), then $\sup_{t\in[0,T]}\|u(t)-v(t)\|_{H^1}=0$.
\end{lemma}
\begin{proof}
	Let $u, v \in {\cal E}^1$ solve (\ref{eq:mild solution of one realization}).
	Using the same calculations in the proof of Lemma 
	\ref{lma:existence of the mild solution with saturated fluxes}, we get
	\[
		\|u(t)-v(t)\|_{H^1}
		\leq C_F(1+C_S C_u) \int_0^t \|S(t-s)\|_{\mathcal{L}(L^2,H^1)} \|u(s)-v(s)\|_{H^1}ds.
	\]
	By noting that $\|u(0)-v(0)\|_{H^1}=0$, $\sup_{t\in[0,T]}\|u(t)-v(t)\|_{H^1}=0$ by Gronwall's 
	inequality.
\end{proof}


\subsection{Proof of Proposition \ref{prop:ldp}}
\label{pf:ldp}

The LDP can be obtained following the strategy of \cite{Gautier2005,deBouard2010}. The first step of the 
proof uses the LDP for the laws of the stochastic convolution $\eps Z$ on the space 
$\mathcal{C}([0,T],H^1(\RR))$, where
\[
	Z(t) =\int_0^t S(t-s) dW(s).
\]
The laws of $\eps Z$ are Gaussian measures and the LDP with a good rate function is a consequence of the 
general result on LDP for centered Gaussian measures on real Banach spaces \cite{Deuschel1989,Gautier2005}. 

The second step is to prove that the mapping $X(t) \mapsto u(t)-U$ is continuous from
$\mathcal{C}([0,T],H^1(\RR))$ into itself, where 
\begin{equation}
	\label{eq:mild solution driven by general force}
	u(t) = S(t)U - \int_0^t S(t-s) (F(u(s)))_x ds + \int_0^t S(t-s) X(s)ds.
\end{equation}
Then the LDP for $u^\eps$ is obtained by the contraction principle \cite{Dembo2010}.

To prove the continuity of the mapping, given two pairs $(X,u)$ and $(Y,v)$ satisfying 
(\ref{eq:mild solution driven by general force}), we show that $v-u\to 0$ as $Y\to X$ in 
$\mathcal{C}([0,T],H^1(\RR))$. By using essentially the same calculations in Lemma 
\ref{lma:existence of the mild solution with saturated fluxes},
\begin{align*}
	\|v(t)-u(t)\|_{H^1} 
	&\leq C_F(1+C_S C_u) \int_0^t \|S(t-s)\|_{\mathcal{L}(L^2,H^1)} \|v(s)-u(s)\|_{H^1} ds \\
	&\quad + \int_0^t \|S(t-s)\|_{\mathcal{L}(H^1,H^1)} \|Y(s)-X(s)\|_{H^1} ds,
\end{align*}
where $C_u=\sup_{t\in[0,T]}\|u_x(t)\|_{L^2}$. Noting that $\|S(t-s)\|_{\mathcal{L}(H^1,H^1)}=\mathcal{O}(1)$ 
and $\|S(t-s)\|_{\mathcal{L}(L^2,H^1)}=\mathcal{O}((t-s)^{-1/2})$, and by Gronwall's inequality, there 
exists a constant $C$ such that for $t\in[0,T]$
\begin{equation}
	\label{eq:bound for contiunity}
	\|v(t)-u(t)\|_{H^1} \leq C e^{C\int_0^T (T-s)^{-1/2}ds}\int_0^T \|Y(s)-X(s)\|_{H^1} ds.
\end{equation}
Because $\int_0^T (T-s)^{-1/2}ds<\infty$, $v-u\to 0$ as $Y\to X$ in $\mathcal{C}([0,T],H^1(\RR))$.


\subsection{Proof of Proposition \ref{prop:maximum principle}}
\label{pf:maximum principle}

For each $n$, we let $h^n\in L^2([0,T],L^2(\RR))$ such that $u=\mathcal{H}[h^n]$ and 
\[
	\frac{1}{2}\int_0^T\|h^n(t,\cdot)\|_{L^2}^2 dt-\frac{1}{n}
	< I(u)\leq \frac{1}{2}\int_0^T\|h^n(t,\cdot)\|_{L^2}^2 dt.
\]
Because $u\neq U(x-\gamma t)$, $I(u)>0$. We let $u^n$ be the mild solution of 
\begin{align*}
	&u^n_t + (F(u^n))_x = (Du^n_x)_x + (1-(nI(u))^{-1})^{1/2}\Phi h^n,\\
	&u^n(0,x)=U(x).
\end{align*}
Then $u^n-U\in\mathcal{C}([0,T],H^1(\RR))$ and 
\[
	I(u^n)\leq\frac{1}{2}\int_0^T(1-(nI(u))^{-1})\|h^n(t,\cdot)\|_{L^2}^2 dt
	\leq \frac{1}{2}\int_0^T\|h^n(t,\cdot)\|_{L^2}^2 dt - \frac{1}{n}
	<I(u).
\]

We show that $u^n - u\to 0$ in $\mathcal{C}([0,T],H^1(\RR))$. Let $\alpha^n=(1-(nI(u))^{-1})^{1/2}$ and 
replace $(X,u)$ and $(Y,v)$ by $(u,\Phi h^n)$ and $(u^n,\alpha^n\Phi h^n)$ respectively in 
(\ref{eq:bound for contiunity}). We have 
\begin{align*}
	\|u^n(t)-u(t)\|_{H^1} 
	&\leq C e^{C\int_0^T (T-s)^{-1/2}ds}\int_0^T (1-\alpha^n)
	\|\Phi\|_{\mathcal{L}(L^2,H^1)}\|h^n(s)\|_{L^2} ds\\
	&\leq 2C e^{C\int_0^T (T-s)^{-1/2}ds} (1-\alpha^n)
	\|\Phi\|_{\mathcal{L}(L^2,H^1)}\left(I(u)+\frac{1}{n}\right).
\end{align*}
Then $u^n - u\to 0$ in $\mathcal{C}([0,T],H^1(\RR))$ as $\alpha^n\to 1$.

\section{Proof of Proposition \ref{prop:1}}
\label{app:proofprop1}%

We use the following technical result:
If $f \in {\cal C} ([0,T], H^1(\RR))$, then $F(t) = \int_0^t S(t-s) \partial_x f (s)ds $
is such that $\int F(t,x) dx= 0$ for any $t \in[0,T]$.
Indeed, since $f(s)\in H^1$ and $H^1 \subset L^\infty$, we have by integrating by parts:
$$
F(t,x) =  -\int_0^t \frac{1}{\sqrt{2 \pi(t-s)}} \int \partial_y \big( e^{-\frac{(x-y)^2}{2(t-s)}} \big) f(s,y) dy ds  ,
$$
so that, for any $a<b$
$$
\int_a^b F(t,x) dx 
= \int_0^t \frac{1}{\sqrt{2 \pi(t-s)}} \int \big( e^{-\frac{(b-y)^2}{2(t-s)}} - e^{-\frac{(a-y)^2}{2(t-s)}} \big) f(s,y) dy ds   .
$$
Let $\delta >0$. There exists $M$ such that $\int_{[-M,M]^c} f(s,y)^2 dy < \delta^2$ for any $s \in [0,T]$.
 We can split the integral
in $y$ into two pieces. The integral over $[-M,M]^c$ can be bounded by Cauchy-Schwarz inequality so we get
$$
\big| \int_a^b F(t,x) dx \big|
=  \frac{\sqrt{2}}{\sqrt[4]{\pi}} \delta  
+ \int_0^t \frac{1}{\sqrt{2 \pi(t-s)}}
 \int_{-M}^M \big( e^{-\frac{(b-y)^2}{2(t-s)}} + e^{-\frac{(a-y)^2}{2(t-s)}} \big) f(s,y) dy ds  ,
$$
and therefore, by Lebesgue's theorem
$$
\limsup_{a \to -\infty , b\to +\infty}\big| \int_a^b F(t,x) dx \big| \leq \frac{\sqrt{2}}{\sqrt[4]{\pi}}  \delta 
.
$$
Since $\delta$ is arbitrary we get the desired  technical result.

We then find by integrating (\ref{eq:mild solution of ueps-U}) in $x$ that the center
of $u^\eps$ is  given by (\ref{center1}).
With the condition that $C$  is in $L^1(\RR\times \RR)$,  the last term is proportional to a Brownian motion.
\section{Proofs in Section \ref{sec:displacement_general}}

\subsection{Proof of Lemma \ref{lma:sifficient condition for Hilbert-Schmidt Phi}}
\label{pf:sifficient condition for Hilbert-Schmidt Phi}
The values of $\sigma$, $L_0$ and $l_c$ do not affect the result, so in the proof we 
set $\sigma=L_0=l_c=1$ and $\Phi^{l_c}_{L_0}=\Phi$ without loss of generality. $\Phi(x,x')$ is 
Hilbert-Schmidt from $L^2$ to $H^k$ if and only if $\partial_x^j\Phi(x,x')\in L^2(\RR\times\RR)$ for 
$0\leq j\leq k$. Taking the two-dimensional Fourier transform on 
$\partial_x^j\Phi(x,x')$, we have 
\begin{align*}
	\mathcal{F}_{\xi,\xi'}\{\partial_x^j\Phi(x,x')\}
	&= \mathcal{F}_\xi\{\mathcal{F}_{\xi'}\{\partial_x^j\Phi(x,x')\}\}
	 = \mathcal{F}_\xi\{\partial_x^j[\phi_0(x)\mathcal{F}_{\xi'}\{\phi_1(x-x')\}]\}\\
	&= \mathcal{F}_\xi\{\partial_x^j[\phi_0(x)e^{ix\xi'}\hat{\phi}_1(-\xi')\}]\}
	 = (i\xi)^j\hat{\phi}_1(-\xi')\mathcal{F}_\xi\{\phi_0(x)e^{ix\xi'}\}\\
	&= (i\xi)^j\hat{\phi}_1(-\xi')\hat{\phi}_0(\xi+\xi').
\end{align*}
By a simple calculation it is easy to see that 
$(i\xi)^j\hat{\phi}_1(-\xi')\hat{\phi}_0(\xi+\xi') \in L^2(\RR\times\RR)$ if 
$\phi_0$ and $\phi_1$ are both in $H^j(\RR)$.

\subsection{Proof of Lemma \ref{lma:lower bound of J(A_delta)}}
\label{pf:lower bound of J(A_delta)}
It suffices to show the case that $\delta_n=1/n$. For each $n$, let $u_n\in A_{1/n}$, such that 
$\mathcal{J}(A_{1/n})\leq I(u^n)<\mathcal{J}(A_{1/n})+1/n$; $\{I(u^n)\}$ are bounded from above by 
$\mathcal{J}(A_1)+1<\infty$. Because $I$ is a good rate function and compactness is equivalent to 
sequentially compactness in  $\mathcal{E}^1$, $\{u^n\}$ has a convergent subsequence $\{u^{n_l}\}$ whose 
limit $u^*$ is in $A$. As $I$ is lower semicontinuous, then
\[
	\mathcal{J}(A) \geq \lim_n\mathcal{J}(A_{1/n}) 
	= \lim_l I(u^{n_l}) = \liminf_l I(u^{n_l}) 
	\geq I(u^*) \geq \mathcal{J}(A).
\]

\subsection{Proof of Lemma \ref{lma:approximation of phi_0 for small |x_0|}}
\label{pf:approximation of phi_0 for small |x_0|}

Because $\phi_0(x)\equiv 1$ on $x\in(-1,1)$ by Assumption \ref{asmp:conditions of Phi},
\begin{align*}
	&\|[\phi_0^{-1}(\cdot/L_0)-1]U_x(\cdot-(t/T)(\gamma T+x_0))\|_{L^2}^2\\
	&\quad =\int_{-\infty}^{-L_0}+\int_{L_0}^{\infty}
	\{[\phi_0^{-1}(x/L_0)-1]U_x(x-(t/T)(\gamma T+x_0))\}^2 dx\\
	&\quad =\int_{-\infty}^{-L_0-\frac{t}{T}(\gamma T+x_0)}+\int_{L_0-\frac{t}{T}(\gamma T+x_0)}^{\infty}
	\{[\phi_0^{-1}(x/L_0+\frac{t}{T}(\gamma T+x_0)/L_0)-1]U_x(x)\}^2 dx.
\end{align*}
Let $L_0=\gamma T+x_0+C_v$ with $C_v\geq 1$. Because $\phi_0^{-1}(x)\geq 1$ is monotonically increasing 
for $x\geq 0$ and $C_v \geq 1$,
\[
	[\phi_0^{-1}(x/L_0+\frac{t}{T}(\gamma T+x_0)/L_0)-1]^2 \leq [\phi_0^{-1}(x+1)-1]^2,\quad x \geq 0.
\]
Similarly, because $\phi_0^{-1}(x)\geq 1$ is monotonically decreasing for $x\leq 1$,
\[
	[\phi_0^{-1}(x/L_0+\frac{t}{T}(\gamma T+x_0)/L_0)-1]^2 \leq [\phi_0^{-1}(x)-1]^2,\quad x\leq 0.
\] 
In addition, $C_v\leq L_0-\frac{t}{T}(\gamma T+x_0)$ and $-C_v\geq-L_0-\frac{t}{T}(\gamma T+x_0)$. 
Therefore,
\begin{align*}
	&\|[\phi_0^{-1}(\cdot/L_0)-1]U_x(\cdot-(t/T)(\gamma T+x_0))\|_{L^2}^2\\
	&\quad \leq \int_{-\infty}^{-C_v} \{[\phi_0^{-1}(x)-1]U_x(x)\}^2 dx
	+ \int_{C_v}^{\infty} \{[\phi_0^{-1}(x+1)-1]U_x(x)\}^2 dx.
\end{align*}
By noting that $\phi_0^{-1}(x)$ has at most polynomial growth, there exists a uniform $C_v\geq 1$ for 
(\ref{eq:L^2 approximation of phi_0 for small |x_0|}).

\subsection{Proof of Lemma \ref{lma:approximation of phi_1 for small |x_0|}}
\label{pf:approximation of phi_1 for small |x_0|}
It is easy to check from the properties of the traveling wave $U$ that 
$v_t+F(v)_x-Dv_{xx}\in\mathcal{C}([0,T],\mathcal{S}(\RR))$. By (\ref{eq:covariance kernel Phi}) and 
(\ref{eq:driving force expression of v}), we have
\[
	v_t+F(v)_x-Dv_{xx} = \sigma \phi_0\Big(\frac{x}{L_0}\Big)
	\left[\frac{1}{l_c}\phi_1\Big(\frac{x}{l_c}\Big)*h^{l_c}_{L_0}(t,x) \right].
\]
Taking the Fourier transform on the both sides,
\[
	\hat{h}^{l_c}_{L_0}(t,\xi) = \sigma^{-1} (\hat{\phi}_1(l_c\xi))^{-1}
	\mathcal{F}_\xi\{(v_t+F(v)_x-Dv_{xx})/\phi_0(\cdot/L_0)\}.
\]
Because $v_t+F(v)_x-Dv_{xx}\in\mathcal{S}(\RR)$ and $1/\phi_0(\cdot/L_0)\in\mathcal{C}^\infty$ has at most 
polynomial growth, $\mathcal{F}_\xi\{(v_t+F(v)_x-Dv_{xx})/\phi_0(\cdot/L_0)\}$ is well-defined and also
in $\mathcal{C}([0,T],\mathcal{S}(\RR))$. In addition, $1/|\hat{\phi}_1(l_c\xi)|$ also has at most 
polynomial growth and then indeed $h^{l_c}_{L_0}(t,\xi)\in L^2(\RR)$.

Because $\mathcal{F}_\xi\{(v_t+F(v)_x-Dv_{xx})/\phi_0(\cdot/L_0)\}$ is in 
$\mathcal{C}([0,T],\mathcal{S}(\RR))$ and $1/|\hat{\phi}_1(\xi)|$ has at most polynomial growth,
\begin{align*}
	\|h^{l_c}_{L_0}(t,\cdot)\|_{L^2}^2 
	&\to \frac{1}{2\pi\sigma^2} \int |\hat{\phi}_1(0)|^{-2} 
	|\mathcal{F}_\xi\{(v_t+F(v)_x-Dv_{xx})/\phi_0(\cdot/L_0)\}|^2 d\xi\\
	&=\frac{1}{\sigma^2} \|(v_t+F(v)_x-Dv_{xx})/\phi_0(\cdot/L_0)\|_{L^2}^2.
\end{align*}
as $l_c\to 0$, uniformly in $t\in[0,T]$. Then we have 
(\ref{eq:L^2 approximation of phi_1 for small |x_0|}).

\subsection{Proof of Proposition \ref{prop:general form of lower bounds}}
\label{pf:general form of lower bounds}
For any pair $(u,h)$ satisfying (\ref{eq:control1a}),
\begin{align*}
	\|h(t,\cdot)\|_{L^2}^2
	&=\sup_{f,f\neq 0}\frac{\langle h(t,\cdot),f\rangle^2}{\langle f,f\rangle}
	\geq\frac{\langle h(t,\cdot),(\Phi^{l_c}_{L_0})^T 1\rangle^2}
	{\langle (\Phi^{l_c}_{L_0})^T 1,(\Phi^{l_c}_{L_0})^T 1\rangle}\\
	&= \|(\Phi^{l_c}_{L_0})^T 1\|_{L^2}^{-2}\langle \Phi^{l_c}_{L_0} h(t,\cdot),1\rangle^2
	= \|(\Phi^{l_c}_{L_0})^T 1\|_{L^2}^{-2} \langle u_t+F(u)_x-Du_{xx},1\rangle^2.
\end{align*}
Because $\Phi^{l_c}_{L_0} h \in {\cal C}([0,T],H^1(\RR))$, we have $u -U\in {\cal C}([0,T],H^1(\RR))$. Therefore 
$\lim_{x\to\pm\infty}u(t,x)=u_{\pm}$ and $\lim_{x\to\pm\infty}u_x(t,x)=0$ for all $t \in [0,T]$. Thus,
\begin{align*}
	I(u) &\geq \frac{1}{2}\|(\Phi^{l_c}_{L_0})^T 1\|_{L^2}^{-2} 
	\int_0^T \langle u_t+F(u)_x-Du_{xx},1\rangle^2 dt\\
	&= \frac{1}{2}\|(\Phi^{l_c}_{L_0})^T 1\|_{L^2}^{-2} 
	\int_0^T\left(\int u_t(t,x) dx+F(u_+)-F(u_-)\right)^2 dt.
\end{align*}
By the Rankine-Hugoniot condition (\ref{eq:Rankine-Hugoniot})
\begin{align*}
	I(u) &\geq \frac{1}{2}\|(\Phi^{l_c}_{L_0})^T 1\|_{L^2}^{-2} 
	\int_0^T\left(\int u_t (t,x)dx+\gamma(u_+ - u_-)\right)^2 dt\\
	&= \frac{1}{2}\|(\Phi^{l_c}_{L_0})^T 1\|_{L^2}^{-2}
	\int_0^T\left(\int[u_t(t,x)-\frac{d}{dt}U(x-\gamma t)]dx\right)^2 dt\\
	&= \frac{1}{2}\|(\Phi^{l_c}_{L_0})^T 1\|_{L^2}^{-2}
	\int_0^T\left(\frac{d}{dt}\int[u(t,x)-U(x-\gamma t)]dx\right)^2 dt.
\end{align*}
By the Schwarz inequality and noting that $u\in A$,
\begin{align*}
	I(u) &\geq \frac{1}{2}T^{-1}\|(\Phi^{l_c}_{L_0})^T 1\|_{L^2}^{-2}
	\left(\int_0^T\frac{d}{dt}\int[u(t,x)-U(x-\gamma t)]dx dt\right)^2\\
	&= \frac{1}{2}T^{-1}\|(\Phi^{l_c}_{L_0})^T 1\|_{L^2}^{-2} 
	\left(\int[u(T,x)-U(x-\gamma T)]dx\right)^2\\
	&= \frac{1}{2}T^{-1}\|(\Phi^{l_c}_{L_0})^T 1\|_{L^2}^{-2} 
	\left(\int[U(x-\gamma T-x_0)-U(x-\gamma T)]dx\right)^2\\
	&= \frac{1}{2}T^{-1}\|(\Phi^{l_c}_{L_0})^T 1\|_{L^2}^{-2}\left(\int[U(x-x_0)-U(x)]dx\right)^2.
\end{align*}

\subsection{Proof of Lemma \ref{lma:Phi^T 1}}
\label{pf:Phi^T 1}
We first compute $(\Phi^{l_c}_{L_0})^T$. For any test functions $f$ and $g$,
\begin{align*}
	\langle \Phi^{l_c}_{L_0} f,g \rangle 
	&= \int\sigma\phi_0(\frac{x}{L_0}) \int\frac{1}{l_c}\phi_1(\frac{x-x'}{l_c})f(x')dx'g(x)dx\\
	&= \int f(x') \int\sigma\phi_0(\frac{x}{L_0})\frac{1}{l_c}\phi_1(\frac{x-x'}{l_c})g(x)dxdx'
	= \langle f,(\Phi^{l_c}_{L_0})^T g \rangle.
\end{align*}
Thus $(\Phi^{l_c}_{L_0})^T g(x) = \sigma[\phi_0(\frac{x}{L_0})g(x)]*\frac{1}{l_c}\phi_1(-\frac{x}{l_c})$ and 
$(\Phi^{l_c}_{L_0})^T 1(x) = \sigma\phi_0(\frac{x}{L_0})*\frac{1}{l_c}\phi_1(-\frac{x}{l_c})$. Then 
\begin{align*}
	\|(\Phi^{l_c}_{L_0})^T 1\|_{L^2}^2 
	&= \frac{1}{2\pi}\sigma^2 L_0^2 \int \hat{\phi}_0^2(L_0\xi) \hat{\phi}_1^2(-l_c\xi)d\xi \\
	& \overset{l_c\to 0}{\to} \frac{1}{2\pi}\sigma^2 L_0^2 \int \hat{\phi}_0^2(L_0\xi) \hat{\phi}_1^2(0)d\xi
	= \sigma^2 L_0 \|\phi_0\|_{L^2}^2.
\end{align*}

\subsection{Proof of Lemma \ref{lma:approximation of phi_0 for large |x_0|}}
\label{pf:approximation of phi_0 for large |x_0|}
The proof for $F(w)_x -Dw_{xx}$ is similar to the proof of Lemma 
\ref{lma:approximation of phi_0 for small |x_0|} so we skip it. For $U(x-\gamma T-x_0)-U(x)$, because 
$\phi_0(x)\equiv 1$ on $x\in(-1,1)$ by Assumption \ref{asmp:conditions of Phi},
\begin{align*}
	\lefteqn{\|[\phi_0^{-1}(\cdot/L_0)-1][U(\cdot-\gamma T-x_0)-U]\|_{L^2}^2}\\
	&=\int_{-\infty}^{-L_0}+\int_{L_0}^\infty\{[\phi_0^{-1}(x/L_0)-1][U(x-\gamma T-x_0)-U(x)]\}^2 dx\\
	&\leq\int_{-\infty}^{-L_0}\{[\phi_0^{-1}(x/L_0)-1][u_- - U(x)]\}^2 dx\\
	&\quad + \int_{L_0}^\infty\{[\phi_0^{-1}(x/L_0)-1][U(x-\gamma T-x_0)-u_+]\}^2 dx\\
	&= \int_{-\infty}^{-L_0}\{[\phi_0^{-1}(x/L_0)-1][u_- - U(x)]\}^2 dx\\
	&\quad + \int_{L_0-\gamma T-x_0}^\infty\{[\phi_0^{-1}(x/L_0+(\gamma T+x_0)/L_0)-1][U(x)-u_+]\}^2 dx\\
	&\leq \int_{-\infty}^{-C_w}\{[\phi_0^{-1}(x)-1][u_- - U(x)]\}^2 dx
	+ \int_{C_w}^\infty\{[\phi_0^{-1}(x+1)-1][U(x)-u_+]\}^2 dx.
\end{align*}
The last inequality holds because $\phi_0^{-1}(x)\geq 1$ is increasing for $x\geq 0$, decreasing for 
$x<0$ and $L_0\geq 1$. Then we can find a uniform $C_w\geq 1$.

\bibliographystyle{siam}
\bibliography{reference}

\begin{thebibliography}{10}

\bibitem{Cardon-Weber1999}
{\sc C.~Cardon-Weber}, {\em {L}arge deviations for a {B}urgers'-type {SPDE}},
  Stochastic Process. Appl., 84 (1999), pp.~53--70.

\bibitem{Cardon-Weber2001}
{\sc C.~Cardon-Weber and A.~Millet}, {\em {A} {S}upport {T}heorem for a
  {G}eneralized {B}urgers {SPDE}}, Potential Analysis, 15 (2001), pp.~361--408.

\bibitem{Chen1993}
{\sc J.-C. Chen, D.~Lu, J.S. Sadowsky, and K.~Yao}, {\em {O}n importance
  sampling in digital communications. {I}. {F}undamentals}, Selected Areas in
  Communications, IEEE Journal on, 11 (1993), pp.~289 --299.

\bibitem{DaPrato1992}
{\sc G.~Da~Prato and J.~Zabczyk}, {\em {S}tochastic equations in infinite
  dimensions}, vol.~44 of Encyclopedia of Mathematics and its Applications,
  Cambridge University Press, Cambridge, 1992.

\bibitem{Dautray1992}
{\sc R.~Dautray and J.-L. Lions}, {\em {M}athematical analysis and numerical
  methods for science and technology. {V}ol. 5}, Springer-Verlag, Berlin, 1992.

\bibitem{deBouard2010}
{\sc A.~de~Bouard and E.~Gautier}, {\em {E}xit problems related to the
  persistence of solitons for the {K}orteweg-de {V}ries equation with small
  noise}, Discrete Contin. Dyn. Syst., 26 (2010), pp.~857--871.

\bibitem{Dembo2010}
{\sc A.~Dembo and O.~Zeitouni}, {\em {L}arge deviations techniques and
  applications}, vol.~38 of Stochastic Modelling and Applied Probability,
  Springer-Verlag, Berlin, 2010.

\bibitem{Deuschel1989}
{\sc J.-D. Deuschel and D.~W. Stroock}, {\em {L}arge deviations}, vol.~137 of
  Pure and Applied Mathematics, Academic Press Inc., Boston, MA, 1989.

\bibitem{E2004}
{\sc W.~E, W.~Ren, and E.~Vanden-Eijnden}, {\em {M}inimum action method for the
  study of rare events}, Comm. Pure Appl. Math., 57 (2004), pp.~637--656.

\bibitem{Freidlin1998}
{\sc M.~I. Freidlin and A.~D. Wentzell}, {\em {R}andom perturbations of
  dynamical systems}, vol.~260 of Grundlehren der Mathematischen Wissenschaften
  [Fundamental Principles of Mathematical Sciences], Springer-Verlag, New York,
  second~ed., 1998.

\bibitem{Gautier2005}
{\sc {\'E}.~Gautier}, {\em {L}arge deviations and support results for nonlinear
  {S}chr\"odinger equations with additive noise and applications}, ESAIM
  Probab. Stat., 9 (2005), pp.~74--97 (electronic).

\bibitem{Glasserman1997}
{\sc P.~Glasserman and Y.~Wang}, {\em {C}ounterexamples in importance sampling
  for large deviations probabilities}, Ann. Appl. Probab., 7 (1997),
  pp.~731--746.

\bibitem{Iaccarino2011}
{\sc G.~Iaccarino, R.~Pecnik, J.~Glimm, and D.~Sharp}, {\em {A} {QMU} approach
  for characterizing the operability limits of air-breathing hypersonic
  vehicles}, Reliability Engineering \& System Safety, 96 (2011), pp.~1150 --
  1160.

\bibitem{Ilin1960}
{\sc A.~M. Il$'$in and O.~A. Ole{\u\i}nik}, {\em {A}symptotic behavior of
  solutions of the {C}auchy problem for some quasi-linear equations for large
  values of the time}, Mat. Sb. (N.S.), 51 (93) (1960), pp.~191--216.

\bibitem{Jones1993}
{\sc C.~K. R.~T. Jones, R.~Gardner, and T.~Kapitula}, {\em {S}tability of
  travelling waves for non-convex scalar viscous conservation laws},
  Communications on Pure and Applied Mathematics, 46 (1993), pp.~505--526.

\bibitem{Lax1957}
{\sc P.~D. Lax}, {\em {H}yperbolic systems of conservation laws. {II}}, Comm.
  Pure Appl. Math., 10 (1957), pp.~537--566.

\bibitem{LeVeque2002}
{\sc R.~J. LeVeque}, {\em {F}inite volume methods for hyperbolic problems},
  Cambridge Texts in Applied Mathematics, Cambridge University Press,
  Cambridge, 2002.

\bibitem{Mariani2010}
{\sc M.~Mariani}, {\em {L}arge deviations principles for stochastic scalar
  conservation laws}, Probab. Theory Related Fields, 147 (2010), pp.~607--648.

\bibitem{Nocedal2006}
{\sc J.~Nocedal and S.~J. Wright}, {\em {N}umerical optimization}, Springer
  Series in Operations Research and Financial Engineering, Springer, New York,
  second~ed., 2006.

\bibitem{Sadowsky1996}
{\sc J.~S. Sadowsky}, {\em {O}n {M}onte {C}arlo estimation of large deviations
  probabilities}, Ann. Appl. Probab., 6 (1996), pp.~399--422.

\bibitem{Sadowsky1990}
{\sc J.~S. Sadowsky and J.~A. Bucklew}, {\em {O}n large deviations theory and
  asymptotically efficient {M}onte {C}arlo estimation}, IEEE Trans. Inform.
  Theory, 36 (1990), pp.~579--588.

\bibitem{Siegmund1976}
{\sc D.~Siegmund}, {\em {I}mportance sampling in the {M}onte {C}arlo study of
  sequential tests}, Ann. Statist., 4 (1976), pp.~673--684.

\bibitem{Taylor2011}
{\sc M.~E. Taylor}, {\em {P}artial differential equations {III}. {N}onlinear
  equations}, vol.~117 of Applied Mathematical Sciences, Springer, New York,
  second~ed., 2011.

\bibitem{Varadhan1966}
{\sc S.~R.~S. Varadhan}, {\em {A}symptotic probabilities and differential
  equations}, Comm. Pure Appl. Math., 19 (1966), pp.~261--286.

\bibitem{West2011}
{\sc N.~West, G.~Papanicolaou, P.~Glynn, and G.~Iaccarino}, {\em {A}
  {N}umerical {S}tudy of {F}iltering and {C}ontrol for {S}cramjet {E}ngine
  {F}low}, in 20th AIAA Computational Fluid Dynamics Conference, vol.~4, 2011,
  pp.~3010--3028.

\bibitem{Zhou2008}
{\sc X.~Zhou, W.~Ren, and W.~E}, {\em {A}daptive minimum action method for the
  study of rare events}, The Journal of Chemical Physics, 128 (2008),
  p.~104111.

\end{thebibliography}
\end{document}